\documentclass[11pt,letterpaper]{amsart}
\usepackage{amsmath, amsthm, amsfonts,tikz, mathrsfs, amscd}
\author{Katharine Turner}
\title{Representing Vineyard Modules}

\usepackage{algorithm}
\usepackage{algpseudocode}
\usepackage{nicematrix,booktabs}
\usetikzlibrary{arrows.meta,
                chains,
                positioning,
                shapes.geometric
                }
\usepackage{makecell}
\usepackage{sidecap}
\usetikzlibrary{calc}

\usepackage{graphicx,fullpage,tikz,tikz-cd,amsmath,amssymb,comment,mathtools}
\usetikzlibrary{arrows.meta,positioning,decorations.pathreplacing}

\tikzset{
  treenode/.style = {shape=rectangle, rounded corners,
                     draw, align=center,},
  root/.style     = {treenode, font=\small}, 
  env/.style      = {treenode, font=\ttfamily\normalsize},
  dummy/.style    = {circle,draw}
}

\tikzstyle{level 1}=[level distance=2.8cm, sibling distance=10cm]
\tikzstyle{level 2}=[level distance=2.8cm, sibling distance=5cm]

\date{}

\theoremstyle{plain}
\newtheorem*{theorem*}{Theorem}
\newtheorem{theorem}{Theorem}[section]
\newtheorem{lemma}[theorem]{Lemma}

\newtheorem{corollary}[theorem]{Corollary}

\newtheorem{remark}[theorem]{Remark}
\newtheorem{assumption}[theorem]{Assumption}
\newtheorem {proposition}[theorem] {Proposition}

\theoremstyle{definition}
\newtheorem{definition}[theorem]{Definition}
\newtheorem{defn}[theorem]{Definition}

\newcommand{\R}{\mathbb{R}}

\DeclareMathOperator{\supp}{supp}
\newcommand{\Crit}{\operatorname{Crit}}

\newcommand{\PD}{\mathcal{PD}}

\newcommand{\Simp}{\operatorname{Simp}}

\newcommand{\birth}{\mathbf{b}} 
\newcommand{\death}{\mathbf{d}} 

\DeclareMathOperator{\im}{im}

\newcommand{\Mod}[1]{\mathcal{#1}} 

\newcommand{\Vine}[1]{\mathscr{#1}} 

\newcommand{\Mat}{\text{Mat}}
\newcommand{\Mor}{\text{Mor}}
\newcommand{\MatC}{\mathcal{MAT}}
\newcommand{\BTM}{\mathcal{BTM}}
\newcommand{\ObMat}{\mathcal{OBSMAT}}
\newcommand{\MorC}{\mathcal{MOR}}
\newcommand{\Vineyard}[1]{\mathbb{#1}} 

\usetikzlibrary{calc}
\tikzset{
	l/.style={gray,dashed},
	ll/.style={redlayer,dashed},
	grn/.style=green!70!black,
	b/.style=bluelayer,
	r/.style=redlayer,
	p/.style=dark gray}
\usepackage{caption}

\usepackage{xcolor}

\usetikzlibrary{plotmarks}

\newcommand{\Id}{\mathbb{I}}

\DeclareMathOperator{\id}{id}

\usepackage{url}
\usepackage{enumerate}

\definecolor{redlayer}{RGB}{200, 80, 80}
\definecolor{bluelayer}{RGB}{80, 120, 200}

\begin{document}

\begin{abstract}
Time-series of persistence diagrams, more commonly known as vineyards, are a useful way to capture how multi-scale topological features vary over time. However, we do lose significant information in how the persistent homology at different time stamps relate to each other. A natural algebraic version of vineyards is a time series of persistence modules equipped with interleaving maps between the persistence modules at different time stamps. We call this a vineyard module. 
In this paper we will set up the framework for representing a vineyard module via an indexed set of vines alongside a collection of matrices. We outline an algorithmic way to transform the bases of the persistence modules at each time step within the vineyard module to make the matrices within this representation as simple as possible. With some reasonable assumptions on the vineyard modules, this simplified representation can be completely described (up to isomorphism) by the associated vineyard and a vector of finite length. While this vineyard and vector representation is not necessarily unique we can prove that the vector will always be the zero vector when the vineyard module is isomorphic to the direct sum of vine modules. We also show uniqueness when the vineyard module contains two vines (over field $\mathbb{F}_2$) and can characterise and count which are realisable, giving a complete account of isomorphism classes of vineyard modules with two vines. To define and analyse this approach we first must set up a lot of preliminary results about changes of bases for pairs of persistence modules where we are given $\varepsilon$-interleaving maps for sufficiently close $\varepsilon$.  This new perspective on vineyards also provides a new approach to the study of multi-parameter persistence modules.
\end{abstract}
\maketitle

\section{Introduction}

Vineyards (\cite{munch2013applications, morozov2008homological, cohen2006vines}) are established within the topological data analysis literature as a way to study time varying data with applications including music classification \cite{bergomi2020homological}, detecting dynamical regime change \cite{dee2021detecting}, EEG dynamics \cite{yoo2016topological}, and studying fMRI data \cite{salch2021mathematics}. Historically vineyards are defined as a continuous map from a finite real interval to the space of persistence diagrams. Informally, there is an intuitive decomposition of vineyards into paths of points within the persistence diagrams which are called vines. However, to formally and rigorously decompose vineyards we need to view them as algebraic objects. In turn this requires informative ways to represent this algebraic information.

To view vineyards as an algebraic object we first need to consider them as a continuous map from an interval to the space of persistence modules instead of persistence diagrams. There is now an important choice - do we merely require the existence of appropriate interleaving maps (which means we have no more information than the continuous map into the space of persistence diagrams) or do we incorporate the interleaving maps as part of the algebraic object? To distinguish these situations we use the term \textbf{vineyard} for a continuous map from a closed interval $[T_0, T_1]$ to the space of persistence diagrams, and \textbf{vineyard module} to denote the algebraic object consisting of both the parameterised family of the original persistence modules alongside interleaving maps which are required to commute with the transition maps within each of the persistence modules. There is a dramatic difference in the types of indecomposable vineyards under these two different paradigms. Vineyard modules contain strictly more information than vineyards; different vineyard modules can become isomorphic as vineyards (when we forget the interleaving maps). In this paper we consider the representation of vineyard modules, up to isomorphism, which have a common associated vineyard.

Vineyards decompose naturally into paths of points within the plane. If no persistence diagram has any points with higher multiplicity, then this decomposition is guaranteed to be unique by continuity. We will follow the existing literature and call these paths \textbf{vines}. We can define vine modules as vineyard modules whose corresponding vineyard is a vine. The definition of an indecomposable vineyard module stems naturally from the definitions of morphisms and direct sums of vineyard modules. Morphisms between vineyard modules are a family of morphisms, one for each time stamp, which commute appropriately with the interleaving maps between the persistence modules at different time stamps in each of the domain and codomain. Direct sums can be constructed by taking direct sums for the persistence modules at each time stamp and constructing the interleaving maps as direct sums of interleaving maps for each summand. In section \ref{sec:two vines} we provide an example of a vineyard module which has two vines but which is provably not isomorphic to the direct sum of two vine modules, and hence is indecomposable. This example is constructed as the persistent homology of sublevel set filtrations, over a common domain, for a continuous family of functions showing that this indecomposable phenomenon is realisable through persistent homology.

A complete characterisation of indecomposable vineyard modules is beyond the scope of this paper. Here we focus on the prerequisite step of creating a framework to represent vineyard modules, and on the restricted settings of trivial vineyard modules, and vineyard modules with two vines. We define a vines and matrices representation of a vineyard module which consists of an indexed set of vines and a collection of matrices. We show that two vineyard modules with the same vines and matrices representation are isomorphic. If all the matrices in this vines and matrices representation satisfy a common block diagonal form then we can automatically split the vineyard module as a direct sum of vineyard modules constructed over each of the blocks. 

There are many different vines and matrices representations for the same vineyard module as it is very dependent on the choices of bases made for the persistence modules at each of the time stamps. To counteract this ambiguity we outline an algorithmic way to change the bases of the persistence modules at each time step within the vineyard module in a systematic fashion which ensures that locally the matrices corresponding to the interleaving maps are as simple as possible. This simplified representation of the vineyard module is determined (up to isomorphism) by the associated vineyard and a vector of finite length. 

The paper then explores the question of uniqueness of this vineyard and vector representation for trivial vineyard modules. For a forwards and then backwards simplified vineyard module the corresponding vector will always be zero when the vineyard module is isomorphic to the direct sum of vine modules. As such it provides an algorithmic method of determining when a vineyard module decomposes into vine modules (what we call trivial). It also provides a canonical representation when the associated vineyard has two vines and we fix the underlying field as $\mathbb{F}_2$.

Vineyard modules allow us to approach $2$-parameter persistence from a new perspective; one where your head is tilted $45$ degrees. Vineyard modules contain identical information to their corresponding $2$-parameter persistence modules - the change is mainly one of notation. For each $t\in \R$ we consider the restriction of the $2$-parameter persistence module to the line $y=x+t$ which is just a $1$-parameter persistence module. The horizontal and vertical transition maps become the interleaving maps.

Given the algebraic equivalence  of vineyard modules to 2-parameter persistence modules there is an overlap of research to the study of 2-parameter persistence modules. Between the initial arXiv version of this paper and now, a related process of simplification to determine whether persistence modules are indecomposable was independently developed by Kerber and Russold in \cite{kerber2024graphcode, kerber2025representing}. There are three main ways their approach differs to that in this paper. Firstly, instead of what would be the slanted lines in the rotated version of vineyard modules, in \cite{kerber2024graphcode, kerber2025representing} they use the horizontal lines in the 2-parameter persistence module. This means that instead of interleaving maps they work with morphisms. This means there is no analogy to the backwards simplification in this paper. Secondly, they record the relationships between the different persistence modules via edges in the graph (instead of the matrices in this paper). Thirdly they are using a discrete parameter space instead of continuous (both in height and time), but with potentially more complicated morphism between these discrete time steps. In common is the appreciation of the dependence on the choice of basis and an exploration of when a change of basis can reduce the off-diagonal entries/edges. Notably, Kerber and Russold have implemented code for this computation of the graphcodes and decision algorithm for whether the module completely decomposes, with input a presentation of the 2-parameter persistence module.

Other related work includes the characterisation of the space of bases for a persistence module and the isomorphism classes and matrix representations of ladders (\cite{jacquard2023space, escolar2016persistence, asashiba2019matrix}), and in particular \cite{urbanvcivc2024ladder} which focuses on ladders which are part of a pair of interleaving maps. Other related work includes algorithms for updating persistence diagrams along vineyards such as in \cite{cohen2006vines, dey2021updating}. We hope that adaptation of these algorithms are potential avenues for efficient computation of representations of vineyard modules.

The structure of the paper is as follows. We first cover in Section \ref{sec:poset} the background theory of persistence modules over a poset. With this perspective we can view $\epsilon$-morphisms  and $\epsilon$-interleavings as persistent modules over appropriate posets. We then introduce in Section \ref{sec:vineyard} the vineyard module as a persistence module over a poset and relate it to the (geometric) vineyard which forgets the algebra of the non-vertical arrows in the poset. In order to represent our vineyard modules we need to be able to express the $\epsilon$-morphisms contained within them in terms of matrices. These matrices are naturally dependent on choices of bases. The foundational theory is covered in Section \ref{sec:matrices}. We then can use this to define the vines and matrices representation of a vineyard module in Section \ref{sec:vineyard and vector}, where we also introduce the notion of forwards (and backwards) simplified families of bases and how they allow us to represent a vineyard module by its associated vineyard and a finite length vector. This motivates a process to forwards simplify vineyard modules. In Section \ref{sec:segments} we start with segments - vineyard modules over time intervals short enough to bound the behaviour of critical times and heights. We then, in Section \ref{sec:combining segments} sequentially process segments to forwards (or backwards) simplifying families of bases showing that this is independent of the choice of segments if the input is already forwards or backwards simplified. The representations we compute do determine the isomorphism class of the vineyard module but, in general, there can be multiple representations for the same vineyard module. Thankfully we can show in Section \ref{sec:trivial} that when the vineyard module is trivial (decomposes into vine modules) then we can get uniqueness. Section \ref{sec:two vines} then explores the nicely behaved setting of vineyard modules with two vines. We finish with a long selection of future directions. To increase readability we moved all proofs of lemmas to the appendices. 

\subsection*{AI statement}
For this paper, AI was used for proofreading and to generate tikz code for the figures.

\section{Persistence modules over a poset}\label{sec:poset}

Persistence modules were initially considered only for a parameter space of the real line or a nice (often discrete) subset of the real line. However, this notion has since been generalised to base spaces which are any poset.

 \begin{definition}
A \textbf{partially ordered set}, also called a \textbf{poset}, $(\Theta,\leq)$ is a set $\Theta$ with a relation $\leq$ which is
\begin{itemize}
\item Reflexive: that is $\alpha \leq \alpha$ for all $\alpha \in \Theta$,
\item Antisymmetric: that is $\alpha \leq \beta$ and $\beta \leq \alpha$ implies $\alpha=\beta$, and
\item Transitive: that is $\alpha\leq \beta$ and $\beta\leq \gamma$ implies $\alpha\leq \gamma$.
\end{itemize}
If the poset is also comparable (for all $\alpha,\beta$ either $\alpha\leq \beta$ or $\beta\leq \alpha$) then it is called \textbf{totally ordered}.
\end{definition}

\begin{definition}
Fix a field $\mathbb{F}$ and let $\Theta$ be a poset. A \textbf{persistence module} $\mathcal{P}$ over $\Theta$ is a family $\{V_\alpha\}_{\alpha\in \Theta}$ of $\mathbb{F}$-vector spaces indexed by elements of $\Theta$, together with a family of homomorphisms $\{\varphi_\alpha^\beta:V_\alpha\to V_\beta\}$ such that  $\varphi^\gamma_\alpha= \varphi^\gamma_\beta \circ \varphi^\beta_\alpha $ for all $\alpha \leq \beta \leq \gamma$, and $\varphi_\alpha^\alpha= \id_{V_\alpha}$. We call the $\varphi_\alpha^\beta$   \textbf{transition maps}. We call $\Theta$ the parameter space of $\mathcal{P}$. We say $\mathcal{P}$ is \textbf{pointwise finite dimensional} if the $V_\alpha$ are finite dimensional for all $\alpha\in \Theta$. 
\end{definition}

\begin{definition}
Given two persistence modules $\Mod{X}$ and $\Mod{Y}$ over the same parameter space $\Theta$, a \textbf{morphism} between them, $\alpha: \Mod{X} \to \Mod{Y}$, is a parameterised family of linear maps $\{\alpha_a : X_a \to Y_{a}\mid a\in \Theta\}$ which commute with the transition maps of both $\Mod{X}$ and $\Mod{Y}$. An isomorphism between persistence modules is a morphism which has an inverse which is also a morphism.
\end{definition}

While we have already defined a morphism $\alpha: \Mod{X} \to \Mod{Y}$ between persistence modules (over $\R$) via a parameterised family of linear maps $\{\alpha_a : X_a \to Y_{a}\mid a\in \R\}$, we could have equivalently defined it through the construction of a new persistence module over a larger poset. Create a poset $\Theta$ which contains two disjoint copies of $\R$; $\{(a,X)\mid a\in \R\} \cup \{(a,Y)\mid a\in \R\}$. We then define a partial order on $\Theta$ by setting 
\begin{itemize}
\item $(a,X)\leq (b,X)$ whenever $a\leq b$,
\item $(a,Y)\leq (b,Y)$ whenever $a\leq b$, and
\item $(a,X)\leq (b,Y)$ whenever $a\leq b$.
\end{itemize}
We can then construct a persistence module over $(\Theta, \leq)$ from $\alpha$ via the transition maps on $\Mod{X}$, $\Mod{Y}$ and the linear maps $\alpha_a$. Everything will commute appropriately precisely because this was required in the definition for the morphism $\alpha$.

In the algebraic theory of persistence modules there are often technical requirements about tameness, and being pointwise finite dimensional is generally a sufficient condition. This is a  reasonable assumption in almost any application.  The most important algebraic result is the decomposition theorem.

\begin{theorem}[\cite{botnan2022introduction} Theorem 4.2] \label{thm:decomp}
For a poset  $P$ and persistence module  $M: P \rightarrow \operatorname{Vect}$,

\begin{enumerate}
\item If $M$ is pointwise finite-dimensional, then there exists a collection of indecomposables  $\{M_\lambda\}_{\lambda \in \Lambda}$  such that
$$M \cong \bigoplus_{\lambda \in \Lambda} M_\lambda.$$

\item If
$$M \cong \bigoplus_{\lambda \in \Lambda} M_\lambda \cong \bigoplus_{\gamma \in \Gamma} M_\gamma$$
with each $M_\lambda$ and $M_\gamma$ indecomposable, then there exists a bijection $\sigma: \Lambda \rightarrow
\Gamma$ such that $M_\lambda \cong M_{\sigma(\lambda)}$  for all $\lambda \in \Lambda.$
\end{enumerate}

\end{theorem}

The default poset parameter space is the real line with the standard $\leq$ relation. When the parameter space is not explicitly stated, any persistence module in this paper will be assumed to be of this form. Within the representation theory of persistence modules, the building blocks are interval modules. 

An important set of persistence modules over poset $(\R, \leq)$ is that of interval modules. The \textbf{interval module} over the real interval $A$ is a persistence module with a one-dimensional vector space attached to each $h\in A$ and $0$ otherwise. The transition maps $\phi_{h_1}^{h_2}$ are the identity when $h_1<h_2$ are both in $A$ and $0$ otherwise. We denote the interval module over the interval $[b,d)$ by $\Mod{I}[b,d)$. We refer to the value $b$ as the \textbf{birth value} of $\Mod{I}[b,d)$ and denote by $\birth(\Mod{I}[b,d))$. We refer to the value $d$ as the \textbf{death value} of $\Mod{I}[b,d)$ and denote by $\death(\Mod{I}[b,d))$. A persistence module over $(\R, \leq)$ is indecomposable if and only if it is an interval module \cite{botnan2022introduction}.

\begin{assumption}\label{assumption:halfclosed}
Throughout this paper we assume that every persistence module (with parameter space $\mathbb{R}$) will be isomorphic to one of the form
 $\bigoplus\limits_{i=1}^{N} \Mod{I}[b_i, d_i),$
 with $b_i \in \R$ finite, $N$ finite and $[b_i,d_i)\neq [b_j,d_j)$ for all $i\neq j$.  We know that the intervals in this decomposition are unique up to a permutation of the order of indices.
 \end{assumption}

In the interests of full disclosure, in full generality there are four types of intervals that can appear in the interval decomposition of a persistence module (parameter space $\R$). These four types are open-open -- $(b,d)$, open-closed-- $(b,d]$, closed-open -- $[b,d)$, and closed-closed -- $[b,d]$. However, for the sake of clarity, we will assume no intervals of these other forms appear. This is a reasonable restriction as it naturally occurs in virtually all applications. 

We can visually represent a persistence module $\oplus_{i}\Mod{I}_i$ as the multiset of intervals that correspond to the supports of the interval modules in its decomposition, that is $\{[\birth(\Mod{I}_i), \death(\Mod{I}_i))\}$. We call this set of intervals a \textbf{persistence barcode}. If we represent each bar in the persistence barcode as a point in $\mathbb{R}^2$ with the first coordinate $\birth(\Mod{I}_i)$ and second coordinate $ \death(\Mod{I}_i))$, then the resulting multi-set obtained is called its \textbf{persistence diagram}. 

We will denote the space of persistence diagrams by $\PD$. Note that by our simplifying assumptions all our persistence diagrams contain only finitely many off-diagonal points and each point has a unique location (and hence the persistence diagrams will be honest sets). The space of persistence diagrams is equipped with many metrics but in this paper we will only consider the bottleneck distance.

The bottleneck distance is a form of optimal transport metric. For $X$ and $Y$ persistence diagrams with off-diagonal points $\{x_i=(a_i, b_i)\}$ and $\{y_j=(c_j, d_j)\}$ respectively, a \textbf{matching} between $X$ and $Y$ is a subset $M\subset X\times Y$ such that each $x_i$ and $y_j$ appears in at most one pair. Let $U(X)\subset X$ be the set of $x_i\in X$ not appearing in any pair in $M$ and $U(Y)\subset Y$ the set of $y_j\in Y$ not appearing in any pair in $M$. The cost associated to $M$ is 
$$\text{cost}(M)=\max\left\{\sup_{((a_i, b_i),(c_j,d_j))\in M} (\max(|a_i-c_j|, |b_i-d_j|)), \sup_{(a_i,b_i)\in U(X)}(|a_i-b_i|/2),  \sup_{(c_j,d_j)\in U(Y)}(|c_j-d_j|/2) \right\}$$
The \textbf{bottleneck distance}, $d_{bot}(\cdot, \cdot)$, is defined as the infimum of the costs over all matchings.

One of the fundamental shifts in perspective required for topology to be useful in applications is the ability to quantify how far from isomorphic two objects are. To do this we loosen the definition of morphism to incorporate $\varepsilon\geq 0$ worth of wiggle room, which we call an $\varepsilon$-morphism. 

\begin{definition}
Let $\Mod{X}=(\{X_a\}, \{\psi_a^b\})$ and  $\Mod{Y}=(\{Y_a\}, \{\phi_a^b\})$ be persistence modules over poset $(\R, \leq)$.
An \textbf{$\varepsilon$-morphism} $\alpha: \Mod{X} \to \Mod{Y}$ is a parameterised family of linear maps $\{\alpha_a : X_a \to Y_{a+\varepsilon}\mid a\in \R\}$ which commute with the transition maps of both $\Mod{X}$ and $\Mod{Y}$. That is, for all $a<b$ we have $\alpha_b \circ  \psi_a^b =\phi_{a+\varepsilon}^ {b+\varepsilon} \circ \alpha_a$. Note that a $0$-morphism is the traditional notion of morphism. 
\end{definition}

Given persistence modules $\Mod{X}$ and $\Mod{Y}$ and a fixed choice of $\varepsilon\geq 0$ the set of all $\varepsilon$-morphisms from $\Mod{X}$ to $\Mod{Y}$ forms a vector space with (pointwise) addition and scalar multiplication. Denote this vector space of $\varepsilon$-morphisms by $\MorC(\Mod{X}, \Mod{Y}, \varepsilon)$.

In this wiggle-room universe the analogous concept for an isomorphism is an $\varepsilon$-interleaving.  

\begin{definition}
An \textbf{$\varepsilon$-interleaving} ($\varepsilon \geq 0$) between $\Mod{X}$ and $\Mod{Y}$ consists of a pair of $\varepsilon$-morphisms $\alpha: \Mod{X} \to \Mod{Y}$ and $\beta:\Mod{Y} \to \Mod{X}$ such that all the maps $\alpha_a$, $\beta_a$ and the transition maps $\{\psi_a^b\}$ in $\Mod{X}$ and $\{\phi_{a}^b\}$ in $\Mod{Y}$ all commute appropriately. That is, for all $a\in \R$ we have $\beta_{a+\varepsilon} \circ \alpha_a=\psi_a^{a+2\varepsilon}$ and $\alpha_{a+\varepsilon} \circ \beta_a=\phi_{a}^{a+2\varepsilon}$.
\end{definition}

We can use interleaving maps to define the \textbf{interleaving distance} between persistence modules;
$$d_{int}(\Mod{X}, \Mod{Y})=\inf \{\varepsilon\geq 0\mid \text{there exists an $\varepsilon$-interleaving between }\Mod{X} \text{ and }\Mod{Y}\}.$$
It is well known that the interleaving distance between two persistence modules is the same as the bottleneck distance between their respective diagrams \cite{lesnick2015theory}. For a vineyard module we have a persistence module at each time value so we will be using interleaving distances (as opposed to bottleneck distances).


Just as we could convert a morphism between persistence modules into a new persistence modules over a larger poset, we can do the same process with $\varepsilon$-interleaving maps. Suppose that $\Mod{X}$ and $\Mod{Y}$ are persistence modules (over poset $(\R, \leq)$) and $\alpha:\Mod{X}  \to \Mod{Y}$ and $\beta:\Mod{Y} \to \Mod{X}$ form an $\varepsilon$-interleaving.
Create a poset $\Theta$ which contains two disjoint copies of $\R$; $\{(a,X)\mid a\in \R\} \cup \{(a,Y)\mid a\in \R\}$. We then define a partial order on $\Theta$ by setting 
\begin{itemize}
\item $(a,X)\leq (b,X)$ whenever $a \leq b$,
\item $(a,Y)\leq (b,Y)$ whenever $a\leq b$
\item $(a,X)\leq (b,Y)$ whenever $a+\varepsilon \leq b$, and
\item $(a,Y)\leq (b,X)$ whenever $a+\varepsilon \leq b$.
\end{itemize}
We can then construct a persistence module over $(\Theta, \leq)$ from $\alpha$ and $\beta$ via the transition maps on $\Mod{X}$, $\Mod{Y}$ and the linear maps $\alpha_a$ and $\beta_a$. Everything will commute appropriately precisely because this was required in the definition for the $\varepsilon$-interleaving.

\begin{figure}
\begin{center}
\begin{tikzpicture}[
    >=Stealth,
    thick,
    black_arrow/.style={->, thick, black},
    grey_arrow/.style={->, gray!40},
    module_line/.style={very thick},
    label_style/.style={font=\small, fill=white, inner sep=1pt}
]

\def\eps{1.2}       
\def\h{1.0}         
\def\xsep{6}        
\def\baseA{0.5}     
\def\baseAprime{3.8}  

\def\heightA{\baseA}
\def\heightB{3.0}   
\def\heightAprime{\baseAprime}
\def\heightBprime{5.3}  

\coordinate (X_bottom) at (0, 0);
\coordinate (X_top) at (0, 7);
\coordinate (X_a) at (0, \heightA);
\coordinate (X_b) at (0, \heightB);

\coordinate (Y_bottom) at (\xsep, 0);
\coordinate (Y_top) at (\xsep, 7);
\coordinate (Y_aprime) at (\xsep, \heightAprime);
\coordinate (Y_bprime) at (\xsep, \heightBprime);

\pgfmathsetmacro{\Yaeps}{\heightA + \eps}
\pgfmathsetmacro{\Yaepsh}{\heightA + \eps + \h}
\coordinate (Y_a_eps) at (\xsep, \Yaeps);
\coordinate (Y_a_eps_h) at (\xsep, \Yaepsh);

\pgfmathsetmacro{\Xaprimeeps}{\heightAprime + \eps}
\pgfmathsetmacro{\Xaprimeepsh}{\heightAprime + \eps + \h}
\coordinate (X_aprime_eps) at (0, \Xaprimeeps);
\coordinate (X_aprime_eps_h) at (0, \Xaprimeepsh);


\draw[grey_arrow] (0, 1.8) -- (\xsep, 1.8 + \eps);
\draw[grey_arrow] (0, 2.3) -- (\xsep, 2.3 + \eps + 0.4);
\draw[grey_arrow] (0, 4.2) -- (\xsep, 4.2 + \eps);
\draw[grey_arrow] (0, 5.0) -- (\xsep, 5.0 + \eps + 0.5);

\draw[grey_arrow] (\xsep, 0.8) -- (0, 0.8 + \eps + 0.3);
\draw[grey_arrow] (\xsep, 1.4) -- (0, 1.4 + \eps);
\draw[grey_arrow] (\xsep, 4.5) -- (0, 4.5 + \eps);
\draw[grey_arrow] (\xsep, 5.5) -- (0, 5.5 + \eps + 0.2);


\draw[module_line] (X_bottom) -- (X_top);
\node[above, font=\large] at (0, 7.1) {$\Mod{X}$};

\draw[module_line] (Y_bottom) -- (Y_top);
\node[above, font=\large] at (\xsep, 7.1) {$\Mod{Y}$};


\draw[black_arrow] ($(X_a) + (-0.25, 0.05)$) -- ($(X_b) + (-0.25, -0.05)$);
\node[left, label_style] at ($(X_a)!0.5!(X_b) + (-0.3, 0)$) {$\psi_a^b$};

\draw[black_arrow] ($(\xsep, \heightAprime) + (0.25, 0.05)$) -- ($(\xsep, \heightBprime) + (0.25, -0.05)$);
\node[right, label_style] at ($(\xsep, \heightAprime)!0.5!(\xsep, \heightBprime) + (0.3, 0)$) {$\phi_{a'}^{b'}$};


\draw[black_arrow] (X_a) -- (Y_a_eps);
\node[label_style, below] at ($(X_a)!0.5!(Y_a_eps) + (0, -0.1)$) {$\alpha_a$};

\draw[black_arrow] (X_a) -- (Y_a_eps_h);
\node[label_style, above] at ($(X_a)!0.45!(Y_a_eps_h) + (0, 0.1)$) 
    {$\phi_{a+\varepsilon}^{a+\varepsilon+h} \circ \alpha_a$};


\draw[black_arrow] (Y_aprime) -- (X_aprime_eps);
\node[label_style, below] at ($(Y_aprime)!0.5!(X_aprime_eps) + (0, -0.1)$) {$\beta_{a'}$};

\draw[black_arrow] (Y_aprime) -- (X_aprime_eps_h);
\node[label_style, above] at ($(Y_aprime)!0.55!(X_aprime_eps_h) + (0, 0.1)$) 
    {$\psi_{a'+\varepsilon}^{a'+\varepsilon+h} \circ \beta_{a'}$};


\fill (X_a) circle (2.5pt);
\node[left, font=\footnotesize] at ($(X_a) + (-0.15, 0)$) {$a$};

\fill (X_b) circle (2.5pt);
\node[left, font=\footnotesize] at ($(X_b) + (-0.15, 0)$) {$b$};

\fill (X_aprime_eps) circle (2.5pt);
\node[left, font=\footnotesize] at ($(X_aprime_eps) + (-0.15, 0)$) {$a'\!+\!\varepsilon$};

\fill (X_aprime_eps_h) circle (2.5pt);
\node[left, font=\footnotesize] at ($(X_aprime_eps_h) + (-0.15, 0)$) {$a'\!+\!\varepsilon\!+\!h$};

\fill (Y_aprime) circle (2.5pt);
\node[right, font=\footnotesize] at ($(Y_aprime) + (0.15, 0)$) {$a'$};

\fill (Y_bprime) circle (2.5pt);
\node[right, font=\footnotesize] at ($(Y_bprime) + (0.15, 0)$) {$b'$};

\fill (Y_a_eps) circle (2.5pt);
\node[right, font=\footnotesize] at ($(Y_a_eps) + (0.15, 0)$) {$a\!+\!\varepsilon$};

\fill (Y_a_eps_h) circle (2.5pt);
\node[right, font=\footnotesize] at ($(Y_a_eps_h) + (0.15, 0)$) {$a\!+\!\varepsilon\!+\!h$};

\end{tikzpicture}
\caption{A persistence module constructed for an $\varepsilon$-interleaving between persistence modules $\Mod{X}$ and $\Mod{Y}$. Only a subset of the transition maps are depicted. All transition maps can be constructed from compositions of the transition maps within $\Mod{X}$, the transition maps within $\Mod{Y}$, and the maps within the $\varepsilon$-morphisms $\alpha $ and $\beta$.}
\end{center}
\end{figure}

The algebraic decomposition of the persistence modules corresponding to morphisms and $\varepsilon$-interleavings are covered in \cite{urbanvcivc2024ladder}. They show that the more complicated indecomposable structure relate to when bars in the interval decompositions of $\Mod{X}$ and $\Mod{Y}$ (decomposed as persistence modules over $\R$) are contained within each other.

\section{Introducing vineyard modules}\label{sec:vineyard}
Alongside introducing the definition of vineyard module, this section will focus on the various simplifying assumptions that we will make throughout this paper. These assumptions relate to finiteness and genericity.

\begin{definition}
A \textbf{vineyard} is a map from $[T_0, T_1]$ to $\PD$ which is continuous with respect to the bottleneck distance. 
The \textbf{support} of vineyard $\Vineyard{X}=\{\Vineyard{X}(s)\}$ is the set of $s$ such that $\Vineyard{X}(s)$ contains an off-diagonal point.
A \textbf{vine} is a vineyard whose support is a non-empty interval and the number of off-diagonal points at any one time is at most one.
\end{definition}

An alternative approach is to consider the vineyard as a subset of $\R^3$ and then under genericity assumptions the vines become the connected components of the vineyard.  This makes vines a simple example of a canopy in \cite{canopy}.

While vines within vineyards have historically used the persistence diagram description and visualisation of persistence modules, in this paper we will prefer to employ the alternative language and visuals of persistence barcode. Instead of a vine being a path in 
\begin{align*}
\gamma: \supp(\gamma) &\to \PD\\
t &\mapsto (\birth(\gamma(t)), \death(\gamma(t)))
\end{align*} 
we can represent the same information as a subset of $\R^2$;  $\gamma=\{(t, h)\mid t\in \supp(\gamma), h\in [\birth(\gamma(t)), \death(\gamma(t)))\}$. This is illustrated in Figure \ref{fig:vine-two-views}.

\begin{figure}
\centering
\definecolor{vinecol}{RGB}{90,110,200}
\hfill
\begin{tikzpicture}[
    scale=0.5,
    baseline=(current bounding box.center),
    x={(1cm,0cm)},        
    y={(0.45cm,0.30cm)},  
    z={(0cm,1cm)},        
    font=\small
]
\fill[black!8]
  (0,0,0) -- (10,0,0) -- (10,6,6) -- (0,6,6) -- cycle;
\draw[black!35]
  (0,0,0) -- (10,0,0) -- (10,6,6) -- (0,6,6) -- cycle;
\draw[->,thick] (0,0,0) -- (10.8,0,0) node[right] {time $t$};
\draw[->,thick] (0,0,0) -- (0,7,0)    node[right] {birth};
\draw[->,thick] (0,0,0) -- (0,0,7)    node[above] {death};
\draw[very thick, vinecol, smooth] plot coordinates {
  (1,3,3) (2,2.5,4.2) (3.5,2.1,5.2) (5,2.0,5.6)
  (6.5,2.3,5.3) (8,3.1,4.6) (9,4,4)
};
\draw[very thick, vinecol, fill=white] (1,3,3) circle (2.6pt);
\draw[very thick, vinecol, fill=white] (9,4,4) circle (2.6pt);
\node[vinecol] at (5,2.0,6.3) {$\gamma$};
\end{tikzpicture}
\hfill
\begin{tikzpicture}[
    scale=0.5,
    baseline=(current bounding box.center),
    font=\small
]
\fill[vinecol!25]
  (2.2,3.8)
    .. controls (4,2.3) and (7,2.3) .. (9,4.6)    
    .. controls (7,7.0) and (4,7.2) .. (2.2,3.8)  
  -- cycle;
\draw[very thick, vinecol]
  (2.2,3.8) .. controls (4,2.3) and (7,2.3) .. (9,4.6);
\draw[very thick, vinecol, dashed]
  (2.2,3.8) .. controls (4,7.2) and (7,7.0) .. (9,4.6);
\draw[very thick, vinecol, fill=white] (2.2,3.8) circle (3pt);
\draw[very thick, vinecol, fill=white] (9,4.6)   circle (3pt);
\draw[->, thick] (0,0) -- (11,0)  node[right] {time $t$};
\draw[->, thick] (0,0) -- (0,8.5) node[above] {height $h$};
\node[vinecol] at (6.4,1.85) {$\birth(\gamma(t))$};
\node[vinecol] at (5.5,7.75) {$\death(\gamma(t))$};
\draw[dashed, gray] (2.2,3.8) -- (2.2,0);
\draw[dashed, gray] (9,4.6)   -- (9,0);
\draw[very thick, vinecol] (2.2,0) -- (9,0);
\draw[very thick, vinecol, fill=white] (2.2,0) circle (3pt);
\draw[very thick, vinecol, fill=white] (9,0)   circle (3pt);
\node[below, vinecol] at (5.6,-0.55) {$\supp(\gamma)$};
\end{tikzpicture}
\hfill
\caption{Visualising a vine via persistence diagrams vs via barcodes.}
\label{fig:vine-two-views}
\end{figure}

\begin{assumption}\label{assumptionLipschitz}
We will assume that all the vineyards in this paper are Lipschitz with Lipschitz constant $1$. That is, for vineyard $\Vineyard{V}: [T_0, T_1] \to \PD$ we have 
$$d_{bot}(\Vineyard{V}(s), \Vineyard{V}(t))\leq |s-t|.$$ 
\end{assumption}

If we want to remove this Lipschitz assumption we could just allow for an initial re-parameterisation of time.

In this paper we take the vineyard and make it ``more algebraic''. Instead of persistence diagrams we consider the persistence modules, and instead of using bottleneck distances we use interleaving distances -- or more precisely the actual interleavings.

\begin{definition}[Vineyard module]
A \textbf{vineyard module} over $[T_0, T_1]\subset \R$ is a family of persistence modules $\{\Mod{X}_t=(\{X^t_a\}, \{\psi(t)_a^b\}) \mid T_0\leq t\leq T_1\}$, alongside, for each $T_0\leq s<t\leq T_1$, a pair of $|s-t|$-morphisms $\alpha^{s\to t}: \Mod{X}_s \to \Mod{X}_t$ and $\beta^{t\to s}: \Mod{X}_t\to \Mod{X}_s$ such that all the compositions of maps with the same domain and codomain commute. That is for all $r<s<t$ and all $a<b$ we have 
\begin{itemize}
\item $\psi(t)_{a+|s-t|}^{b+|s-t|}\circ \alpha^{s\to t}_a = \alpha^{s\to t}_{b}\circ \psi(s)_a^b$
\item $\psi(s)_{a+|s-t|}^{b+|s-t|}\circ \beta^{t\to s}_a = \beta^{t\to s}_{b}\circ \psi(t)_a^b$
\item $\alpha^{s\to t}_{a+|s-t|} \circ \beta^{t\to s}_{a} = \psi(t)_a^{a+2|s-t|}$,
\item $\beta^{t\to s}_{a+|s-t|} \circ \alpha^{s\to t}_{a} = \psi(s)_a^{a+2|s-t|}$,
\item $\alpha^{s\to t}_{a+|r-s|}\circ \alpha^{r \to s}_a = \alpha^{r\to t}_a$, and
\item $\beta^{s\to r}_{a+|t-s|}\circ \beta^{t \to s}_a = \beta^{t\to r}_a$
\end{itemize} 
We call these $\alpha^{s\to t}$ and $\beta^{t\to s}$ the \textbf{interleaving maps} within the vineyard module.
\end{definition}

Note that the restriction of the vineyard module to two discrete time stamps $\{s,t\}$ is the persistence module of an $|s-t|$-interleaving. Just as we can interpret an $\varepsilon$-interleaving as a persistence module over a specific poset, we can alternatively interpret a vineyard module as a persistence module. The proof follows directly from the definition of vineyard module.

\begin{remark}\label{lem:poset}
Let $\mathcal{P}$ be the poset with parameter space $\Theta=\{(t,h)\mid t\in [T_0, T_1], h\in \R\}$ and order relation $(t_1, h_1)\leq (t_2, h_2)$ 
if and only if $h_1+|t_1-t_2| \leq h_2$. Every vineyard module over $[T_0, T_1]$ determines a persistence module over $\mathcal{P}$.
\end{remark}

For readers familiar with 2-parameter persistence, we can think of vineyard modules as the re-parameterisation of a subset of the plane within a 2-parameter persistence module. The interleaving maps become horizontal and vertical maps as we rotate it $45$ degrees. This is illustrated in Figure \ref{fig:rotate}. This is not directly relevant to the content of this paper but may provide intuition and context.

\begin{center}
\begin{figure}[h]
\includegraphics[height=6cm]{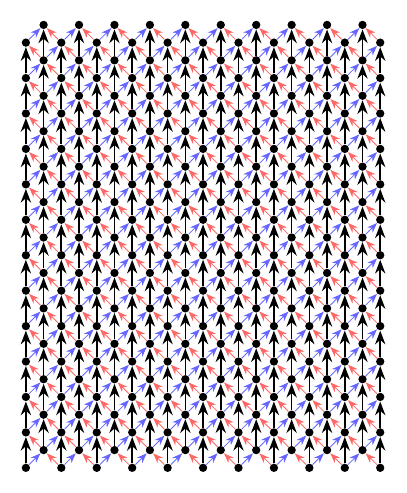}
\includegraphics[height=6cm]{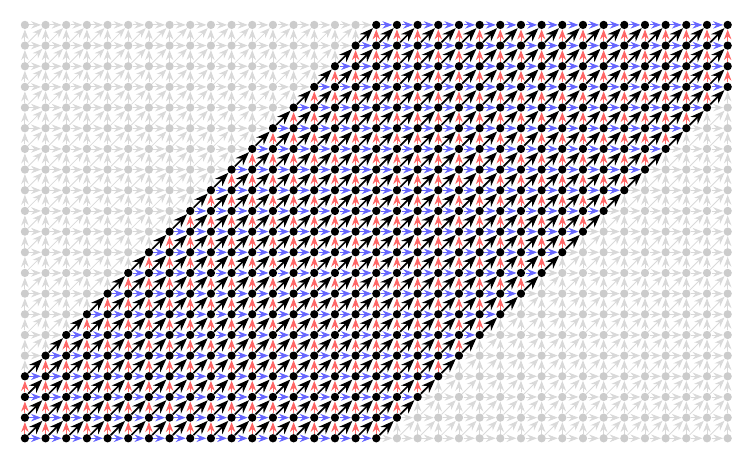}\\
(a) \hspace{6cm}(b)\hspace{3cm}
\caption{(a) A generating set of relations generating a discrete subset of the poset over which the vineyard module is a persistence module. (b) The poset for a vineyard module can be considered as a subset of the poset used in standard 2-parameter persistence modules. The interleaving maps in the vineyard module become horizontal and vertical maps in 2-parameter persistence.}\label{fig:rotate}
\end{figure}
\end{center}

As we have two-dimensions of parameters to consider we will use different terminology to help discriminate. We will refer to the parameter in the  vineyard which references which persistence module we are in as the \textbf{time} and the parameter within a persistence module (for a fixed time) as the \textbf{height}. A height within a persistence module is \textbf{critical} if it is the birth or death height of some interval within the interval decomposition. A time is \textbf{critical} if the birth and death values within the corresponding persistence module are not all distinct. The time when a point enters or leaves the diagonal is also deemed critical - philosophically it is when the birth and death times of an interval coincide (and thus it becomes empty). A more nuanced discussion on types of critical times and heights is covered in Section \ref{sec:vineyard and vector}.

\begin{assumption}\label{assumption1}
We assume that all persistence diagrams contain only finitely many off-diagonal points. We will assume that all vineyards contain only finitely many critical times and critical heights. We also will assume at the critical times that no more than at most two critical heights can coincide (as in only one pair of birth values, one pair of death values or one pair consisting of a birth and a death value can be the same).
\end{assumption}

Our genericity assumptions are analogous to those found within Cerf theory which studies one-parameter families of Morse functions \cite{cerf1970stratification}. The following is well-known as folklore.

\begin{remark}\label{lem:finte_vines}
If vineyard $\Vineyard{V}$ satisfies Assumption \ref{assumption1} then it can be written as the finite union of vines (which is unique up to reordering).
\end{remark}

Our simplifying assumptions assure us that the set of vines within a vineyard module is uniquely determined as there are no points with higher multiplicity than $1$ in any of the persistence diagrams. Whenever there are a multiplicity of points in the persistence diagrams there will be a combinatorial explosion of different potential decompositions of the vineyard into vines. We can use the decomposition of a vineyard into vines to give a consistent labelling of the basis elements within the persistence modules of a vineyard module. This will substantially ease the bookkeeping required.

There is a natural map from vineyard modules to vineyards where we forget the extra algebraic structure encoded in the interleaving maps and just keep the information of the interval decomposition at each separate time step. This information is equivalent to the persistence diagram at each time step. The interleaving maps ensure that these persistence diagrams change continuously with respect to the bottleneck distance, and thus the set of persistence diagrams form a vineyard. We will call this the \textbf{associated vineyard} of the vineyard module.

An important special case of vineyard modules are when the associated vineyard consists of a single vine.

\begin{definition}
Given a vine $\gamma$ we can construct its \textbf{vine module} $\Vine{I}(\gamma)$ by setting  $\Vine{I}(\gamma)_t$ to be the interval module $\Mod{I}[\birth(\gamma(t)), \death(\gamma(t)))$ for all $t\in \supp(\gamma)$ and the zero persistence module otherwise. Note that $(\Vine{I}(\gamma)_t)_h$ is a copy of the field for  $t\in \supp(\gamma)$ and $h\in [\birth(\gamma(t)), \death(\gamma(t)))$, and zero otherwise. 
 We construct the $|s-t|$-morphisms between $\Vine{I}(\gamma)_s$ and $\Vine{I}(\gamma)_t$ by $\alpha^{s\to t}_h: (\Vine{I}(\gamma)_s)_h \to (\Vine{I}(\gamma)_t)_{h+|s-t|}$ the identity for $\birth(\gamma(s))\leq h<\death(\gamma(t))-|s-t|$, and otherwise the zero map, and $\beta^{t\to s}_h: (\Vine{I}(\gamma)_t)_h \to (\Vine{I}(\gamma)_s)_{h+|s-t|}$ defined symmetrically. 
\end{definition}

If we consider a vine $\gamma$ as a subset of $\R^2$ then the vine module $\Vine{I}(\gamma)$  is the persistence module where we attach a one dimensional vector space over each $(t,h)$ with $t\in \supp(\gamma)$ and $h\in [\birth(\gamma(t)), \death(\gamma(t)))$, and zero outside that set, with identity linear maps where possible. In this language the vine modules become naturally equivalent to interval modules in two parameter persistence modules. 
 
 As vineyard modules are persistence modules (albeit over a complicated poset), we naturally inherit a notion of morphism and direct sum. However we will find it useful to rewrite these notions from the perspective of a function of time.
 
\begin{definition}
Let  $\Vine{A}=\{\Mod{A}_t\}$ and $\Vine{B}=\{\Mod{B}_t\}$ be vineyard modules over the same time interval.
A \textbf{morphism} $\rho:\Vine{A} \to \Vine{B}$ is a family of persistence module (over $\R$) morphisms $\rho_t:\Mod{A}_t \to \Mod{B}_t$ which commute appropriately with the interleaving and transition maps within each of the vineyard modules. 
\end{definition} 

This definition coincides with the definition of morphism between persistence modules over a poset if we consider the vineyard module as a persistence module over a poset as described in Remark \ref{lem:poset}.

\begin{definition}
Let $\Vine{V}=(\{\Mod{V}_t\},\{\alpha_V^{s\to t}\},  \{\beta_V^{t\to s}\})$ and $\Vine{W}=(\{\Mod{W}_t\}, \{\alpha_W^{s\to t}\},  \{\beta_W^{t\to s}\})$ be vineyard modules. Their \textbf{direct sum} $\Vine{V}\oplus \Vine{W}$ is the vineyard module with persistence modules $\{\Mod{V}_t\oplus \Mod{W}_t\}$ and interleaving maps $\{\alpha_V^{s\to t}\oplus \alpha_W^{s\to t}\}$ and   $\{\beta_V^{t\to s}\oplus \beta_W^{t\to s}\}$.
\end{definition}

The \textbf{zero vineyard module} is the vineyard module where the persistence module $\Mod{V}_t$ is the zero persistence module for all times $t$. From the definitions of morphism and direct sum we define a vineyard module $\Vine{X}$ as \textbf{indecomposable} if whenever $\Vine{X}$ is isomorphic to $\Vine{V}\oplus\Vine{W}$ then one of $\Vine{V}$ or $\Vine{W}$ must be the zero vineyard module. Every vineyard module is isomorphic to a direct sum of indecomposable vineyard modules. This decomposition is essentially unique (up to permutation of the indices) by Theorem \ref{thm:decomp}.

%
%
%

We can prove that in this decomposition, each vine must be fully contained in a single summand. Note here we are using our global assumption that the multiplicity of points on the corresponding persistence diagrams is never two or greater which ensures our choice of vines is unique up to reindexing.

\begin{proposition}\label{prop:whole vine}
Let $\Vine{X}$ be a vineyard module which is the direct sum of vineyard modules $\Vine{V}=\{\Mod{V}_t, \alpha_V, \beta_V\}$ and $\Vine{W}=\{\Mod{W}_t, \alpha_W, \beta_W\}$.  Let $\Vineyard{X}$ be the associated vineyard of $\Vine{X}$ and let $\gamma$ be a vine of $\Vineyard{X}$. Then either $[\birth(\gamma(t)), \death(\gamma(t)))$ is an interval in the interval decomposition of $\Mod{V}_t$ for all $t$ in the support of $\gamma$, or  $[\birth(\gamma(t)), \death(\gamma(t)))$ is an interval in the interval decomposition of $\Mod{W}_t$ for all $t$ in the support of $\gamma$.
\end{proposition}

\begin{proof}
Since $\Vine{X}=\Vine{V}\oplus\Vine{W}$ we also have $\Mod{X}_t=\Mod{V}_t \oplus \Mod{W}_t$ for all $t$. For each $t\in \supp(\gamma)$,  $[\birth(\gamma(t)), \death(\gamma(t)))$ is an interval in the interval decomposition of $\Mod{X}_t$ so it must either be an interval in $\Mod{V}_t$ or an interval in $\Mod{W}_t$. Let $A^V, A^W\subset \supp(\gamma)$ be the sets of values of $t$ where  $[\birth(\gamma(t)), \death(\gamma(t)))$ is an interval in the interval decomposition of $\Mod{V}_t$ and $\Mod{W}_t$ respectively. If either of these sets is empty we are done. Suppose neither set is empty. Since $\supp(\gamma)$ is a connected interval and $ A^W\cup A^V=\supp(\gamma)$ without loss of generality (swapping the roles of $\Vine{V}$ and $\Vine{W}$ if necessary) there exists a value $t\in A^V$ and a sequence $\{t_n\}$ in $A^W$ which converges to $t$. 

Let $\varepsilon>0$ be the minimum distance (in the $\ell^\infty$ metric) from  $[\birth(\gamma(t)), \death(\gamma(t)))$ to any interval in $\Mod{X}_t$ or the diagonal. This $\varepsilon$ is non-zero by our assumption that there are no intervals of higher multiplicity. There is an element of $s\in \{t_n\}$ with distance less than $\varepsilon/2$ from $t$. As $\Vine{V}$ is a vineyard module the bottleneck distance between the persistence diagrams corresponding to $\Mod{V}_t$ and $\Mod{V}_s$ is bounded by $|s-t|$. However, since $[\birth(\gamma(s)), \death(\gamma(s)))$ is not an interval in the interval decomposition of $\Mod{V}_s$ (as it is an interval in the decomposition of $\Mod{W}_s$)  there is no interval within $\Mod{V}_s$ suitable to pair with $[\birth(\gamma(t)), \death(\gamma(t)))$ in $\Mod{V}_t$. This causes a contradiction.
\end{proof}

From Proposition \ref{prop:whole vine} we know that if we can decompose a vineyard module into a direct sum of submodules, then the corresponding associated vineyards for each of these summands will also partition the vines by which associated vineyard they lie in.

\begin{corollary}
Let $\Vine{V}$ be a vineyard module with vines $\{\gamma_1, \gamma_2, \ldots \gamma_N\}$. Let $\bigoplus_{i=1}^K \Vine{V}_i$ be a decomposition of $\Vine{V}$ into (non-zero) vineyard modules. Then $K\leq N$ and there exists a partition $P=\sqcup_{i=1}^K P_i$ of $\{1, \ldots,  N\}$ such that the vineyard associated with $\Vine{V}_i$ consists of the union of the vines $\{\gamma_j \mid  j\in P_i\}$. 
\end{corollary}
\begin{proof}
By Proposition \ref{prop:whole vine} we know that each vine must be entirely contained in the associated vineyard of $\Vine{V}_i$ for exactly one $i$. We know that $K\leq N$ as whenever the associated vineyard of a vineyard module has no vines it must have come from the zero vineyard module. 
\end{proof}

This notion of completely decomposing into vine modules is a fundamental question for the study of vineyard modules. It covers the case  when there is no extra algebraic interaction between the vines, and as such we describe this as the trivial case.

\begin{definition}
Given a vineyard $\Vineyard{V}$ with vines $\{\gamma_1, \gamma_2, \ldots ,\gamma_N\}$,  the \textbf{trivial vineyard module with associated vineyard $\Vineyard{V}$} is the direct sum of the vine modules $\Vine{I}(\gamma_i)$.
\end{definition}

\section{Matrix representations of $\varepsilon$-morphisms}\label{sec:matrices}

Throughout we will be using matrix representations of $\varepsilon$-morphisms between persistence modules. This first requires understanding what a basis of a persistence module is. 
Given a persistence module there can be many possible choices of basis. A detailed study on the space of all possible bases in a persistence module is in \cite{jacquard2023space}.  Here we will use much more condensed notation.

Note that our definition and notation does require the assumption our persistence modules are in the form $\oplus_{i=1}^m \Mod{I}[b_i,d_i)$ and other definitions would need to be used if we were considering intervals with different choices of closed/open endpoints. Before we define a basis we must first define the birth and death time of an element within a persistence module.

\begin{definition}\label{def:birth}
Let $\Mod{X}=(X_a, \phi_{a}^{b})$ be a persistence module. For non-zero  $x\in X_a$ we define the  \textbf{birth} of $x$, denoted $\birth(x)$, as $\inf\{b\leq a\mid x\in \im(\phi_{b}^{a})\}$. We define the \textbf{death} of $x$, denoted $\death(x)$, as $\inf\{d>a\mid \phi_{a}^d(x)=0\}$.  
\end{definition}

Equivalently we can say that $x$ was \textbf{born} at $\birth(x)$ and \textbf{dies} at $\death(x)$. We are slightly overloading the notation of birth and death functions, using the same symbol whether the input is an interval module or an element within a persistence module. However, it is always clear from context which definition to use. Furthermore, in practise we will only be considering elements within a persistence module that form a basis (definition to follow) and their births and deaths coincide with the intervals they correspond to making the overlap of notation feel natural.

\begin{definition}
Suppose $\Mod{X}=(X_a, \phi_{a}^{b})$ is a persistence module with interval decomposition $\oplus_{i=1}^N \Mod{I}[b_i,d_i)$. 
The set $\{x_1, x_2, \ldots , x_N\mid x_i \in X_{b_i}\}$ is an (indexed) \textbf{basis} for $\Mod{X}$ if $\birth(x_i)=b_i$, $\death(x_i)=d_i$, and $\{\phi_{\birth(x_i)}^{a}(x_i) \mid  i \text{ such that } b_i\leq a < d_i\}$ is a basis of $X_a$ for each $a\in \R$.
\end{definition}

Throughout we will be keeping a consistent indexing of the basis elements, where the index is that of the corresponding vine. To make the bookkeeping easier later we will want to have the same number of basis elements throughout the time period of a vineyard.  To this end it will be helpful to generalise our notion of basis to allow for extra zero elements. We will introduce the definition of extended interval decompositions of a persistence module and their corresponding extended interval basis.

\begin{definition}\label{def:extended}
We say that $\oplus_{j=1}^{N} \Mod{I}[b_j,d_j)$ is an \textbf{extended interval module decomposition} of persistence module $\Mod{X}$ if $\oplus_{\{j\mid d_j>b_j\}} \Mod{I}[b_j,d_j)$ is an interval module decomposition of $\Mod{X}$. Note that $[a_i, a_i)$ is an empty interval so we have $\Mod{I}[a_i,a_i)$ is the zero persistence module and thus $\oplus_{j=1}^{N} \Mod{I}[b_j,d_j)$ and $\oplus_{\{j\mid d_j>b_j\}} \Mod{I}[b_j,d_j)$ are isomorphic.
\end{definition}

This extended interval module decomposition is philosophically similar to adding a point on the diagonal to a persistence diagram. For algebraic manipulation it will be useful to have labelled elements which act as a proxy for a basis element for each of these additional empty intervals. These are effectively indexed zero elements and combining them with a basis will create our extended bases.

\begin{definition}\label{def:basis}
Let  $\oplus_{j=1}^{N} \Mod{I}[b_j,d_j)$ be an extended interval module decomposition and let $J=\{j \mid b_j<d_j\}$. 
We say that $B'=\{x_1, \dots, x_N\mid x_i \in X_{b_i}\}$ is an \textbf{extended basis} of $\oplus_{j=1}^{N} \Mod{I}[b_j,d_j)$ if $B=\{x_j\in B' \mid j\in J\}$ is a basis of  $\oplus_{j\in J} \Mod{I}[b_j,d_j)$ and $x_i=0$ whenever $b_i=d_i$. We call the elements $x_i\in B'\backslash B$ the \textbf{zero basis elements}  and declare $\birth(x_i):=b_i=:\death(x_i)$. Note that within an extended basis the order of the indices of the zero and non-zero elements may be mixed up. We say that $B'$ is an \textbf{extension} of $B$ and that $B$ is the basis \textbf{contained} in $B'$. 
\end{definition}

If we fix a choice of extended basis of  two persistence modules, then we can construct $\varepsilon$-morphisms from matrices. However unlike linear maps between vector spaces from matrices, we will need to restrict the space of allowable matrices to respect the structure of the persistence modules.

\begin{definition}
Let $\Mod{X}=(X_a, \phi_{a}^{b})$ and $\Mod{Y}=(Y_a, \psi_a^b)$ be persistence modules with indexed extended interval module decompositions $\oplus_{j=1}^{N_X} \Mod{I}[b_j^X,d_j^X)$ and $\oplus_{i=1}^{N_Y} \Mod{I}[b_i^Y,d_i^Y)$ respectively. We say that  an $N_Y \times N_X$ matrix $A$ is \textbf{compatible} with $\varepsilon$-morphisms from $\Mod{X}$ to $\Mod{Y}$ if 
\begin{itemize}
\item $b_j^X\neq d_j^X$ and $b_i^Y\neq d_i^Y$ (the relevant intervals have positive length), and
\item $b_i^Y\leq b_j^X+\varepsilon$ and $d_i^Y\leq d_j^X+\varepsilon$ 
\end{itemize}
whenever $A(i,j)\neq 0$. The set of all matrices that are compatible with $\varepsilon$-morphisms from $\Mod{X}$ to $\Mod{Y}$ forms a vector space which we will denote $\MatC(\Mod{X},\Mod{Y},\varepsilon)$.
\end{definition}

\begin{lemma}\label{lem:constructionMor}
For fixed extended bases $B_X=\{x_i\}$ and $B_Y=\{y_i\}$ of persistence modules $\Mod{X}=(\{X_t\}, \{\phi_s^t\})$ and $\Mod{Y}=(\{Y_t\}, \{\psi_s^t\})$, and a matrix $A\in \MatC(\Mod{X}, \Mod{Y}, \varepsilon)$ we can construct an $\varepsilon$-morphism $$\Mor_{(B_X,  B_Y, \varepsilon)}(A): \Mod{X} \to \Mod{Y}$$ by setting
$$\Mor_{(B_X, B_Y, \varepsilon)}(A)_t\left(\sum_{\{j\mid \birth(x_j)\leq t<\death(x_j)\}} \lambda_j \phi_{\birth(x_j)}^{t}(x_j)\right)= \sum_{\{j\mid \birth(x_j)\leq t<\death(x_j)\}} \sum_{\{i\mid \birth(y_i)\leq t+\varepsilon\}} \lambda_j A(i,j) \psi_{\birth(y_i)}^{t +\varepsilon}(y_i).$$
\end{lemma}

This map from matrices to morphisms is linear and we can characterise its kernel.

\begin{lemma}\label{lem:kernel}
For any choice of bases $B_X$ and $B_Y$ of persistence modules $\Mod{X}$ and $\Mod{Y}$, the map 
$$\Mor_{(B_X, B_Y, \varepsilon)}:\MatC(\Mod{X},\Mod{Y}, \varepsilon) \to \MorC(\Mod{X}, \Mod{Y}, \varepsilon)$$ is linear and $A\in \ker(\Mor_{(B_X,B_Y,\varepsilon)})$ if and only if 
$A(i,j)=0$ for all $(i,j)$ where $\birth(x_j)+\varepsilon <\death(y_i)$.
\end{lemma}

Of interest is a restricted subset of $\varepsilon$-compatible matrices which we will call \textbf{observable}. These are $A\in \MatC(\Mod{X},\Mod{Y}, \varepsilon)$ such that $A(i,j)=0$ for all $i,j$ where $d_i^Y \leq b_j^X+\varepsilon$. Denote this set of matrices by $\ObMat(\Mod{X}, \Mod{Y}, \varepsilon)$. The above lemma ensures that $\Mor_{(B_X,B_Y, \varepsilon)}$ is injective when restricted to $\ObMat(\Mod{X}, \Mod{Y}, \varepsilon)$. There is a natural projection map 
$$\pi_{(\Mod{X}, \Mod{Y}, \varepsilon)}:(N_Y \times N_X)\text{-matrices} \to (N_Y \times N_X)\text{-matrices} $$ defined by 
$$(\pi_{(\Mod{X}, \Mod{Y}, \varepsilon)}(A))(i,j)=\begin{cases}
A(i,j) & \text{if }d_i^Y>b_j^X+\varepsilon, \text{ and } b_i^Y \neq d_i^Y \text{ and } b_j^X \neq  d_j^X
\\
0 & \text{otherwise}.
\end{cases}$$

When we restrict $\pi_{(\Mod{X}, \Mod{Y}, \varepsilon)}$ to $\MatC(\Mod{X},\Mod{Y}, \varepsilon)$ then the image is $\ObMat(\Mod{X}, \Mod{Y}, \varepsilon)$.

\begin{center}
\begin{figure}

\begin{tikzpicture}[scale=0.5]
    \begin{scope}[xshift=0cm]
        \node[bluelayer, text width=2cm, align=center] at (-1,6) {$\Mod{X}$};
        
        \draw[bluelayer, thick] (0.5,0) -- (0.5,7);
        \filldraw[bluelayer] (0.5,0) circle (0.1);
        \draw[bluelayer, thick] (0.5,7) circle (0.1);
        \node[bluelayer, left] at (0.3,-0.3) {$x_1$};
        
        \draw[bluelayer, thick] (1.5,4) -- (1.5,11);
        \filldraw[bluelayer] (1.5,4) circle (0.1);
        \draw[bluelayer, thick] (1.5,11) circle (0.1);
        \node[bluelayer, left] at (1.3,3.7) {$x_2$};
    \end{scope}

    \begin{scope}[xshift=7cm]
        \node[redlayer, text width=2cm, align=center] at (0,6) {$\Mod{Y}$};
        
        \draw[redlayer, thick] (0.5,-3) -- (0.5,2);
        \filldraw[redlayer] (0.5,-3) circle (0.1);
        \draw[redlayer, thick] (0.5,2) circle (0.1);
        \node[redlayer, left] at (1.4,-3.3) {$y_1$};
        
        \draw[redlayer, thick] (1.5,-1) -- (1.5,6);
        \filldraw[redlayer] (1.5,-1) circle (0.1);
        \draw[redlayer, thick] (1.5,6) circle (0.1);
        \node[redlayer, left] at (2.4,-1.3) {$y_2$};
        
        \draw[redlayer, thick] (2.5,6) -- (2.5,9);
        \filldraw[redlayer] (2.5,6) circle (0.1);
        \draw[redlayer, thick] (2.5,9) circle (0.1);
        \node[redlayer, left] at (3.4,5.7) {$y_3$};
    \end{scope}
    
    \draw[gray, dashed] (-0.5,0.5) -- (10,0.5);
    \node[gray, right] at (10.2,0.5) {$\birth(x_1)+\varepsilon$};
    
    \draw[gray, dashed] (-0.5,7.5) -- (10,7.5);
    \node[gray, right] at (10.2,7.5) {$\death(x_1)+\varepsilon$};
    
    \draw[gray, dashed] (-0.5,4.5) -- (10,4.5);
    \node[gray, right] at (10.2,4.5) {$\birth(x_2)+\varepsilon$};
    
    \draw[gray, dashed] (-0.5,11.5) -- (10,11.5);
    \node[gray, right] at (10.2,11.5) {$\death(x_2)+\varepsilon$};
\end{tikzpicture}

$\MatC(\Mod{X}, \Mod{Y}, \varepsilon)= \left\{\begin{pmatrix}
a & b \\
c & d \\
0 & 0
\end{pmatrix} \mid a,b,c, d\in \mathbb{F} \right\}
\qquad
\ObMat(\Mod{X}, \Mod{Y}, \varepsilon)= \left\{\begin{pmatrix}
a & 0 \\
c & d \\
0 & 0
\end{pmatrix}\mid a,c,d\in \mathbb{F} \right\}$
\caption{An example of the space of matrices corresponding to $\varepsilon$-morphisms. For $A\in \MatC(\Mod{X}, \Mod{Y}, \varepsilon)$ we are forced to have $A(3,1)$ and $A(3,2)$ to be zero as $\birth(y_3)>\birth(x_j)+\varepsilon$. We can't ``observe'' the value $b$ as $\death(y_1)<\birth(x_2)+\varepsilon$ so $\psi_{\birth(y_1)}^{\birth(x_2)+\varepsilon}(y_1)=0$.}
\end{figure}
\end{center}

Once we have fixed a choice of basis $B_X=\{x_j\}$ for $\Mod{X}=(\{X_a\}, \{\phi_a^b\})$, and $B_Y=\{y_i\}$ for $\Mod{Y}=(\{Y_a\}, \{\psi_a^b\})$, we can construct a partial inverse to $\Mor_{(B_X,B_Y, \varepsilon)}$ creating a matrix representation for any $\varepsilon$-morphism $\alpha:\Mod{X} \to \Mod{Y}$ with respect to these bases.

\begin{definition}
Given an $\varepsilon$-morphism $\alpha:\Mod{X} \to \Mod{Y}$ we construct its \textbf{matrix with respect to a choice of bases} $B_X$ and $B_Y$ as $\Mat_{B_X}^{B_Y}(\alpha)$ by requiring
$$\alpha_{\birth(x_j)}(x_j)=\sum_{\{i\mid \birth(x_j)+\varepsilon \in [\birth(y_i), \death(y_i))\}} \Mat_{B_X}^{B_Y}(\alpha)(i,j)\psi_{\birth(y_i)}^{\birth(x_j) +\varepsilon}(y_i).$$
and setting $\Mat_{B_X}^{B_Y}(\alpha)(i,j)=0$ whenever $\birth(x_j) +\varepsilon\notin [\birth(y_i), \death(y_i))$.
In other words, the $j^{th}$ column in $\Mat_{B_X}^{B_Y}(\alpha)$ corresponds to $\alpha_{\birth(x_j)}(x_j)$.
\end{definition}

This is well defined as  $\{\psi_{\birth(y_i)}^{t}(y_i)\mid \birth(y_i)\leq t<\death(y_i)\}$ is a basis of  the vector space $Y_t\in \Mod{Y}$.
Note that if $\Mat_{B_X}^{B_Y}(\alpha)(i,j)\neq 0$ then $$\birth(y_i) \leq \birth(x_j)+\varepsilon<\death(y_i)\leq \death(x_j)+\varepsilon.$$

\begin{lemma}\label{lem:Mat(Mor)}
Let $B_X$ and $B_Y$ be bases of persistence modules $\Mod{X}$ and $\Mod{Y}$. Then  $$\Mor_{(B_X,B_Y, \varepsilon)}(\Mat_{B_X}^{B_Y}(\alpha))=\alpha$$ for all $\alpha\in \MorC(\Mod{X}, \Mod{Y}, \varepsilon)$, and $$\Mat_{B_X}^{B_Y}(\Mor_{(B_X,B_Y, \varepsilon)}(A))=\pi_{(\Mod{X}, \Mod{Y}, \varepsilon)}(A)$$ for all $A\in  \MatC(\Mod{X}, \Mod{Y}, \varepsilon)$. 

In particular, $\Mat_{B_X}^{B_Y}(\Mor_{(B_X,B_Y, \varepsilon)}(A))=A$ for all $A\in \ObMat(\Mod{X}, \Mod{Y}, \varepsilon)$.
\end{lemma}

We will be considering one-parameter families of persistence modules, where the compositions of morphisms  only depend on the beginning and ending time. As a result it is very helpful that $\Mor$ plays nicely with compositions.

\begin{lemma}\label{lem:compositionMor}
Fix extended bases $B_X$, $B_Y$ and $B_Z$ for persistence modules $\Mod{X}$, $\Mod{Y}$ and $\Mod{Z}$.
 If $P\in \MatC(\Mod{X}, \Mod{Y}, \varepsilon)$ and $Q\in\MatC(\Mod{Y}, \Mod{Z}, \delta)$ then $QP\in \MatC(\Mod{X}, \Mod{Z}, \varepsilon+\delta)$ and $$\Mor_{(B_X, B_Z, \varepsilon+\delta)}(QP)=\Mor_{(B_Y,B_Z, \delta)}(Q)\circ \Mor_{(B_X, B_Y, \varepsilon)}(P).$$
\end{lemma}

For linear maps between vectors spaces $f:X\to Y$ and $g:Y\to Z$ the matrix representation from the spaces of linear maps to the space of matrices behaves nicely with composition;  under consistent choices of bases we have $\Mat(g)\Mat(f)=\Mat(g\circ f)$. This does not hold for matrix representations of $\varepsilon$-morphisms between persistence modules. 

\begin{figure}

\begin{tikzpicture}[
    >=Stealth,
    interval/.style={line width=3pt, line cap=round},
    faint/.style={line width=0.5pt, gray!30},
    morphism/.style={->, thick, dashed}
]

\draw[thick, ->] (0,-0.5) -- (0,6) node[above] {$t$};
\foreach \y in {0,1,2,3,4,5} {
    \draw (-0.1,\y) -- (0.1,\y) node[left] {$\y$};
}

\foreach \y in {0,1,2,3,4,5} {
    \draw[dotted, gray!40, line width=0.5pt] (0,\y) -- (7,\y);
}

\draw[faint] (2,-0.5) -- (2,6);
\draw[thick, green!50!black] (2,3) -- (2,5);
\filldraw[green!50!black] (2,3) circle (2pt);
\draw[green!50!black] (2,5) circle (2pt);
\node[above, green!50!black, font=\footnotesize] at (2,5.5) {$\Mod{X}=\Mod{I}[3,5)$};

\draw[faint] (4,-0.5) -- (4,6);
\draw[thick, red!60!black] (4,1) -- (4,4);
\filldraw[red!60!black] (4,1) circle (2pt);
\draw[red!60!black] (4,4) circle (2pt);
\node[above, red!60!black, font=\footnotesize] at (4,4.5) {$\Mod{Y}=\Mod{I}[1,4)$};

\draw[faint] (6,-0.5) -- (6,6);
\draw[thick, blue!60!black] (6,0) -- (6,2);
\filldraw[blue!60!black] (6,0) circle (2pt);
\draw[blue!60!black] (6,2) circle (2pt);
\node[above, blue!60!black, font=\footnotesize] at (6,2.5) {$\Mod{Z}=\Mod{I}[0,2)$};

\draw[morphism, black] (2,3.5) -- (4,3.5) node[midway, above, font=\small] {$\alpha$};
\draw[morphism, black] (4,1.5) -- (6,1.5) node[midway, above, font=\small] {$\beta$};

\node[align=center, font=\footnotesize] at (4,-1.5) {
    $\text{Mat}(\alpha) = [1], \quad \text{Mat}(\beta) = [1]$ \\[3pt]
    $\text{Mat}(\beta \circ \alpha) = [0] \neq [1][1] = \text{Mat}(\beta) \cdot \text{Mat}(\alpha)$
};

\end{tikzpicture}

\caption{Counterexample for composition of matrices (without suitable projection). We have $\Mat_{\{x\}}^{\{z\}}(\beta \circ \alpha)=[0]\neq [1][1]=\Mat_{\{y\}}^{\{z\}}(\beta)\Mat_{\{x\}}^{\{y\}}(\alpha)$}\label{fig:composition}
\end{figure}

We can find simple counterexamples, as illustrated in Figure \ref{fig:composition}. Here set $\Mod{X}=\Mod{I}[3,5)$ with basis element $x$, $\Mod{Y}=\Mod{I}[1,4)$ with basis element $y$, and $\Mod{Z}=\Mod{I}[0,2)$ with basis element $z$. Denote all transition maps by $\phi$. There is a $0$-morphism $\alpha:\Mod{X} \to \Mod{Y}$ with $\alpha_3(x)=\phi_1^3(y)\neq 0$ and thus $\Mat_{\{x\}}^{\{y\}}(\alpha)=[1]$. There is a $0$-morphism $\beta:\Mod{Y} \to \Mod{Z}$ with $\beta_1(y)=\phi_0^1(z)\neq 0$ and thus $\Mat_{\{y\}}^{\{z\}}(\beta)=[1]$. We have $\beta \circ \alpha :\Mod{X} \to \Mod{Z}$ is the $0$ map as no non-trivial $0$-morphisms exist. Thus $$\Mat_{\{x\}}^{\{z\}}(\beta \circ \alpha)=[0]\neq [1][1]=\Mat_{\{y\}}^{\{z\}}(\beta)\Mat_{\{x\}}^{\{y\}}(\alpha).$$ This counterexample demonstrates that an obstacle to matrix presentations having nice composition properties is attaining matrices that are not observable. However, not all is lost, as we can compensate via our projection map.

\begin{proposition}\label{prop:compositionMat}
Fix extended bases $B_X$, $B_Y$ and $B_Z$ for persistence modules $\Mod{X}$, $\Mod{Y}$ and $\Mod{Z}$. Let $\pi_{(\Mod{X}, \Mod{Z}, \varepsilon+\delta)}$ be the projection map onto $\ObMat(\Mod{X}, \Mod{Z}, \varepsilon+\delta)$.
Then for any $\varepsilon$-morphism $\alpha:\Mod{X}\to \Mod{Y}$ and any $\delta$-morphism $\beta:\Mod{Y} \to \Mod{Z}$ we have $$\pi_{(\Mod{X}, \Mod{Z}, \varepsilon+\delta)}(\Mat_{B_Y}^{B_Z}(\beta) \Mat_{B_X}^{B_Y}(\alpha))=\Mat_{B_X}^{B_Z}(\beta \circ \alpha).$$
\end{proposition}
 \begin{proof}
 From Lemma  \ref{lem:Mat(Mor)} we know $\alpha=\Mor_{(B_X,B_Y, \varepsilon)}\left(\Mat_{B_X}^{B_Y}(\alpha)\right)$ and $\beta=\Mor_{(B_Y,B_Z, \delta)}\left(\Mat_{B_Y}^{B_Z}(\beta)\right)$. Combining with Lemma \ref{lem:compositionMor} we have 
 \begin{align*}\beta \circ \alpha&= \Mor_{(B_Y,B_Z, \delta)}\left(\Mat_{B_Y}^{B_Z}(\beta)\right) \circ \Mor_{(B_X,B_Y, \varepsilon)}\left(\Mat_{B_X}^{B_Y}(\alpha)\right)\\
 &=\Mor_{(B_X,B_Z, \varepsilon+\delta)}\left(\Mat_{B_Y}^{B_Z}(\beta)\Mat_{B_X}^{B_Y}(\alpha)\right).
 \end{align*}
By applying $\Mat_{B_X}^{B_Z}$ to both sides we get $$\Mat_{B_X}^{B_Z}(\beta \circ \alpha)=\Mat_{B_X}^{B_Z}\left(\Mor_{(B_X,B_Z, \varepsilon+\delta)}\left(\Mat_{B_Y}^{B_Z}(\beta)\Mat_{B_X}^{B_Y}(\alpha)\right)\right)$$
which is $\pi_{(\Mod{X}, \Mod{Z}, \varepsilon+\delta)}\left(\Mat_{B_Y}^{B_Z}(\beta)\Mat_{B_X}^{B_Y}(\alpha)\right)$ by Lemma  \ref{lem:Mat(Mor)}.
\end{proof}

An important special case of $\varepsilon$-morphisms and their matrix representations are isomorphisms (with $\varepsilon=0$). Here we can think of the morphism or matrix as creating a new basis from the old one. Instead of using change of basis matrices (such as explored in \cite{jacquard2023space}) we will instead consider the new basis as a linear transformation of the previous basis. That is, we wish to write the new basis elements as a linear combination of the old basis elements. This will reduce the linear algebra calculations needed later and avoid the issue of using inverses which are not well-defined when working with extended bases. This will also allow us to use the set up of matrices of morphisms in the previous section. Variants of some of the results here for (non-extended) bases are proved in \cite{jacquard2023space} but with such vastly different notation and perspective that it is much easier to prove them directly here than attempt the translation.

\begin{definition}
Let $\Mod{Y}=(\{Y_t\}, \{\phi_s^t\})$ be a persistence module with extended interval decomposition $\oplus_{i=1}^N \Mod{I}[b_i,d_i)$ such that no intervals appear with multiplicity greater than $1$ and let $J=\{j\mid b_j<d_j\}$ be the subset of indices that determine the non-zero basis elements (that is $\{y_j\mid j\in J\}$ is a non-extended  basis of $\Mod{Y}$). Then an $N$ by $N$ matrix $P$ is a \textbf{basis transformation matrix} if  $P(j,j)\neq 0$ for all $j\in J$ and $P(i,j)\neq 0$, with $i\neq j$, implies $i,j\in J$, $b_i\leq b_j$ and $d_i\leq d_j$.
\end{definition}

A good reason for allowing $P(i,i)\neq 0$ even when $i\notin J$ is that the identity matrix will be a basis transformation matrix (which, as we would hope, leaves the basis unchanged). Denote the space of basis transformation matrices  for $\Mod{Y}$ (with a given extended interval decomposition) as $\BTM(\Mod{Y})$. In this paper we will only consider basis transformations which retain the same indexing with respect to some interval decomposition. It would be possible to generalise to allow for permutations of the indexing of the intervals or to bases where the interval decomposition has intervals with multiplicity greater than $1$. However in the context of generic vineyard modules this is unnecessary and a potential source of confusion.

\begin{lemma}\label{lemma:newbasis}
Let $\Mod{Y}=(\{Y_t\}, \{\phi_s^t\})$ be a persistence module with extended interval decomposition $\oplus_{i=1}^N \Mod{I}[b_i,d_i)$ such that no intervals appear with multiplicity greater than $1$.
 If $P$ is a basis transformation matrix then for any extended basis $B=\{y_1, \ldots y_N\}$ for $\oplus_{i=1}^N \Mod{I}[b_i,d_i)$ the set $B^{new}:=\{y^{new}_1, y^{new}_2, \ldots , y^{new}_N\}$ forms an extended basis of  $\Mod{Y}$ where 
$$y^{new}_j=\sum_i P(i,j) \phi_{\birth(y_i)}^{\birth(y_j)}(y_i)\in Y_{\birth(y_j)}$$
and the (non-extended) basis contained in $B^{new}$ is $\{y_j^{new}\mid b_j < d_j\}$.
\end{lemma}

Given an extended basis $B$ of $\Mod{Y}$ and $P\in \BTM(\Mod{Y})$ we call $B^{new}:=\{y^{new}_1, y^{new}_2, \ldots , y^{new}_N\}$ where $$y^{new}_j=\sum_i P(i,j) \phi_{\birth(y_i)}^{\birth(y_j)}(y_i)\in Y_{\birth(y_j)}$$ the \textbf{transformed basis under $P$}.  We denote this new basis by $P(B)$.

As would be expected from our previous composition results, we can express how the matrices of morphisms change when we transform the basis of the codomain. 

\begin{lemma}\label{lem:transformed_morphism}
Let $\Mod{X}$ and $\Mod{Y}$ 
be persistence modules with bases $B_X$ and $B_Y$, and $\alpha\in \MorC(\Mod{X}, \Mod{Y}, \varepsilon)$.
If $\Mat_{B_X}^{B_Y}(\alpha)=PA$ with $P\in \BTM(\Mod{Y})$ and $A\in\MatC(\Mod X,\Mod Y,\varepsilon)$, then for new basis $B_Y^{new}=P(B_Y)$ of $\Mod{Y}$ we have $\Mat_{B_X}^{B_Y^{new}}(\alpha)=\pi_{(\Mod{X}, \Mod{Y}, \varepsilon)}(A)$. 
\end{lemma}

It will be useful later to have the analogous result where we compare the matrices of the morphisms after a basis transformation of the domain (instead of the codomain).

\begin{lemma}\label{lem:transformed_morphism_domain}
Let $\Mod{X}$ and $\Mod{Y}$ be persistence modules with bases $B_X$ and $B_Y$, and
$\alpha\in \MorC(\Mod{X}, \Mod{Y}, \varepsilon)$. If $A=\Mat_{B_X}^{B_Y}(\alpha)$ with
$P\in \BTM(\Mod{X})$, then for the new basis $B_X^{new}=P(B_X)$ of $\Mod{X}$ we have
$\Mat_{B_X^{new}}^{B_Y}(\alpha)=\pi_{(\Mod{X}, \Mod{Y}, \varepsilon)}(AP)$.
\end{lemma}

\section{Vines and Matrices, and Vineyard and Vector Representations}\label{sec:vineyard and vector}

We are interested in objects that completely determine the algebraic structure of a vineyard module (up to isomorphism), while not expecting them to be unique. We will start with rather complicated and large representations and then drill down to more concise ones. To begin with we will define a vines and matrices representation for vineyard modules. 
\begin{definition}
Let $\Vine{V}=(\Mod{V}_t, \alpha^{s\to t}, \beta^{t\to s})$ be a vineyard module over time interval $[T_0,T_1]$.
A \textbf{vines and matrices} representation of $\Vine{V}=(\Mod{V}_t, \alpha^{s\to t}, \beta^{t\to s})$ consists of 
 a set of vines $\{\gamma_1, \gamma_2, \ldots \gamma_N\}$ each of which is supported over a connected subset of $[T_0,T_1]$, alongside 
 families of matrices $$\{M(\alpha^{s\to t})\mid T_0\leq s<t\leq T_1\} \quad \text{ and }\quad \{M(\beta^{t\to s}) \mid T_0\leq s<t\leq T_1\},$$
such that there exists a family of extended bases $\{B_t\}$ (where $B_t$ is an extended basis of $\Mod{V}_t$ and the indexing agrees with the order of the vines) such that $M(\alpha^{s\to t})=\Mat_{B_s}^{B_t}(\alpha^{s\to t})$  and $M(\beta^{t\to s})=\Mat_{B_t}^{B_s}(\beta^{t\to s})$. 
We call $\{B_t\}$ an associated family of bases for that representation.
\end{definition}

Notably this representation via vines and matrices is not unique as it depends on the choice of bases. However, we do at least know that if two vines and matrices representations agree then their corresponding vineyard modules are isomorphic.

\begin{proposition}\label{prop:isom}
Let $\Vine{V}=(\Mod{V}_t, \alpha_V^{s\to t}, \beta_V^{t\to s})$ and $\Vine{W}=(\Mod{W}_t,\alpha_W^{s\to t}, \beta_W^{t\to s})$  be vineyard modules over the same vineyard with vines $\{\gamma_i\}$. Suppose that $(\{\gamma_i\}, \{ M(\alpha^{s\to t})\},\{M(\beta^{t\to s})\})$ is a vines and matrices representation of both $\Vine{V}$ and $\Vine{W}$. Then $\Vine{V}$ and $\Vine{W}$ are isomorphic as vineyard modules.
\end{proposition}
\begin{proof}
There must exist bases $B_t^V$ of $\Mod{V}_t$ and $B_t^W$ of $\Mod{W}_t$ such that 
$$\Mat_{B^V_s}^{B^V_t}(\alpha_V^{s\to t})=M(\alpha^{s\to t})=\Mat_{B^W_s}^{B^W_t}(\alpha_W^{s\to t}) \quad \text{ and } \quad \Mat_{B^V_t}^{B^V_s}(\beta_V^{t\to s})=M(\beta^{t\to s})=\Mat_{B_t^W}^{B_s^W}(\beta_W^{t\to s})$$ for all $s<t$.

Let $\rho_t=\Mor_{(B_t^V, B_t^W, 0)}(\Id)$. From Lemma \ref{lem:compositionMor} we know that $\rho_t$ commutes appropriately with all the interleaving and transition maps and determines a morphism $\rho: \Vine{V} \to \Vine{W}$. This vineyard module morphism $\rho$ is also clearly invertible with inverse $(\rho^{-1})_t=\Mor_{(B_t^W, B_t^V, 0)}(\Id)$.
\end{proof}

Given the vines and matrices representations of a finite number of vineyard modules there is an obvious construction of a vines and matrices representation of the vineyard module of their direct sum via block matrices. Being able to write the matrices in a block-matrix form provides an easy sufficient condition for when a vineyard module decomposes. 

\begin{proposition}\label{prop:block}
Let $\Vine{X}=(\Mod{X}_t, \alpha^{s\to t}, \beta^{t\to s})$ be a vineyard module with vines and matrices representation 
$(\{\gamma_i\}, \{M(\alpha^{s\to t})\}, \{M(\beta^{t\to s})\})$. 
Suppose that for all $T_0\leq s<t \leq T_1$ both $M(\alpha^{s\to t})$ and $M(\beta^{t\to s})$ are block diagonal matrices with block index sets $J_1, J_2, \ldots J_m$. Then we can construct vineyard modules $\Vine{X}_1, \Vine{X}_2, \ldots, \Vine{X}_m$ with $\Vine{X}\cong \oplus_{j=1}^m\Vine{X}_j$ where $\Vine{X}_j$ has a vines and matrices representation  $$(\{\gamma_i\mid i\in J_j\}, \{\pi_{J_j}(M(\alpha^{s\to t}))\}, \{ \pi_{J_j}( M(\beta^{t\to s}))\}).$$ Here $\pi_{J_j}(A)$ is the restriction of matrix $A$ to the coordinates in $J_j$. 
\end{proposition}
\begin{proof}
Let $M_j(\alpha^{s\to t})$ and $M_j(\beta^{t\to s})$ be the sub-block matrices with block index set $J_j$ within $M(\alpha^{s\to t})$ and $M(\beta^{t\to s})$ respectively. For each $t$ choose any basis $B_j(t)$ for a persistence module with intervals $\{\gamma_i(t) \mid i\in J_j\}$. We can construct a vineyard module $\Vine{X}_j$ with vines $\{\gamma_i \mid i\in J_j\}$ by defining its interleaving maps by $\Mor_{(B_j(s), B_j(t), |s-t|)}(M_j(\alpha^{s\to t}))$ and $\Mor_{(B_j(t), B_j(s), |s-t|)}(M_j(\beta^{t\to s}))$. Since the multiplication of block diagonal matrices stays block diagonal, and the commutativity conditions holds for $ \{M(\alpha^{s\to t})\}, \{M(\beta^{t\to s})\}$, it is easy to check that the commutativity conditions also hold for $ \{M_j(\alpha^{s\to t})\}, \{M_j(\beta^{t\to s})\}$ for each $j$.

From the definition of direct sum $\oplus_{j=1}^m\Vine{X}_j$ we automatically get back the vines and matrices representation $(\{\gamma_i\}, \{M(\alpha^{s\to t})\}, \{M(\beta^{t\to s})\})$. By Proposition \ref{prop:isom} we have $\Vine{X}$ is isomorphic to $ \oplus_{j=1}^m\Vine{X}_j$ as they have a common vines and matrices representation.
\end{proof}

Finding necessary conditions for decomposition in terms of the vines and matrices representation is much harder. Depending on the choice of associated bases, we can not expect that a vineyard module which is the direct sum of vine modules will necessarily have matrices that will split up into a block diagonal form.

One potential way to construct an invariant for vineyard modules is to consider the equivalence classes of vines and matrices representations. While theoretically definable, a more satisfying version would be  to construct in some canonical representative for each equivalence class. The first step towards this elusive goal is to find ways to transform the bases of the persistence modules over the different times $t$ so that the matrices for the interleaving $\varepsilon$-morphisms are locally of a nice form.  
  
We have to identify the local obstructions for transforming the bases of the persistence modules so that the interleaving $\varepsilon$-morphisms are morphisms constructed from the identity matrix. If we restrict to sufficiently short intervals and consider them independently of their (time) surroundings we would have far more freedom in our choices of bases than if we condition that the basis at the start of the interval is fixed. It is this second scenario we are interested in. More precisely, we want obstruction to modifying the bases just a bit in the future given a fixed choice of the basis right now.  We can articulate these obstructions in terms of changes in a natural partial order on the intervals within the persistence modules.

A standard partial order on $\R^2$ is $(b_1, d_1)\leq (b_2, d_2)$ whenever $b_1\leq b_2$ and $d_1\leq d_2$. We can adapt this to provide a partial order on the intervals within a persistence module $\Mod{X}=\oplus_i \Mod{I}[b_i, d_i)$ where $\Mod{I}[b_i, d_i)\leq \Mod{I}[b_j, d_j)$ whenever $b_i\leq b_j$ and $d_i\leq  d_j$. It is worth observing that for matrix $A$ to be in $\MatC(\Mod{X}, \Mod{X}, 0)$  we require that $A(i,j)\neq 0$ implies $\Mod{I}[b_i, d_i)\leq \Mod{I}[b_j, d_j)$.

We want to identify the times where the partial order on the intervals corresponding to the vines changes. This critical behaviour can be determined on the vineyard level as it is independent of the algebra of the interleaving maps.

As the partial order of intervals only depends on the partial order of the set of birth and death times we wish to track when this order changes. We need to consider the locations where birth and or death values coincide, or a new interval appears/disappears (philosophically its own birth and death values coincide). We define these time values as \textbf{critical}. As the number of vines does not necessarily remain constant we will have to define these via one-sided limits and using the supports of the vines.

\begin{definition}
For each value $t$ let $J^\pm(t)=\{j\mid \lim_{s\to t^\pm} 1_{\supp(\gamma_j)}(s)=1\}$ and define multi-sets $$\Crit^\pm(t)=\{ \lim_{s\to t^\pm} \birth(\gamma_j(s)) \mid j\in J^\pm(t)\} \cup \{ \lim_{s\to t^\pm} \death(\gamma_j(s)) \mid j\in J^\pm(t)\}. $$ These are the one-sided limits of the critical height values. 

The time $t$ is \textbf{right critical} (respectively \textbf{left critical}) if the multiset $\Crit^+(t)$ (respectively $\Crit^-(t)$) contains a repeated value. A time is \textbf{critical} if it is left and/or right critical.
\end{definition}

Note that generally $\Crit^+(t)=\Crit^-(t)$, as the one-sided limits will only differ in the cases where a point in the vineyard either enters or leaves the diagonal. When the number of vines change then the appropriate one-sided limit will have repetition of critical heights as the birth and death value of the disappearing vine will coincide in the limit. We will use $\Crit(t)$ when $\Crit^+(t)=\Crit^-(t)$. Our assumptions on our vineyards ensure that there are only finitely many critical times.

When a relation in the partial order ceases to exist, either going forwards or going backwards, we say that time is either forwards or backwards incompatible.

\begin{definition}
Let $\Vineyard{V}$ be a vineyard with vines $\{\gamma_i\}$. We say that  $\Vineyard{V}$ is \textbf{forwards incompatible} at $s$  by vines $(\gamma_k, \gamma_l)$ if \begin{itemize}
\item $s\in \supp(\gamma_k)\cap \supp(\gamma_l)$, 
\item $\gamma_k(s)\leq \gamma_l(s)$ and 
\item $\gamma_k(t)\nleq \gamma_l(t)$ for all $t\in (s, s+\delta)$ for some $\delta>0$. 
\end{itemize} 
We say that $s$ is \textbf{forwards compatible} if it is not forwards incompatible. We say $\Vineyard{V}$ is \textbf{backwards incompatible} at $t$ by vines $(\gamma_k, \gamma_l)$  if 
\begin{itemize}
\item  $t\in \supp(\gamma_k)\cap \supp(\gamma_l)$
\item $\gamma_k(t)\leq \gamma_l(t)$ and
\item $\gamma_k(s)\nleq \gamma_l(s)$ for all $s\in (t-\delta, t)$ for $\delta>0$ sufficiently small. 
\end{itemize}
We say that $t$ is \textbf{backwards compatible} if it is not backwards incompatible.
\end{definition}

\begin{figure}[h]
\centering
\begin{minipage}{\textwidth}
\begin{minipage}{0.3\textwidth}
\raggedleft\textbf{Vines $\gamma_1$ and $\gamma_2$}
\end{minipage}%
\begin{minipage}{0.7\textwidth}
\qquad
\begin{tikzpicture}[scale=0.35]
\begin{scope}[xshift=-7cm, y={(0cm, 0.25cm)}]  
    \definecolor{redlayer}{RGB}{200, 80, 80}
    \fill[redlayer, opacity=0.2] 
        (0,0) -- (10,10) -- (10,11) -- (0,21) -- cycle;
    \draw[thick, redlayer] (0,0) -- (10,10);
    \draw[thick, redlayer] (0,21) -- (10,11);
    
    \node[font=\bfseries, redlayer] at (5,10) {$\gamma_1$};
    
    \draw[->, thin] (-1,0) -- (11,0) node[right] {$t$};
    \draw[->, thin] (0,-2) -- (0,26) node[above] {$h$};
    
\end{scope}

\begin{scope}[xshift=7cm, y={(0cm, 0.25cm)}]  
    \definecolor{bluelayer}{RGB}{80, 120, 200}
    \fill[bluelayer, opacity=0.2] 
        (0,14) -- (10,4) -- (10,25) -- (0,15) -- cycle;
    \draw[thick, bluelayer] (0,14) -- (10,4);
    \draw[thick, bluelayer] (0,15) -- (10,25);
    
    \node[font=\bfseries, bluelayer] at (5,15) {$\gamma_2$};
    
    \draw[->, thin] (-1,0) -- (11,0) node[right] {$t$};
    \draw[->, thin] (0,-2) -- (0,28) node[above] {$h$};
    
\end{scope}
\end{tikzpicture}
\end{minipage}
\end{minipage}

\vspace{1cm}
\begin{minipage}{\textwidth}
\begin{minipage}{0.3\textwidth}
\raggedleft\textbf{A vineyard with \quad \\two vines \quad }
\end{minipage}%
\begin{minipage}{0.7\textwidth}
\qquad
\begin{tikzpicture}[
    scale=0.8,
    x={(1cm, 0cm)},     
    y={(0cm, 0.25cm)},   
    z={(0.3cm, 0.5cm)}  
]

\fill[redlayer, opacity=0.2] 
    (0,0,0) -- (10,10,0) -- (10,11,0) -- (0,21,0) -- cycle;
\draw[thick, redlayer] (0,0,0) -- (10,10,0);
\draw[thick, redlayer] (0,21,0) -- (10,11,0);

\fill[bluelayer, opacity=0.2] 
    (0,14,1) -- (10,4,1) -- (10,25,1) -- (0,15,1) -- cycle;
\draw[thick, bluelayer] (0,14,1) -- (10,4,1);
\draw[thick, bluelayer] (0,15,1) -- (10,25,1);

\node[font=\bfseries, redlayer] at (5,10,0) {$\gamma_1$};
\node[font=\bfseries, bluelayer] at (5,15,1) {$\gamma_2$};

\draw[black, thick, dashed] 
    (0,0,-0.5) -- (0,26,-0.5) -- (0,26,2) -- (0,0,2) -- cycle;

\draw[black, thick, dashed] 
    (3,0,-0.5) -- (3,26,-0.5) -- (3,26,2) -- (3,0,2) -- cycle;

\draw[black, thick, dashed] 
    (5,0,-0.5) -- (5,26,-0.5) -- (5,26,2) -- (5,0,2) -- cycle;

\draw[black, thick, dashed] 
    (7,0,-0.5) -- (7,26,-0.5) -- (7,26,2) -- (7,0,2) -- cycle;

\draw[black, thick, dashed] 
    (10,0,-0.5) -- (10,26,-0.5) -- (10,26,2) -- (10,0,2) -- cycle;

\draw[redlayer, thick] (0,0,0) -- (0,21,0);
\filldraw[redlayer] (0,0,0) circle (3pt);
\draw[redlayer, thick] (0,21,0) circle (3pt);
\draw[bluelayer, thick] (0,14,1) -- (0,15,1);
\filldraw[bluelayer] (0,14,1) circle (3pt);
\draw[bluelayer, thick] (0,15,1) circle (3pt);

\draw[redlayer, thick] (3,3,0) -- (3,18,0);
\filldraw[redlayer] (3,3,0) circle (3pt);
\draw[redlayer, thick] (3,18,0) circle (3pt);
\draw[bluelayer, thick] (3,11,1) -- (3,18,1);
\filldraw[bluelayer] (3,11,1) circle (3pt);
\draw[bluelayer, thick] (3,18,1) circle (3pt);

\draw[redlayer, thick] (5,5,0) -- (5,16,0);
\filldraw[redlayer] (5,5,0) circle (3pt);
\draw[redlayer, thick] (5,16,0) circle (3pt);
\draw[bluelayer, thick] (5,9,1) -- (5,20,1);
\filldraw[bluelayer] (5,9,1) circle (3pt);
\draw[bluelayer, thick] (5,20,1) circle (3pt);

\draw[redlayer, thick] (7,7,0) -- (7,14,0);
\filldraw[redlayer] (7,7,0) circle (3pt);
\draw[redlayer, thick] (7,14,0) circle (3pt);
\draw[bluelayer, thick] (7,7,1) -- (7,22,1);
\filldraw[bluelayer] (7,7,1) circle (3pt);
\draw[bluelayer, thick] (7,22,1) circle (3pt);

\draw[redlayer, thick] (10,10,0) -- (10,11,0);
\filldraw[redlayer] (10,10,0) circle (3pt);
\draw[redlayer, thick] (10,11,0) circle (3pt);
\draw[bluelayer, thick] (10,4,1) -- (10,25,1);
\filldraw[bluelayer] (10,4,1) circle (3pt);
\draw[bluelayer, thick] (10,25,1) circle (3pt);

\draw[black, dashed] (0,0,-0.5) -- (0,0,2);
\draw[black, dashed] (0,14,-0.5) -- (0,14,2);
\draw[black, dashed] (0,15,-0.5) -- (0,15,2);
\draw[black, dashed] (0,21,-0.5) -- (0,21,2);

\draw[black, dashed] (3,3,-0.5) -- (3,3,2);
\draw[black, dashed] (3,11,-0.5) -- (3,11,2);
\draw[black, dashed] (3,18,-0.5) -- (3,18,2);

\draw[black, dashed] (5,5,-0.5) -- (5,5,2);
\draw[black, dashed] (5,9,-0.5) -- (5,9,2);
\draw[black, dashed] (5,16,-0.5) -- (5,16,2);
\draw[black, dashed] (5,20,-0.5) -- (5,20,2);

\draw[black, dashed] (7,7,-0.5) -- (7,7,2);
\draw[black, dashed] (7,14,-0.5) -- (7,14,2);
\draw[black, dashed] (7,22,-0.5) -- (7,22,2);

\draw[black, dashed] (10,4,-0.5) -- (10,4,2);
\draw[black, dashed] (10,10,-0.5) -- (10,10,2);
\draw[black, dashed] (10,11,-0.5) -- (10,11,2);
\draw[black, dashed] (10,25,-0.5) -- (10,25,2);

\draw[->, thin] (0,0,0) -- (11,0,0) node[right] {$t$};
\draw[->, thin] (0,0,0) -- (0,26,0) node[above] {$h$};
\end{tikzpicture}
\end{minipage}
\end{minipage}

\vspace{1cm}
\begin{minipage}{\textwidth}
\begin{minipage}{0.3\textwidth}
\raggedleft\textbf{The persistence\\modules at\\specific times}
\end{minipage}%
\begin{minipage}{0.7\textwidth}
\qquad
\begin{tikzpicture}[scale=0.35,
    x={(1cm, 0cm)},     
    y={(0cm, 0.5cm)}, ]  
\definecolor{redlayer}{RGB}{200, 80, 80}
\definecolor{bluelayer}{RGB}{80, 120, 200}

\begin{scope}[xshift=-12cm]
    \node[font=\bfseries] at (1.25,23) {$t=0$};
    
    \draw[redlayer, thick] (0,0) -- (0,21);
    \filldraw[redlayer] (0,0) circle (5pt);
    \draw[redlayer, thick] (0,21) circle (5pt);
    
    \draw[bluelayer, thick] (1,14) -- (1,15);
    \filldraw[bluelayer] (1,14) circle (5pt);
    \draw[bluelayer, thick] (1,15) circle (5pt);
    
    \draw[black, dashed] (0,0) -- (1,0);
    \draw[black, dashed] (0,14) -- (1,14);
    \draw[black, dashed] (0,15) -- (1,15);
    \draw[black, dashed] (0,21) -- (1,21);
\end{scope}

\begin{scope}[xshift=-6cm]
    \node[font=\bfseries] at (1.25,23) {$t=3$};
    
    \draw[redlayer, thick] (0,3) -- (0,18);
    \filldraw[redlayer] (0,3) circle (5pt);
    \draw[redlayer, thick] (0,18) circle (5pt);
    
    \draw[bluelayer, thick] (1,11) -- (1,18);
    \filldraw[bluelayer] (1,11) circle (5pt);
    \draw[bluelayer, thick] (1,18) circle (5pt);
    
    \draw[black, dashed] (0,3) -- (1,3);
    \draw[black, dashed] (0,11) -- (1,11);
    \draw[black, dashed] (0,18) -- (1,18);
\end{scope}

\begin{scope}
    \node[font=\bfseries] at (1.25,23) {$t=5$};
    
    \draw[redlayer, thick] (0,5) -- (0,16);
    \filldraw[redlayer] (0,5) circle (5pt);
    \draw[redlayer, thick] (0,16) circle (5pt);
    
    \draw[bluelayer, thick] (1,9) -- (1,20);
    \filldraw[bluelayer] (1,9) circle (5pt);
    \draw[bluelayer, thick] (1,20) circle (5pt);
    
    \draw[black, dashed] (0,5) -- (1,5);
    \draw[black, dashed] (0,9) -- (1,9);
    \draw[black, dashed] (0,16) -- (1,16);
    \draw[black, dashed] (0,20) -- (1,20);
\end{scope}

\begin{scope}[xshift=6cm]
    \node[font=\bfseries] at (1.25,23) {$t=7$};
    
    \draw[redlayer, thick] (0,7) -- (0,14);
    \filldraw[redlayer] (0,7) circle (5pt);
    \draw[redlayer, thick] (0,14) circle (5pt);
    
    \draw[bluelayer, thick] (1,7) -- (1,22);
    \filldraw[bluelayer] (1,7) circle (5pt);
    \draw[bluelayer, thick] (1,22) circle (5pt);
    
    \draw[black, dashed] (0,7) -- (1,7);
    \draw[black, dashed] (0,14) -- (1,14);
    \draw[black, dashed] (0,22) -- (1,22);
\end{scope}

\begin{scope}[xshift=12cm]
    \node[font=\bfseries] at (1.25,23) {$t=10$};
    
    \draw[redlayer, thick] (0,10) -- (0,11);
    \filldraw[redlayer] (0,10) circle (5pt);
    \draw[redlayer, thick] (0,11) circle (5pt);
    
    \draw[bluelayer, thick] (1,4) -- (1,25);
    \filldraw[bluelayer] (1,4) circle (5pt);
    \draw[bluelayer, thick] (1,25) circle (5pt);
    
    \draw[black, dashed] (0,4) -- (1,4);
    \draw[black, dashed] (0,10) -- (1,10);
    \draw[black, dashed] (0,11) -- (1,11);
    \draw[black, dashed] (0,25) -- (1,25);
\end{scope}
\end{tikzpicture}
\end{minipage}
\end{minipage}

\vspace{1cm}
\begin{minipage}{\textwidth}
\begin{minipage}{0.3\textwidth}
\raggedleft\textbf{The partial order of\\the intervals at\\different times}
\end{minipage}%
\begin{minipage}{0.7\textwidth}
\qquad \qquad 
\begin{tabular}{|c|c|c|c|c|}
\hline
$t < 3$ & $t = 3$ & $3 < t < 7$ & $t = 7$ & $t > 7$ \\
\hline
- & $\gamma_1 < \gamma_2$ & $\gamma_1 < \gamma_2$ & $\gamma_1 < \gamma_2$ & - \\
\hline
\end{tabular}
\end{minipage}
\end{minipage}
\caption{Forwards and Backwards incompatibility for a vineyard with two vines. It is backwards incompatible at $t=3$ and forwards incompatible at $t=7$.}
\end{figure}

Note that our genericity assumptions imply that at most one change of partial order can occur going forwards and only one change can occur going backwards so when a time is forwards or backwards incompatible there is a unique pair $(\gamma_k, \gamma_l)$ that causes this incompatibility.

We will want to use our critical times as time stamps with respect to which we can define what it means for a family of bases to be forwards or backwards simplified.

 \begin{definition}
Let $\Vine{V}=(\Mod{V}_t, \{\alpha^{s\to t}\}, \{\beta^{t\to s}\})$ be a vineyard module over $[T_0, T_1]$. For each $s\in [T_0, T_1)$ let 
$$c^{f}(s):=\inf\left\{\{ t>s| t \text{ is a critical time}\} \cup \{T_1\}\right\}$$ (the earliest critical time strictly after $s$). We say $\{B_t\}$ is a \textbf{forwards simplified} family of bases if 
\begin{itemize}
\item if $s$ is forwards compatible then $\Mat_{B_s}^{B_t}(\alpha^{s\to t})=\pi_{(\Mod{V}_s, \Mod{V}_t, |s-t|)}(\Id)$ for all $t\in(s,  c^f(s)]$
\item if $s$ is forwards incompatible by vines $(\gamma_k, \gamma_l)$  there is some $\lambda \in \mathbb{F}$ such that
$\Mat_{B_s}^{B_t}(\alpha^{s\to t})=\pi_{(\Mod{V}_s, \Mod{V}_t, |s-t|)}(e^{\lambda}_{kl} )$ for all $t\in (s,c^f(s)]$. 
\end{itemize}
\end{definition}

Here $e^{\lambda}_{kl}$ is the elementary matrix with $\lambda$ in the $(k,l)$ position. Note that if $\lambda=0$ then $e^{\lambda}_{kl}=\Id$.

 \begin{definition}
Let $\Vine{V}=(\Mod{V}_t, \{\alpha^{s\to t}\}, \{\beta^{t\to s}\})$ be a vineyard module over $[T_0, T_1]$. For each $s\in (T_0, T_1]$ let 
$$c^{b}(t):=\sup\left\{\{ s<t| s \text{ is a critical time}\} \cup \{T_0\}\right\}$$ 
(the latest critical time strictly before $t$). We say $\{B_s\}$ is a \textbf{backwards simplified} family of bases if 
\begin{itemize}
\item if $t$ is backwards compatible then $\Mat_{B_t}^{B_s}(\beta^{t\to s})=\pi_{(\Mod{V}_t, \Mod{V}_s, |s-t|)}(\Id)$ for all $s\in [c^b(t), t)$
\item if $t$ is backwards incompatible by vines $(\gamma_k, \gamma_l)$  there is some $\lambda \in \mathbb{F}$ such that
$\Mat_{B_t}^{B_s}(\beta^{t\to s})=\pi_{(\Mod{V}_t, \Mod{V}_s, |s-t|)}(e^{\lambda}_{kl} )$ for all $s\in [c^b(t), t)$.
\end{itemize}
\end{definition}

Despite the definition of being forwards or backwards simplified involving morphisms between potentially distant time values, it is inherently a local phenomenon glued together. 

\begin{lemma}\label{lem:union}
Let $\Vine{V}=(\Mod{V}_t, \{\alpha^{s\to t}\}, \{\beta^{t\to s}\})$ be a vineyard module over $[T_0, T_1]$ and $\{B_t\}_{t\in [T_0, T_1]}$ a family of bases for $\Vine{V}$. Let $S\in (T_0, T_1)$ and consider the restrictions $\Vine{V}|_{[T_0, S]}$ and $\Vine{V}|_{[S, T_1]}$ to time intervals $[T_0, S]$ and $[S, T_1]$. If $\{B_t\}_{t\in [T_0, S]}$ and  $\{B_t\}_{t\in [S, T_1]}$ are both forwards simplified families of bases for $\Vine{V}|_{[T_0, S]}$ and $\Vine{V}|_{[S, T_1]}$ respectively, then $\{B_t\}_{t\in [T_0, T_1]}$ is forwards simplified for $\Vine{V}$.
\end{lemma}

Given a forwards (respectively backwards) simplified bases we know that the corresponding vines and matrices representation is of a particularly nice form.  Locally it always is a projection of the identity matrix except at discrete times of incompatibility. This discrete set of times is completely determined by the associated vineyard (independent of the algebra of the interleaving maps). Furthermore, even at the times of incompatibility the matrices are still elementary with only one non-diagonal entry of interest. Together these mean that the entire family of matrices in the vines and matrices representation can be summarised by a single vector. This leads to the definition of the vineyard and vector representation.

\begin{definition}
Let $\Vine{V}=(\Mod{V}_t, \alpha^{s\to t}, \beta^{t\to s})$ be a vineyard module over time interval $[T_0, T_1]$ with forwards incompatible times $(s_1, s_2, \ldots s_K)$. A \textbf{forwards vineyard and vector} representation of $\Vine{V}=(\Mod{V}_t, \alpha^{s\to t}, \beta^{t\to s})$ consists of its associated vineyard $\Vineyard{V}$ and a vector $(\lambda_1, \ldots \lambda_K)$
such that there exists a forwards simplified family of bases $\{B_t\}$ such that
$\Mat_{B_{s_i}}^{B_t}(\alpha^{s_i\to t})=\pi_{(\Mod{V}_{s_i}, \Mod{V}_t, |t-s_i|)}(e^{\lambda_i}_{k_il_i})$  for all $t>s_i$ sufficiently close, where $s_i$ is forwards incompatible by vines $(\gamma_{k_i}, \gamma_{l_i})$.  
\end{definition}

Symmetrically we define the backwards vineyard and vector representation.

\begin{definition}
Let $\Vine{V}=(\Mod{V}_t, \alpha^{s\to t}, \beta^{t\to s})$ be a vineyard module over time interval $[T_0, T_1]$ with backwards incompatible times $(t_1, t_2, \ldots t_K)$.
A \textbf{backwards vineyard and vector} representation of $\Vine{V}=(\Mod{V}_t, \alpha^{s\to t}, \beta^{t\to s})$ consists of its associated vineyard $\Vineyard{V}$ and a vector $(\lambda_1, \ldots \lambda_K)$
such that there exists a backwards simplified family of bases $\{B_s\}$ such that
$\Mat_{B_{t_i}}^{B_s}(\beta^{t_i\to s})=\pi_{(\Mod{V}_{t_i}, \Mod{V}_s, |s-t_i|)}(e^{\lambda_i}_{k_il_i})$  for all $s<t_i$ sufficiently close, where $t_i$ is backwards incompatible by vines $(\gamma_{k_i}, \gamma_{l_i})$. \end{definition}

Here $\{\gamma_1, \gamma_2, \ldots \gamma_N\}$ is an indexing of the vines. The choice of indexing does not affect the vineyard and vector representation. 

We can canonically construct vines and matrices representations from vineyard and vector representations. This implies, with Proposition \ref{prop:isom}, that if two forwards vineyard and vector representations agree (or similarly two backwards vineyard and vector representations agree) then the original vineyard modules must be isomorphic. What is unclear is when different vineyard and vector representations correspond to isomorphic vineyard modules.

It will become useful to use the inherent symmetry of forwards vs backwards simplification. To this end we want to define a construction of a new vineyard module reversing its direction.

\begin{definition}
Let $\Vine{V}=(\Mod{V}_t, \alpha^{s\to t}, \beta^{t\to s})$ be a vineyard module over time interval $[T_0, T_1]$.
Construct a new vineyard module $\Vine{V}^{op}$ over time parameters $[-T_1, -T_0]$ with $\Mod{V}^{op}_{-t}=\Mod{V}_t$, $\alpha_{op}^{-t\to -s}=\beta^{t\to s}$ and $\beta_{op}^{-s\to -t}=\alpha^{s\to t}$. 
\end{definition}

Note that $\Vine{V}$ is forwards simplified if and only if $\Vine{V}^{op}$ is backwards simplified and $\Vine{V}$ is backwards simplified if and only if $\Vine{V}^{op}$ is forwards simplified. 

One useful property of vineyard and vector representations is how straightforward it is to construct vineyard modules with a family of bases that result in that vineyard and vector representation. 

\begin{proposition}\label{prop:realise}
Let $\Vineyard{V}$ be a vineyard over $[T_0, T_1]$ with forwards incompatible times $(s_1, s_2, \ldots  s_K)$. Then for all vectors $(\lambda_1, \lambda_2, \ldots, \lambda_K)$ there exists a vineyard module $\Vine{V}$, with associated vineyard $\Vineyard{V}$, and a family of forwards simplified bases $\{B_t\}$, such that $(\lambda_1, \lambda_2, \ldots, \lambda_K)$ is the vector in the corresponding vineyard and vector representation.

Let $\Vineyard{V}$ be a vineyard over $[T_0, T_1]$ with backwards incompatible times $(t_1, t_2, \ldots  t_L)$. Then for all vectors $(\lambda_1, \lambda_2, \ldots, \lambda_L)$ there exists a vineyard module $\Vine{V}$, with associated vineyard $\Vineyard{V}$, and a family of backwards simplified bases $\{B_t\}$, such that $(\lambda_1, \lambda_2, \ldots, \lambda_L)$ is the vector in the corresponding vineyard and vector representation.
\end{proposition} 

\begin{proof}
Through $\Vine{V}^{op}$ (using the backwards incompatible times in $\Vine{V}$ correspond to forwards incompatible times in $\Vine{V}^{op}$) it is sufficient to prove the statement about realising a forwards vineyard and vector representation.

For each $t\in [T_0, T_1]$ construct a persistence module with interval decomposition matching $\Vineyard{V}_t$ and fix a choice of basis $B_t$. 

Subdivide $[T_0, T_1]$ into finitely many subintervals $[S_i, S_{i+1}]$ such that $\{s_j\}\subset \{S_i\}$. By Lemma \ref{lem:union} it is enough to show that the proposition holds for each of the intervals $[S_i, S_{i+1}]$. 

For intervals $[S_i, S_{i+1}]$ with $S_i$ forwards compatible we can set the interleaving maps to be 
$\alpha^{s\to t}:=\Mor_{(B_s, B_t, |s-t|)}(\pi_{(\Mod{V}_s, \Mod{V}_t, |s-t|)}(\Id))$ and $\beta^{t\to s}:=\Mor_{(B_t, B_s, |s-t|)}(\pi_{(\Mod{V}_t, \Mod{V}_s, |s-t|)}(\Id))$ for all $S_i\leq s<t \leq S_{i+1}$.

For intervals $[S_i, S_{i+1}]$ with $S_i$ forwards incompatible ($S_i=s_{j(i)}$) with respect to vines $(\gamma_k, \gamma_l)$ we can set the interleaving maps to be 
$\alpha^{S_i\to t}=\Mor_{(B_{S_i}, B_t, |S_i-t|)}(\pi_{(\Mod{V}_{S_i}, \Mod{V}_t, |S_i-t|)}(e^{\lambda_{j(i)}}_{kl}))$ and $\beta^{t\to S_i}=\Mor_{(B_t, B_{S_i}, |S_i-t|)}(\pi_{(\Mod{V}_t, \Mod{V}_{S_i}, |S_i-t|)}(e^{-\lambda_{j(i)}}_{kl}))$ for all $S_i<t \leq S_{i+1}$, and
$\alpha^{s\to t}=\Mor_{(B_s, B_t, |s-t|)}(\pi_{(\Mod{V}_s, \Mod{V}_t, |s-t|)}(\Id))$ and $\beta^{t\to s}=\Mor_{(B_t, B_s, |s-t|)}(\pi_{(\Mod{V}_t, \Mod{V}_s, |s-t|)}(\Id))$ for all $S_i< s<t \leq S_{i+1}$.
It is routine to check these commute appropriately over $[S_i, S_{i+1}]$.

Gluing the vineyard modules over the different $[S_i, S_{i+1}]$, setting the interleaving maps between times in different time intervals via composition, constructs the full vineyard module.
\end{proof}

\section{Segments of Vineyard Modules}\label{sec:segments}

In order to find a forwards (or backwards) simplified family of bases we will partition the vineyard modules to restrictions over smaller intervals. We will make them small enough that we can control both changes in the partial order of the intervals and the gaps between critical heights. We call such small intervals segments and we will characterise them into different types. Note that the definition of segment for a vineyard module is only dependent on the associated vineyard $\Vineyard{V}$.

\begin{definition}
A \textbf{segment} is the restriction of a vineyard to a finite time interval $[S, T]$ such that there are no critical times in $(S, T)$ and exactly one of the three conditions hold:
\begin{enumerate}
\item[(Type 1)] $S$ is not right critical, $T$ is not left critical, and the smallest distance between the values within $\Crit^+(S)$, or separately within $\Crit^-(T)$, is strictly bounded below by $4|S-T|$, 
\item[(Type 2)] $S$ is not right critical, and $T$ is left critical and the smallest distance between the distinct values within $\Crit^-(T)$ is strictly bounded below by $4|S-T|$, 
\item[(Type 3)]  $T$ is not left critical, and $S$ is right critical and the smallest distance between the distinct values within $\Crit^+(S)$ is strictly bounded below by $4|S-T|$.
\end{enumerate}
We say that $\Vine{V}|_{[S,T]}$ is a segment if $\Vineyard{V}|_{[S,T]}$ is a segment for its associated vineyard $\Vineyard{V}$. 
\end{definition}

Using the $4|S-T|$ gap between distinct critical values, and the Lipschitz assumption on the birth and death values, we can prove a number of useful lemmas about subintervals of segments being segments.

\begin{lemma}\label{lem:subinterval}
Let $\Vineyard{V}$ be a vineyard and $S\leq s\leq t\leq T$. 
\begin{enumerate}
\item If $\Vineyard{V}|_{[S,T]}$ is a segment of Type 1 then $\Vineyard{V}|_{[s,t]}$ is a segment of Type 1
\item If $\Vineyard{V}|_{[S,T]}$ is a segment of Type 2 then  $\Vineyard{V}|_{[s,T]}$ is a segment of Type 2
\item If $\Vineyard{V}|_{[S,T]}$ is a segment of Type 3 then $\Vineyard{V}|_{[S,t]}$ is a segment of Type 3.
\end{enumerate}
\end{lemma}

\begin{definition}
Given a vineyard module $\Vine{V}$, or its associated vineyard $\Vineyard{V}$, defined over $[T_0, T_1]$. We say $\{S_i\}_{i=0}^N$ is a \textbf{partition into segments} if $T_0=S_0<S_1\ldots < S_N=T_1$ and $\Vineyard{V}|_{[S_i, S_{i+1}]}$ is a segment for all $i$. 
\end{definition}

To this end we must first check that such a partition into segments exists. 

\begin{lemma}\label{lem:existence}
Given any vineyard $\Vineyard{V}$ there exists a partition of $\Vineyard{V}$ into finitely many segments.
\end{lemma}

Later when we compare our forward simplified families of bases constructed via different partitions into segments it will be useful to restrict to when one is a refinement of the other. We say that $\{\hat{S}_i\}$ is a \textbf{refinement} of $\{S_j\}$ if $\{S_j\} \subset \{\hat{S}_i\}$. To restrict to this case we need to justify there always exists a common refinement for any pair of partitions into segments.

\begin{lemma}\label{lem:refinement}
Let $\Vineyard{V}$ be a vineyard defined over $[T_0, T_1]$. Let $\{S_i^1\}$ and $\{S_j^2\}$ be partitions of $\Vineyard{V}$ into segments. Then there exists a partition $\{\hat{S}_k\}$ of $\Vineyard{V}$ into segments such that $\{S_i^1\}\subset \{\hat{S}_k\}$ and $\{S_j^2\}\subset \{\hat{S}_k\}$.
\end{lemma}

As segments of Type $1$ are the easiest to work with  it will be useful to be able to cover all the non-critical times in segments of Type $2$ and Type $3$ by an infinite sequence of segments of Type $1$.

\begin{lemma}\label{lem:sequence}
Let $\Vineyard{V}|_{[S,T]}$ be a segment of Type $2$ ($T$ critical). 
Then there exists an increasing sequence $\{s_n\}$ with $s_0=S$ and $\lim_{n \to \infty} s_n=T$ such that $\Vineyard{V}|_{[s_n, s_{n+1}]}$ and $\Vineyard{V}|_{[s_{n-1}, s_{n+1}]}$ are both segments of Type $1$ for all $n\geq 1$.
\end{lemma}


It is efficient to prove results for either the forwards or the backwards interleaving maps and then use this to infer the other direction.  Intuitively we expect the matrices of the different directions to effectively be inverses of each other. This is true up to some wiggle room coming from the composition not being the identity but instead the transition morphism from a persistence module to a shifted version of itself, with a shift dependent on the difference in time values. The proofs rely on the gap in the distinct critical height values. 

\begin{lemma}\label{lem:IdIFF}
Let $\Vine{V}=(\Mod{V}_t, \{\alpha^{s\to t}\}, \{\beta^{t\to s}\})$ be a vineyard module, and $\Vine{V}|_{[S,T]}$ a segment. $$\Mat_{B_S}^{B_T}(\alpha^{S\to T})=\pi_{(\Mod{V}_S, \Mod{V}_T, |S-T|)}(\Id) \quad  \Longleftrightarrow \quad \Mat_{B_T}^{B_S}(\beta^{T\to S})=\pi_{(\Mod{V}_T, \Mod{V}_S, |S-T|)}(\Id).$$ 

Moreover, if segment $\Vine{V}|_{[S,T]}$ is of Type $1$ then a family of bases $\{B_t\}$ for  $\Vine{V}|_{[S,T]}$ is forwards simplified if and only if it is backwards simplified. 
\end{lemma}

\begin{lemma}\label{lem:lambdaIFF}
Let $\Vine{V}=(\Mod{V}_t, \{\alpha^{s\to t}\}, \{\beta^{t\to s}\})$ be a vineyard module, and $\Vine{V}|_{[S,T]}$ a segment, with either $S$ forwards incompatible or $T$ backwards incompatible with respect to vines $(\gamma_k, \gamma_l)$. Then $$\Mat_{B_S}^{B_T}(\alpha^{S\to T})=\pi_{(\Mod{V}_S, \Mod{V}_T, |S-T|)}(e^{\lambda}_{kl}) \quad  \Longleftrightarrow \quad \Mat_{B_T}^{B_S}(\beta^{T\to S})=\pi_{(\Mod{V}_T, \Mod{V}_S, |S-T|)}(e^{-\lambda}_{kl}).$$ 
\end{lemma}

\begin{lemma}\label{lem:aliveVines}
Let $\Vine{V}=(\Mod{V}_t, \{\alpha^{s\to t}\}, \{\beta^{t\to s}\})$ be a vineyard module over a segment $\Vine{V}|_{[S,T]}$ and $B_S$ and $B_T$ bases of $\Mod{V}_S$ and $\Mod{V}_T$. If $\birth(\gamma_j(S))<\death(\gamma_j(S))$ and $\birth(\gamma_j(T))<\death(\gamma_j(T))$ then $$\Mat_{B_S}^{B_T}(\alpha^{S\to T})(j,j)\Mat_{B_T}^{B_S}(\beta^{T\to S})(j,j)=1.$$
\end{lemma}

We also get equality for certain matrices of morphisms within a triangle of compositions.

\begin{lemma}\label{lem:triangle}
Let $\Vine{V}=(\Mod{V}_t, \{\alpha^{s\to t}\}, \{\beta^{t\to s}\})$ be a vineyard module and $\Vine{V}|_{[S,T]}$ a segment.  Let $S<  s < T$ and $B_S, B_s, B_T$ be bases for $\Mod{V}_S, \Mod{V}_s$ and $\Mod{V}_T$ respectively.  
\begin{itemize}
\item[(i)] Suppose $\Vine{V}|_{[S,T]}$ is a segment of Type 1 or Type 2 ($T$ critical). \\
If $\Mat_{B_S}^{B_s}(\alpha^{S\to s})=\pi_{(\Mod{V}_S, \Mod{V}_s, |s-S|)}(\Id)$
then $\Mat_{B_s}^{B_T}(\alpha^{s\to T})= \Mat_{B_S}^{B_T}(\alpha^{S\to T}).$\\
 If $\Mat_{B_s}^{B_S}(\beta^{s\to S})=\pi_{(\Mod{V}_s, \Mod{V}_S, |s-S|)}(\Id)$
then $\Mat_{B_T}^{B_s}(\beta^{T\to s})= \Mat_{B_T}^{B_S}(\beta^{T\to S}).$

\item[(ii)]
Suppose $\Vine{V}|_{[S,T]}$ is a segment of Type 1 or Type 3  ($S$ critical).\\
 If 
$\Mat_{B_s}^{B_T}(\alpha^{s\to T})=\pi_{(\Mod{V}_s, \Mod{V}_T, |s-T|)}(\Id)$
then $\Mat_{B_S}^{B_s}(\alpha^{S\to s})= \Mat_{B_S}^{B_T}(\alpha^{S\to T}).$\\
 If 
$\Mat_{B_T}^{B_s}(\beta^{T\to s})=\pi_{(\Mod{V}_T, \Mod{V}_s, |s-T|)}(\Id)$
then $\Mat_{B_s}^{B_S}(\beta^{s\to S})= \Mat_{B_T}^{B_S}(\beta^{T\to S}).$
\end{itemize}
\end{lemma}

We finally are ready to prove the construction of forward simplifying a family of bases within a segment $\Vine{V}|_{[S,T]}$. There are two main cases we need to consider -- whether $S$ is forwards compatible or not. When $S$ is forwards compatible we want to effectively chose the basis at $t\in [S,T]$ to be the image of the original basis at $S$ under the morphism $\alpha^{S\to t}$. We need to be careful as the image persistence module can have later birth times than the codomain persistent module. One complication is when $S$ is critical, but forwards compatible, as image of the basis elements of $\Mod{V}_S$ under $\alpha^{S\to t}$ may provide an underdetermined choice and we need to make sure we choose the family of bases in a consistent way.

\begin{theorem}\label{thm:S_compatible}
Let $\Vine{V}=(\Mod{V}_t, \{\alpha^{s\to t}\}, \{\beta^{t\to s}\})$ be a vineyard module and $\Vine{V}|_{[S,T]}$ a segment with $S$ forwards compatible and $\{B_t^{old}\}$ an initial choice of basis of each $\Mod{V}_t$. Let $P_T$ be the matrix with
$$P_T(i,j)=\begin{cases} 
&1 \text{ when }i=j \text{ and } S \notin \supp(\gamma_i)\\
 &\Mat_{B_S^{old}}^{B_T^{old}}(\alpha^{S\to T})(i,j) \text{ otherwise}
 \end{cases}$$
 Then $P_T\in \BTM(\Mod{V}_T)$ and there exists a new choice of basis $\{B_t^{new}\}$ which is forwards simplified such that 
 $B_S^{new}=B_S^{old}$ and $B_T^{new}= P_T(B_T^{old})$.
\end{theorem}

\begin{proof}
\underline{Showing $P_T \in \BTM(\Mod{V}_T)$:} 

If $S\notin \supp(\gamma_j)$ then by construction $P(j,j)=1$. 
If $T\in \supp(\gamma_j)$ and $S\in \supp(\gamma_j)$ then by Lemma \ref{lem:aliveVines} we have $\Mat_{B_S^{old}}^{B_T^{old}}(\alpha^{S\to T})(j,j)\neq 0$. 

Let $i\neq j$ and $P(i,j)\neq 0$. This implies $\Mat_{B_S^{old}}^{B_T^{old}}(\alpha^{S\to T})(i,j)\neq 0$ and hence $T\in \supp(\gamma_i)\cap \supp(\gamma_j)$ and $\birth(\gamma_i(T))\leq \birth(\gamma_j(S))+|S-T|$ and $\death(\gamma_i(T))\leq \death(\gamma_j(S))+|S-T|$.  
Recall that as $\birth(\gamma_j(t))$ and $\death(\gamma_j(t))$ are both Lipschitz in $t$ this implies $\birth(\gamma_i(T))\leq \birth(\gamma_j(T))+2|S-T|$ and $\birth(\gamma_i(S))\leq \birth(\gamma_j(S))+2|S-T|$. 
If $\Vine{V}|_{[S,T]}$ is a segment of Type $1$ or Type $2$ then $\birth(\gamma_i(T))\leq \birth(\gamma_j(T))+2|S-T|$ automatically strengthens to 
$\birth(\gamma_i(T))\leq \birth(\gamma_j(T))$.
If $\Vine{V}|_{[S,T]}$ is a segment of Type $3$ then  $\birth(\gamma_i(S))\leq \birth(\gamma_j(S))+2|S-T|$ strengthens to $\birth(\gamma_i(S))\leq \birth(\gamma_j(S))$. 
We then can use the forwards compatibility of $S$ to say $\birth(\gamma_i(T))\leq \birth(\gamma_j(T))$ (as no critical time appears in $(S,T]$). The same argument works to show $\death(\gamma_i(T))\leq \death(\gamma_j(T))$. We have thus shown $P$ satisfies all the requirements to be in $\BTM(\Mod{V}_T)$. Set $B_T^{new}=P_T(B_T^{old})$.

Let $J=\{j\mid S \in \supp(\gamma_j)\}$ and $\Id_J$ the diagonal matrix with $\Id_J(j,j)=1$ if $j\in J$ and $0$ otherwise.
By construction $\Mat_{B_S^{old}}^{B_T^{old}}(\alpha^{S\to T})=P_T\Id_J$. By Lemma \ref{lem:transformed_morphism} we know $$\Mat_{B_S^{old}}^{B_T^{new}}(\alpha^{S\to T})=\pi_{(\Mod{V}_S, \Mod{V}_T, |S-T|)} (\Id_J)=\pi_{(\Mod{V}_S, \Mod{V}_T, |S-T|)} (\Id).$$

At this point we need to split into the cases of segments of Type $1$, Type $2$ ($T$ critical) and Type $3$ ($S$ critical).

\underline{$\Vine{V}|_{[S,T]}$ is a segment of Type $1$:}

First observe that for all $S<t\leq T$ we have $\Vine{V}|_{[S,t]}$ is also a segment of Type $1$ (Lemma \ref{lem:subinterval}). If we set $P_t$ as the corresponding basis transformation matrix for interval $[S,t]$ and $B_t^{new}=P_t(B_t^{old})$, then we have $\Mat_{B_S^{old}}^{B_t^{new}}(\alpha^{S\to t})=\pi_{(\Mod{V}_S, \Mod{V}_t, |S-t|)} (\Id_J)=\pi_{(\Mod{V}_S, \Mod{V}_t, |S-t|)} (\Id).$
For $S<s<t\leq T$ we can  apply Lemma \ref{lem:triangle} (to segment $\Vine{V}|_{[S,t]}$) to say $\Mat_{B^{new}_s}^{B^{new}_t}(\alpha^{s\to t})=\pi_{(\Mod{V}_s, \Mod{V}_t, |s-t|)} (\Id).$

\underline{$\Vine{V}|_{[S,T]}$ is a segment of Type $2$:} 

By Lemma \ref{lem:sequence} there exists an increasing sequence $\{s_n\}$ with $s_0=S$ and $\lim_{n \to \infty} s_n=T$ such that $[s_n, s_{n+1}]$ and $[s_{n-1}, s_{n+1}]$ are both segments of Type 1 for all $n\geq 1$. Define $B^{new}_{t}$ for $t\in (s_n , s_{n+1}]$ inductively over $n$. From the Type $1$ case covered above, we know the family of bases $B_{t}^{new}=\Mat_{B_{s_n}^{new}}^{B_t^{old}}(\alpha^{s_n\to t})(B_t^{old})$ for $t\in (s_n , s_{n+1}]$ is forwards simplified.

We want to show that $\Mat_{B^{new}_s}^{B^{new}_t}(\alpha^{s\to t})=\pi_{(\Mod{V}_s, \Mod{V}_t, |s-t|)} (\Id)$ for all $S\leq s<t<T$. 
Fix $S\leq s<t<T$. There exists $m$ and $n$ such that $s_m\leq s<t\leq s_{n+1}.$

As $\{B_t^{new}\}$ is forwards simplified over each of the subinterval $[s_i, s_{i+1}]$ for $m\leq i\leq n$ we can inductively infer that  $\{B_t^{new}\}$ is forwards simplified over $[s_m, s_{n+1}]$, and hence  $\Mat_{B^{new}_s}^{B^{new}_t}(\alpha^{s\to t})=\pi_{(\Mod{V}_s, \Mod{V}_t, |s-t|)} (\Id)$. 

As $\Vine{V}|_{[S,T]}$ is a segment of Type $2$ and $\Mat_{B_S^{new}}^{B_s^{new}}(\alpha^{S\to s})=\pi_{(\Mod{V}_S, \Mod{V}_s, |S-s|)}(\Id)$ we can apply Lemma \ref{lem:triangle} to say $$\Mat_{B_s^{new}}^{B_T^{new}}(\alpha^{s\to T})=\Mat_{B_S^{new}}^{B_T^{new}}(\alpha^{S\to T})=\pi_{(\Mod{V}_S, \Mod{V}_T, |S-T|)} (\Id)=\pi_{(\Mod{V}_s, \Mod{V}_T, |s-T|)} (\Id).$$

\underline{$\Vine{V}|_{[S,T]}$ is a segment of Type $3$:}

Consider the vineyard module $\Vine{V}^{op}$ over time parameters $[-T, -S]$, and set the family of bases $\hat{B}_{-t}^{old}=B_{t}^{old}$ for $t\in [S, T)$ but $\hat{B}_{-T}^{old}=B_{T}^{new}=P_T(B_T^{old})$.

Note that $\Vine{V}^{op}|_{[-T, - S]}$ is now a segment of Type $2$ for the vineyard associated with $\Vine{V}^{op}$. This means we can use the case of a segment of Type $2$ covered above to find a new basis $\hat{B}_{r}^{new}$ for $\Vine{V}^{op}$ which is forwards simplified. Let $\{r_n\}$ be the increasing sequence with $r_0=-T$ and $\lim_{n \to \infty} r_n=-S$ used in the construction.

Fix $s\in (S, T)$ so $-s\in (-T, -S)$. We have $-s\in [r_n, r_{n+1}]$ for some $n$. We have $\{\hat{B}_r\}_{r\in [-T,-S]}$ is a forwards simplified basis over $\Vine{V}^{op}|_{\cup_{i=0}^{n}[r_i, r_{i+1}]}$. As each $\Vine{V}^{op}|_{[r_i, r_{i+1}]}$ is a segment of Type $1$, we can combine Lemma \ref{lem:IdIFF} with Lemma \ref{lem:union} to conclude $\{\hat{B}_r^{new}\}$ is a backwards simplified basis over $\cup_{i=0}^{n}[r_i, r_{i+1}]$. Thus for all $S<s<t\leq T$ we have $\Mat_{\hat{B}_{-s}^{new}}^{\hat{B}_{-t}^{new}}(\beta_{op}^{-s\to-t})=\pi_{(\Mod{V}^{op}_{-s}, \Mod{V}^{op}_{-t}, |s-t|)}(\Id)$.

Translating back to $\Vine{V}$ we have a family of bases $\{B_t^{new}\}$ with $B_t^{new}=\hat{B}_{-t}^{new}$. For all $S<s<t\leq T$
$\Mat_{\hat{B}_{-s}^{new}}^{\hat{B}_{-t}^{new}}(\beta_{op}^{-s\to-t})=\pi_{(\Mod{V}^{op}_{-s}, \Mod{V}^{op}_{-t}, |s-t|)}(\Id)$ implies 
$\Mat_{B_{s}^{new}}^{B_{t}^{new}}(\alpha^{s\to t})=\pi_{(\Mod{V}_{s}, \Mod{V}_{t}, |s-t|)}(\Id)$.

For any $S<t<T$, observing that $\Vine{V}|_{[S,t]}$ is a segment of Type $3$, we can apply Lemma \ref{lem:triangle} to say $\Mat_{B^{new}_S}^{B^{new}_t}(\alpha^{S\to t})=\pi_{(\Mod{V}_S, \Mod{V}_t, |S-t|)} (\Id).$ 
\end{proof}

The final, and most interesting, simplification process over a segment $\Vine{V}|_{[S,T]}$ is when $\Vineyard{V}$ is forwards incompatible at $S$. Recall that  a $\Vineyard{V}$ is forwards incompatible at $S$  by vines $(\gamma_k, \gamma_l)$ if $S\in \supp(\gamma_k)\cap \supp(\gamma_l)$ and $\gamma_k(S)\leq \gamma_l(S)$ but $\gamma_k(t)\nleq \gamma_l(t)$ for all $t\in (S, S+\delta)$ for $\delta>0$ sufficiently small. Our assumption is that when $\Vineyard{V}$ is forwards incompatible at $S$ that there will be a unique pair of vines $(\gamma_k, \gamma_l)$ such that $\Vineyard{V}$ is forwards incompatible at $S$ by vines $(\gamma_k, \gamma_l)$. Geometrically there are two cases depending on whether there is a switch in order of the births or the deaths. The cases are:
\begin{itemize}
\item[(i)] a switch in death heights with  $\birth(\gamma_k(S))<\birth(\gamma_l(S)) < \death(\gamma_k(S))=\death(\gamma_l(S))$ but $\birth(\gamma_k(T))<\birth(\gamma_l(T)) <\death(\gamma_l(T))<\death(\gamma_k(T))$
\item[(ii)] a switch in birth heights with $\birth(\gamma_k(S))=\birth(\gamma_l(S))<\death(\gamma_k(S))<\death(\gamma_l(S))$ but $\birth(\gamma_l(T))<\birth(\gamma_k(T))<\death(\gamma_k(T))<\death(\gamma_l(T))$ 
\end{itemize}
which are illustrated in Figure \ref{fig:casesincompatible}.

\begin{figure}[h]
\centering
\begin{tikzpicture}[
    scale=0.32,
    x={(1cm,0cm)},
    y={(0cm,0.45cm)},
    font=\small
]
\definecolor{redlayer}{RGB}{200, 80, 80}   
\definecolor{bluelayer}{RGB}{80, 120, 200} 

\begin{scope}
  \fill[bluelayer,opacity=0.2] (0,6)--(10,9)--(10,14)--(0,18)--cycle;
  \fill[redlayer,opacity=0.2]  (0,2)--(10,5)--(10,22)--(0,18)--cycle;
  \draw[thick,redlayer]  (0,2)--(10,5);     
  \draw[thick,redlayer]  (0,18)--(10,22);   
  \draw[thick,bluelayer] (0,6)--(10,9);     
  \draw[thick,bluelayer] (0,18)--(10,14);   
  \filldraw[redlayer]   (0,2) circle (6pt);  \filldraw[redlayer]   (10,5) circle (6pt);
  \draw[redlayer,thick] (0,18) circle (6pt); \draw[redlayer,thick] (10,22) circle (6pt);
  \filldraw[bluelayer]  (0,6) circle (6pt);  \filldraw[bluelayer]  (10,9) circle (6pt);
  \draw[bluelayer,thick](0,18) circle (6pt); \draw[bluelayer,thick](10,14) circle (6pt);
  \node[font=\bfseries,redlayer]  at (6.7,18.5) {$\gamma_k$};
  \node[font=\bfseries,bluelayer] at (4.3,11)   {$\gamma_l$};
  \draw[dashed] (0,-2)--(0,26);
  \draw[dashed] (10,-2)--(10,26);
  \draw[->,thin] (-1,0)--(11.2,0) node[right]{$t$};
  \draw[->,thin] (0,-2)--(0,27)   node[above]{$h$};
  \node[below] at (0,-2.2) {$S$};
  \node[below] at (10,-2.2){$T$};
  \node at (5,-7) {\textbf{(i)} deaths switch order};
\end{scope}

\begin{scope}[shift={(17,0)}]
  \fill[bluelayer,opacity=0.2] (0,6)--(10,3)--(10,24)--(0,20)--cycle;
  \fill[redlayer,opacity=0.2]  (0,6)--(10,9)--(10,18)--(0,14)--cycle;
  \draw[thick,redlayer]  (0,6)--(10,9);     
  \draw[thick,redlayer]  (0,14)--(10,18);   
  \draw[thick,bluelayer] (0,6)--(10,3);     
  \draw[thick,bluelayer] (0,20)--(10,24);   
  \filldraw[redlayer]   (0,6) circle (6pt);  \filldraw[redlayer]   (10,9) circle (6pt);
  \draw[redlayer,thick] (0,14) circle (6pt); \draw[redlayer,thick] (10,18) circle (6pt);
  \filldraw[bluelayer]  (0,6) circle (6pt);  \filldraw[bluelayer]  (10,3) circle (6pt);
  \draw[bluelayer,thick](0,20) circle (6pt); \draw[bluelayer,thick](10,24) circle (6pt);
  \node[font=\bfseries,redlayer]  at (4.6,13)   {$\gamma_k$};
  \node[font=\bfseries,bluelayer] at (5,19.5)   {$\gamma_l$};
  \draw[dashed] (0,-2)--(0,26);
  \draw[dashed] (10,-2)--(10,26);
  \draw[->,thin] (-1,0)--(11.2,0) node[right]{$t$};
  \draw[->,thin] (0,-2)--(0,27)   node[above]{$h$};
  \node[below] at (0,-2.2) {$S$};
  \node[below] at (10,-2.2){$T$};
  \node at (5,-7) {\textbf{(ii)} births switch order};
\end{scope}
\end{tikzpicture}
\caption{Examples of vines $\gamma_k$ (red) and $\gamma_l$ (blue) over the segment $\Vine{V}|_{[S,T]}$ which is forwards incompatible at $S$ by vines $(\gamma_k, \gamma_l)$. There are two different cases how this can occur. In (i) $\gamma_l$ is contained inside $\gamma_k$ (when restricted to the segment) but at $S$ the death values are equal. This equality implies $\gamma_k(S)<\gamma_l(S)$. The inclusion of the vines at the remaining times implies  $\gamma_k(t)\nleq \gamma_l(t)$ for all $t\in (S, T]$. In (ii) $\gamma_k$ is contained inside $\gamma_l$ but at $S$ the birth values are equal. This equality implies $\gamma_k(S)<\gamma_l(S)$. The inclusion of the vines at the remaining times implies $\gamma_k(t) \nleq \gamma_l(t)$ for all $t\in (S, T]$.}
\label{fig:casesincompatible}
\end{figure}

We will show that when $S$ is forwards incompatible with respect to vines  $(\gamma_{k}, \gamma_{l})$ then the matrices $\Mat_{B_S^{new}}^{B_t^{new}}(\alpha^{S\to t})$, for $\{B_t^{new}\}$ a forward simplified family of bases over segment $\Vine{V}|_{[S,T]}$ can contain an non-zero element in the $(k,l)$ position. This entry is $\lambda= \frac{\Mat_{B_S}^{B_t}(\alpha^{S\to t})(k,l)}{\Mat_{B_S}^{B_t}(\alpha^{S\to t})(k,k)}$ which turns out to be equal for all $t\in [S,T]$.

\begin{lemma}\label{lem:lambda_t}
Let $\Vine{V}=(\Mod{V}_t, \{\alpha^{s\to t}\}, \{\beta^{t\to s}\})$ be a vineyard module over a segment $\Vine{V}|_{[S,T]}$ of Type 3 ($S$ is right critical) such that associated $\Vineyard{V}$ is forwards incompatible and $(\gamma_{k}, \gamma_{l})$ are the vines that cause the forwards incompatibility at $S$. Let $\{B_t \mid t\in [S,T]\}$ be any family of bases.
Then $$\frac{\Mat_{B_S}^{B_t}(\alpha^{S\to t})(k,l)}{\Mat_{B_S}^{B_t}(\alpha^{S\to t})(k,k)}=\frac{\Mat_{B_S}^{B_T}(\alpha^{S\to T})(k,l)}{\Mat_{B_S}^{B_T}(\alpha^{S\to T})(k,k)}$$  for all $t\in (S,T]$.
\end{lemma}

\begin{theorem}\label{thm:S_incompatible}
Let $\Vine{V}=(\Mod{V}_t, \{\alpha^{s\to t}\}, \{\beta^{t\to s}\})$ be a vineyard module over a segment $\Vine{V}|_{[S,T]}$ of Type 3 ($S$ is right critical) such that associated $\Vineyard{V}$ is forwards incompatible and $(\gamma_{k}, \gamma_{l})$ are the vines that cause the forwards incompatibility at $S$. Let $\{B_t \mid t\in [S,T]\}$ be any family of bases.
Let $$\lambda=\frac{\Mat_{B_S}^{B_T}(\alpha^{S\to T})(k,l)}{\Mat_{B_S}^{B_T}(\alpha^{S\to T})(k,k)}$$  
and $e_{kl}^{-\lambda}$ the elementary matrix with $-\lambda$ in the $(k,l)$ entry. 
Then $\Mat_{B_S}^{B_t}(\alpha^{S\to t})e_{kl}^{-\lambda}$ is a basis transformation matrix for $\Mod{V}_t$ for all $t\in (S, T]$.
Set $B_S^{new}=B_S^{old}$ and $B_t^{new}=\Mat_{B_S}^{B_t}(\alpha^{S\to t})e_{kl}^{-\lambda} (B_t)$ for  $t\in (S,T]$.
Then for all $S<s<t\leq T$ we have $$\Mat_{B_S^{new}}^{B_t^{new}}(\alpha^{S\to t})=\pi_{(\Mod{V}_S, \Mod{V}_t, |S-t|)} (e^{\lambda}_{kl})$$ and 
$$\Mat_{B_s^{new}}^{B_t^{new}}(\alpha^{s\to t})=\pi_{(\Mod{V}_s, \Mod{V}_t, |s-t|)} (\Id).$$

\end{theorem}

\begin{proof}
We will prove for the case (i) where $$\birth(\gamma_k(T))<\birth(\gamma_l(T))<\death(\gamma_l(T))<\death(\gamma_k(T)) \quad \text{ and }\quad \birth(\gamma_k(S))<\birth(\gamma_l(S))<\death(\gamma_l(S))=\death(\gamma_k(S)).$$
We will omit the proof for case (ii) ($\birth(\gamma_k(S))=\birth(\gamma_l(S))<\death(\gamma_k(S))<\death(\gamma_l(S))$ but $\birth(\gamma_l(T))<\birth(\gamma_k(T))< \death(\gamma_k(T))<\death(\gamma_l(T))$) as it is highly analogous where we only need to switch the roles of births and deaths.

By Lemma \ref{lem:lambda_t} we have $$\frac{\Mat_{B_S}^{B_t}(\alpha^{S\to t})(k,l)}{\Mat_{B_S}^{B_t}(\alpha^{S\to t})(k,k)}=\frac{\Mat_{B_S}^{B_T}(\alpha^{S\to T})(k,l)}{\Mat_{B_S}^{B_T}(\alpha^{S\to T})(k,k)}$$  for all $t\in (S,T]$ 

\begin{equation}\label{eq:easy}
\left(\Mat_{B_S}^{B_t}(\alpha^{S\to t})e_{kl}^{-\lambda}\right)(i,j)=\begin{cases}
\Mat_{B_S}^{B_t}(\alpha^{S\to t})(i,j) & \text{if }j\neq l \\
\Mat_{B_S}^{B_t}(\alpha^{S\to t})(i,l) -\frac{\Mat_{B_S}^{B_t}(\alpha^{S\to t})(k,l) \Mat_{B_S}^{B_t}(\alpha^{S\to t})(i,k) }{\Mat_{B_S}^{B_t}(\alpha^{S\to t})(k,k)}&\text{if }j=l.
\end{cases}
\end{equation}

The set of $j$ such that $\birth(\gamma_j(t))<\death(\gamma_j(t))$ is the same for all $t\in [S,T]$ as $S$ is forwards incompatible -- with a pair of deaths colliding means that we do not also have a birth and death pair of heights colliding. 
Let $J$ be the set of vines $\gamma_j$ such that $\birth(\gamma_j(S))<\death(\gamma_j(S))$. By construction, both $k$ and $l$ are in $J$.
By definition $\Mat_{B_S}^{B_t}(\alpha^{S\to t})$ must be zero on every row and column with index outside $J$. 

The first step of the proof is showing that $\Mat_{B_S}^{B_t}(\alpha^{S\to t})e_{kl}^{-\lambda}$  is a basis transformation matrix for $\Mod{V}_t$, that is \begin{itemize}
\item[(a)] $\Mat_{B_S}^{B_t}(\alpha^{S\to t})e_{kl}^{-\lambda}(j,j)\neq 0$ when $\birth(\gamma_j(t))<\death(\gamma_j(t))$, and 
\item[(b)] whenever $\left(\Mat_{B_S}^{B_t}(\alpha^{S\to t})e_{kl}^{-\lambda}\right)(i,j)\neq 0$ we have $\birth(\gamma_i(t))\leq  \birth(\gamma_j(t))$ and $\death(\gamma_i(t))\leq \death(\gamma_j(t))$.
\end{itemize}

\underline{Proving (a):}

Suppose $\birth(\gamma_j(t))<\death(\gamma_j(t))$. Lemma \ref{lem:aliveVines}  thus implies $\Mat_{B_S}^{B_t}(\alpha^{S\to t})(j,j)\neq 0$. For $j\neq l$ implies $\Mat_{B_S}^{B_t}(\alpha^{S\to t})e_{kl}^{-\lambda}(j,j)=\Mat_{B_S}^{B_t}(\alpha^{S\to t})(j,j)\neq 0$. 

For $j=l$ \eqref{eq:easy} becomes
\begin{equation}\label{eq:(l,l)}
\Mat_{B_S}^{B_t}(\alpha^{S\to t})e_{kl}^{-\lambda}(l,l)= \Mat_{B_S}^{B_t}(\alpha^{S\to t})(l,l) -\frac{\Mat_{B_S}^{B_t}(\alpha^{S\to t})(k,l) \Mat_{B_S}^{B_t}(\alpha^{S\to t})(l,k)}{\Mat_{B_S}^{B_t}(\alpha^{S\to t})(k,k)}
\end{equation}

Since $S$ is critical and, by assumption of case (i), $\birth(\gamma_l(S))>\birth(\gamma_k(S))$ we have $\birth(\gamma_l(S))>\birth(\gamma_k(S))+4|S-T|$ and hence  $\birth(\gamma_l(t))>\birth(\gamma_k(S))+|S-t|$. This implies $\Mat_{B_S}^{B_t}(\alpha^{S\to t})(l,k)=0$. Substituting into \eqref{eq:(l,l)}, and using  Lemma \ref{lem:aliveVines}, we get 
$$\Mat_{B_S}^{B_t}(\alpha^{S\to t})e_{kl}^{-\lambda}(l,l)=\Mat_{B_S}^{B_t}(\alpha^{S\to t})(l,l)\neq 0.$$

\underline{Proving (b):}
Suppose $\left(\Mat_{B_S}^{B_t}(\alpha^{S\to t})e_{kl}^{-\lambda}\right)(i,j)\neq 0$

If $j\neq l$ then $\Mat_{B_S}^{B_t}(\alpha^{S\to t})(i,j)=\left(\Mat_{B_S}^{B_t}(\alpha^{S\to t})e_{kl}^{-\lambda}\right)(i,j)\neq 0$. By the definition of $\Mat_{B_S}^{B_t}(\alpha^{S\to t})$ this implies $\birth(\gamma_i(t))\leq  \birth(\gamma_j(t))$ and $\death(\gamma_i(t))\leq \death(\gamma_j(t))$.

Now suppose $\left(\Mat_{B_S}^{B_t}(\alpha^{S\to t})e_{kl}^{-\lambda}\right)(i,l)\neq 0$. This implies $i\neq k$ as 
 $$\left(\Mat_{B_S}^{B_t}(\alpha^{S\to t})e_{kl}^{-\lambda}\right)(k,l)=\Mat_{B_S}^{B_t}(\alpha^{S\to t})(k,l) -\frac{\Mat_{B_S}^{B_t}(\alpha^{S\to t})(k,l) \Mat_{B_S}^{B_t}(\alpha^{S\to t})(k,k)}{\Mat_{B_S}^{B_t}(\alpha^{S\to t})(k,k)}=0.$$

From $\left(\Mat_{B_S}^{B_t}(\alpha^{S\to t})e_{kl}^{-\lambda}\right)(i,l)\neq 0$ either  $\Mat_{B_S}^{B_t}(\alpha^{S\to t})(i,l)\neq 0$ or $\Mat_{B_S}^{B_t}(\alpha^{S\to t})(i,k) \neq 0$. If $\Mat_{B_S}^{B_t}(\alpha^{S\to t})(i,l)\neq 0$ then again $\birth(\gamma_i(t))\leq  \birth(\gamma_l(t))$ and $\death(\gamma_i(t))\leq \death(\gamma_l(t))$ again by the definition of $\Mat_{B_S}^{B_t}(\alpha^{S\to t})$. If $\Mat_{B_S}^{B_t}(\alpha^{S\to t})(i,k) \neq 0$ then
$\birth(\gamma_i(t))\leq  \birth(\gamma_k(t))$ and $\death(\gamma_i(t))\leq \death(\gamma_k(t))$. Combining with  $\birth(\gamma_k(t))<\birth(\gamma_l(t))$ and $\death(\gamma_l(t))<\death(\gamma_k(t))$ we get $\birth(\gamma_i(t))\leq  \birth(\gamma_l(t))$ and $\death(\gamma_i(t))\leq \death(\gamma_l(t))$.

Having now shown both (a) and (b) we know $\Mat_{B_S}^{B_t}(\alpha^{S\to t})e_{kl}^{-\lambda}\in \BTM(\Mod{V}_t)$.

Set $B_t^{new}:=\left(\Mat_{B_S}^{B_t}(\alpha^{S\to t})e_{kl}^{-\lambda}\right)(B_t)$ and $B_S^{new}=B_S^{old}$. Observe that $$\Mat_{B_S^{new}}^{B_t^{old}}(\alpha^{S\to t})=\big(\Mat_{B_S}^{B_t}(\alpha^{S\to t})\,e_{kl}^{-\lambda}\big)\,e_{kl}^{\lambda}$$ we can apply Lemma \ref{lem:transformed_morphism} to say
$\Mat_{B_S^{new}}^{B_t^{new}}(\alpha^{S\to t})=\pi_{(\Mod{V}_S, \Mod{V}_t, |S-t|)}(e_{kl}^{\lambda})$.

The final step is to show $\Mat_{B_s^{new}}^{B_t^{new}}(\alpha^{s\to t})=\pi_{(\Mod{V}_s, \Mod{V}_t, |s-t|)} (\Id)$ for all $S<s<t\leq T$. 

Let $\{x_j^r\in (\Mod{V}_r)_{\birth(\gamma_j(r))}\}$ denote the basis elements of $B_r^{new}$. Let $I_j(s, t)=\{i\mid \birth(\gamma_i(t))\leq \birth(\gamma_j(s))+|s-t| < \death(\gamma_i(t))\}$.

First consider $j\neq l$. From the definition of the matrices, and using that the various maps in the vineyard module commute we see
\begin{align*}
\phi_{\birth(\gamma_j(t))}^{\birth(\gamma_j(S))+|S-t|}(x_j^t)&= \alpha_{\birth(\gamma_j(S))}^{S\to t}(x_j^S)\\
&=\alpha_{\birth(\gamma_j(S))+|S-s|}^{s\to t}\left(\alpha_{\birth(\gamma_j(S))}^{S\to s}(x_j^S)\right)\\
&=\alpha_{\birth(\gamma_j(S))+|S-s|}^{s\to t} \left(\phi_{\birth(\gamma_j(s))}^{\birth(\gamma_j(S))+|S-s|}(x_j^s)\right)\\
&=\sum_{i\in I_j(s,t)} \Mat_{B_s^{new}}^{B_t^{new}}(\alpha^{s\to t})(i,j)\phi_{\birth(\gamma_i(t))}^{\birth(\gamma_j(S))+|S-t|}(x_i^t).
\end{align*}

Now consider the case $j=l$.
From the definition of the matrices, and using that the various maps in the vineyard module commute we now get 
\begin{align*}
\phi_{\birth(\gamma_l(t))}^{\birth(\gamma_l(S))+|S-t|}(x_l^t)&+\lambda \phi_{\birth(\gamma_k(t))}^{\birth(\gamma_l(S))+|S-t|}(x_k^t)\\
&= \alpha_{\birth(\gamma_j(S))}^{S\to t}(x_l^S)\\
&=\alpha_{\birth(\gamma_l(S))+|S-s|}^{s\to t}\left(\alpha_{\birth(\gamma_l(S))}^{S\to s}(x_l^S)\right)\\
&=\alpha_{\birth(\gamma_l(S))+|S-s|}^{s\to t} \left(\phi_{\birth(\gamma_l(s))}^{\birth(\gamma_l(S))+|S-s|}(x_l^s)+\lambda \phi_{\birth(\gamma_k(s))}^{\birth(\gamma_l(S))+|S-s|}(x_k^s)\right)\\
&=\sum_{i\in I_l(s,t)} \Mat_{B_s^{new}}^{B_t^{new}}(\alpha^{s\to t})(i,l)\phi_{\birth(\gamma_i(t))}^{\birth(\gamma_l(S))+|S-t|}(x_i^t) \\
&\quad+ \lambda \sum_{i\in I_l(s,t)} \Mat_{B_s^{new}}^{B_t^{new}}(\alpha^{s\to t})(i,k)\phi_{\birth(\gamma_i(t))}^{\birth(\gamma_j(S))+|S-t|}(x_i^t)\\
\end{align*}
As $\lambda \phi_{\birth(\gamma_k(t))}^{\birth(\gamma_l(S))+|S-t|}(x_k^t)= \lambda \sum_{i\in I_l(s,t)} \Mat_{B_s^{new}}^{B_t^{new}}(\alpha^{s\to t})(i,k)\phi_{\birth(\gamma_i(t))}^{\birth(\gamma_l(S))+|S-t|}(x_i^t)$ we can subtract from both sides to get
$$\phi_{\birth(\gamma_l(t))}^{\birth(\gamma_l(S))+|S-t|}(x_l^t)=\sum_{i\in I_l(s,t)} \Mat_{B_s^{new}}^{B_t^{new}}(\alpha^{s\to t})(i,l)\phi_{\birth(\gamma_i(t))}^{\birth(\gamma_l(S))+|S-t|}(x_i^t)$$

Comparing coefficients in the basis $\phi_{\birth(\gamma_i(t))}^{\birth(\gamma_l(S))+|S-t|}(x_i^t)$ we conclude $ \Mat_{B_s^{new}}^{B_t^{new}}(\alpha^{s\to t})=\pi_{(\Mod{V}_s, \Mod{V}_t, |s-t|)}(\Id)$.

\end{proof}

\section{Forwards Simplifying Vineyard Modules}\label{sec:combining segments}

We will define a forwards simplification procedure which takes as input a family of bases and a partition into segments. 

\begin{defn}
Let $\Vine{V}$ be a vineyard module over vineyard $\Vineyard{V}$, with time parameter $[T_0, T_1]$ and with fixed indexing of the vines. Given a partition of $[T_0, T_1]$ into segments with endpoints $\{S_i\}$, let $\Simp^f(\{S_i\}, \{B_t\})$
denote the new associated family of bases obtained by incrementally forward simplifying over the segments $\Vine{V}|_{[S_i, S_{i+1}]}$. We call this the \textbf{forwards simplified family of bases from $\{B_t\}$ with respect to partition into segments $\{S_i\}$}. We can symmetrically define $\Simp^b(\{S_i\}, \{B_t\})$ as the \textbf{backwards simplified family of bases from $\{B_t\}$ with respect to partition into segments $\{S_i\}$}. 
\end{defn}

\begin{theorem}\label{thm:simplified vineyard}
Given a vineyard module $\Vine{V}$ and any partition of $\Vine{V}$ into segments $\{S_i\}$. Then for any initial choice of bases $\{B_t\}$ the output $\Simp^f(\{S_i\}, \{B_t\})$ will be a forwards simplified family of bases.
\end{theorem}

\begin{proof}
The base case is that the $\Simp^f(\{S_i\}, \{B_t\})$ is forwards simplified when restricted to the first segment $\Vineyard{V}|_{[S_0, S_1]}$. This follows from either Theorem \ref{thm:S_compatible} or Theorem \ref{thm:S_incompatible} depending on whether $S_0$ is forwards compatible or not. The inductive hypothesis is that $\Simp^f(\{S_i\}, \{B_t\})$ is forwards simplified when restricted to $[S_0, S_i]$.

For the inductive step we can forwards simplify the family of bases over the segment $\Vineyard{V}|_{[S_i, S_{i+1}]}$ using either Theorem \ref{thm:S_compatible} or Theorem \ref{thm:S_incompatible} depending on whether $S_i$ is forwards compatible or not. Observing that the basis at $S_i$ is unchanged by this we can then apply Lemma \ref{lem:union} to say that the family of bases is forwards simplified when we restrict the vineyard module to $[S_0, S_i]\cup [S_i, S_{i+1}]=[S_0, S_{i+1}]$.
\end{proof}

Symmetrically, we also have that $\Simp^b(\{S_i\}, \{B_t\})$ will always be a backwards simplified family of bases.

Alongside existence we are also interested in the uniqueness for forwards or backwards simplified bases for vineyard modules over a single segment $[S,T]$ if we fix the bases at both of the endpoints.

\begin{lemma}\label{lem:unique}
Let $\Vine{V}|_{[S,T]}$ be a segment and $\{B_t\}$ and $\{\tilde{B}_t\}$ both forwards simplified (or both backwards simplified) family of bases for $\Vine{V}|_{[S,T]}$. 
If $\tilde{B}_S=B_S$ and $\tilde{B}_T=B_T$, then $\tilde{B}_t=B_t$ for all $t\in [S, T]$.

Furthermore, if $\Vine{V}|_{[S,T]}$ is a segment of Type $1$ then if $\tilde{B}_S=B_S$ or $\tilde{B}_T=B_T$, then $\tilde{B}_t=B_t$ for all $t\in [S, T]$.
\end{lemma}

If we are given an arbitrary set of associated bases for a vineyard module then its forward simplification is not necessarily independent of the choice of segmentation. An ambiguity occurs in the case where $[S,T]$ is of Type 3 ($S$ critical), and there exists $k\neq l $ with 
$$\birth(\gamma_k(S))<\death(\gamma_k(S))=\birth(\gamma_l(S))<\death(\gamma_l(S))$$ and  
$$\birth(\gamma_k(T))<\birth(\gamma_l(T))<\death(\gamma_k(T))<\death(\gamma_l(T)).$$ We saw in the proof of Theorem \ref{thm:S_compatible} that the forward simplified basis constructed is dependent on the choice of basis at time $T$. If we shifted this endpoint we could get a different outcome. It would be possible to explicitly state this dependence and prove independence beyond this choice. However, we will not need this for the remainder of the paper and deem it outside the present scope. 

Instead we will focus on the situation where the input is a vineyard module with bases that are either forwards or backwards simplified, happy enough to know such simplification is possible and a construction given. The question we are more interested in is whether forwards simplification under such input is independent of choice of segmentation. Thankfully such independence does hold - allowing us to have well-defined maps between the sets of forwards vineyard and vector representations and the set of backwards vineyard and vector representations.

\begin{lemma}\label{lem:forward_simp_idempotent}
Let $\Vineyard{V}$ be a vineyard and $\Vine{V}$ a vineyard module over $\Vineyard{V}$ with forwards simplified bases $\{B_t\}$. For any partition into segments $\{S_i\}$ we have $\Simp^f(\{S_i\}, \{B_t\})=\{B_t\}$.
That is, if we forwards simplify using any partition into segments then the resulting basis is unchanged. 
\end{lemma}

This means that the forwards simplification process is idempotent. Symmetrically this result also tells us that if we backwards simplify a family of backwards simplified bases we also leave them unchanged. 

If we start with backwards simplified basis and then forwards simplify it we would expect it to change. However, we want this to be well-defined in the sense that it is independent of the choice of partition into segments we use in the forwards simplification algorithm. 

\begin{theorem}\label{thm:backwards independent of segmentation}
Fix a vineyard $\Vineyard{V}$ over $[T_0, T_1]$ and a vineyard module over $\Vineyard{V}$ with backwards simplified bases $\{B_t\}$. The forwards simplified family of bases of $\{B_t\}$ is independent of the choice of partition into segments.
\end{theorem}
\begin{proof}
For any two partitions into segments there exists a partition into segments which is a common refinement (Lemma \ref{lem:refinement}). With this in mind it is sufficient to show $\Simp^f(\{S_i\}, \{B_t\}) =\Simp^f(\{\tilde{S}_j\}, \{B_t\})$ when $\{S_i\}\subset \{\tilde{S}_j\}$. 
Since forwards simplifying occurs sequentially along the segments we can reduce the proof to the case where $[T_0, T_1]$ is a segment, that is $\{S_i\}=\{T_0,T_1\}$.

Let $\{B_t^{new}\}=\Simp^f(\{T_0,T_1\}, \{B_t\}) $ and $\{\tilde{B}_t^{new}\}=\Simp^f(\{\tilde{S}_j\}, \{B_t\})$. 

Let $\psi_{\varepsilon}^t$ denote the $\varepsilon$-morphism from $\Mod{V}_t$ to itself following the transition maps within $\Mod{V}_t$. We will now split into cases depending on which type of segment.

\noindent
\underline{Case where $\Vineyard{V}|_{[T_0, T_1]}$ a segment of Type $1$:}

By Lemma \ref{lem:unique}, since $\{B_t^{new}\}$ and $\{\tilde{B}_t^{new}\}$ are both forwards simplified and $B_{T_0}^{new}=\tilde{B}_{T_0}^{new}$, we also have $B_t^{new}=\tilde{B}_t^{new}$.

\noindent
\underline{Case where $\Vineyard{V}|_{[T_0, T_1]}$ a segment of Type $2$:}

All the distinct critical heights in $\Mod{V}_{T_1}$ are at least $4|T_0-T_1|$ apart. This implies that $\pi_{(\Mod{V}_{T_1}, \Mod{V}_{T_1}, 0)}=\pi_{(\Mod{V}_{T_1}, \Mod{V}_{T_1}, 2|T_1-T_0|)}$. Using $\tilde{B}_{T_0}^{new}=B_{T_0}^{new}$

\begin{align*}
\Mat_{B_{T_1}^{new}}^{\tilde{B}_{T_1}^{new}}(\id)&=\pi_{(\Mod{V}_{T_1}, \Mod{V}_{T_1}, 2|T_1-T_0|)}(\Mat_{B_{T_1}^{new}}^{\tilde{B}_{T_1}^{new}}(\psi_{2|T_0-T_1|}^{T_1}))\\
&=\pi_{(\Mod{V}_{T_1}, \Mod{V}_{T_1}, 2|T_1-T_0|)}\left(\Mat_{\tilde{B}_{T_0}^{new}}^{\tilde{B}_{T_1}^{new}}(\alpha^{T_0 \to T_1})\Mat_{B_{T_1}^{new}}^{B_{T_0}^{new}}(\beta^{T_1\to T_0})\right).
\end{align*}

As $\{\tilde{B}_t^{new}\}$ and $\{B_t^{new}\}$ are both forwards simplified over $[T_0, T_1]$, and $T_0$ forwards compatible (as is can not be critical), we know 
$$\Mat_{\tilde{B}_{T_0}^{new}}^{\tilde{B}_{T_1}^{new}}(\alpha^{T_0 \to T_1})=\pi_{(\Mod{V}_{T_0}, \Mod{V}_{T_1}, |T_0-T_1|)}(\Id) \quad  \text{ and } \quad \Mat_{B_{T_0}^{new}}^{B_{T_1}^{new}}(\alpha^{T_0\to T_1})=\pi_{(\Mod{V}_{T_0}, \Mod{V}_{T_1}, |T_0-T_1|)}(\Id).$$  Applying Lemma \ref{lem:IdIFF} we also know $\Mat_{B_{T_1}^{new}}^{B_{T_0}^{new}}(\beta^{T_1\to T_0})=\pi_{(\Mod{V}_{T_1}, \Mod{V}_{T_0}, |T_1-T_0|)}(\Id)$. 

Thus $\Mat_{B_{T_1}^{new}}^{\tilde{B}_{T_1}^{new}}(\id)=\pi_{(\Mod{V}_{T_1}, \Mod{V}_{T_1}, 0)}(\Id)$ and $\tilde{B}_{T_1}^{new}=B_{T_1}^{new}$. By Lemma \ref{lem:unique}, since $B_{T_1}^{new}=\tilde{B}_{T_1}^{new}$ and $B_{T_0}^{new}=\tilde{B}_{T_0}^{new}$,  we have $B_{t}^{new}=\tilde{B}_{t}^{new}$ for all $t\in [T_0, T_1]$.

\noindent
\underline{Case where $\Vineyard{V}|_{[T_0, T_1]}$ a segment of Type 3 with $T_0$ forwards incompatible:}


Set $$\lambda_{\tilde{S}_1}=\frac{\Mat_{\tilde{B}_{T_0}^{new}}^{B_{\tilde{S}_1}}(\alpha^{T_0 \to \tilde{S}_1})(k,l)}{\Mat_{\tilde{B}_{T_0}^{new}}^{B_{\tilde{S}_1}}(\alpha^{T_0 \to \tilde{S}_1})(k,k)} \quad \text{ and }\quad \lambda_{T_1}=\frac{\Mat_{B_{T_0}^{new}}^{B_{T_1}}(\alpha^{T_0 \to T_1})(k,l)}{\Mat_{B_{T_0}^{new}}^{B_{T_1}}(\alpha^{T_0 \to T_1})(k,k)}.$$
As $\Vine{V}|_{[T_0,T_1]}$ is a segment with $ \tilde{S}_1\in (T_0, T_1]$, and $\tilde{B}_{T_0}^{new}=B_{T_0}^{new}$, Lemma \ref{lem:lambda_t} tells us that $\lambda_{T_1}=\lambda_{ \tilde{S}_1}$. 
From the forwards simplification in Theorem \ref{thm:S_incompatible} we know
$\tilde{B}_{ \tilde{S}_1}^{new}=\Mat_{\tilde{B}_{T_0}^{new}}^{B_{S_1}}(\alpha^{T_0\to  \tilde{S}_1})e^{-\lambda_{ \tilde{S}_1}}_{kl}(B_{ \tilde{S}_1})$ and
\begin{align*}
B_{S_1}^{new}&=\Mat_{B_{T_0}^{new}}^{B_{ \tilde{S}_1}}(\alpha^{T_0\to  \tilde{S}_1})e^{-\lambda_{T_1}}_{kl}(B_{ \tilde{S}_1})\\
&=\Mat_{\tilde{B}_{T_0}^{new}}^{B_{ \tilde{S}_1}}(\alpha^{T_0\to  \tilde{S}_1})e^{-\lambda_{ \tilde{S}_1}}_{kl}(B_{ \tilde{S}_1})\\
&=\tilde{B}_{ \tilde{S}_1}^{new}.
\end{align*}

By Lemma \ref{lem:unique} this implies $\tilde{B}_{t}^{new}=B_{t}^{new}$ for all $t\in [T_0,  \tilde{S}_1]$. To show $\tilde{B}_{t}^{new}=B_{t}^{new}$ for $t\in ( \tilde{S}_1, T_1]$ we can use a finite partition of $\Vine{V}|_{[ \tilde{S}_1, T_1]}$ into segments of Type $1$ and use the Type $1$ case already covered in this proof sequentially along this partition.

\noindent
\underline{Case where $\Vineyard{V}|_{[T_0, T_1]}$ a segment of Type $3$ with $T_0$ forwards compatible:}

Let $J$ be the indices of the vines whose support contains $(T_0, T_1]$ and $\Id_J$ the diagonal matrix with $\Id_J(j,j)=1$ for $j\in J$ and $0$ otherwise.

First observe that since $T_0$ is critical, $T_1\in \supp(\gamma_i)$ if and only if $\tilde{S}_1 \in \supp(\gamma_i)$. Let 
$$P_{T_1}(i,j)=\begin{cases} 
&1 \text{ when }i=j \text{ and }T_1\in \supp(\gamma_i) \text{ and } T_0 \notin \supp(\gamma_i)\\
 &\Mat_{B_{T_0}}^{B_{T_1}}(\alpha^{T_0\to T_1})(i,j) \text{ otherwise}
 \end{cases}$$

$$\tilde{P}_{\tilde{S}_1}(i,j)=\begin{cases} 
&1 \text{ when }i=j \text{ and }T_1\in \supp(\gamma_i) \text{ and } T_0 \notin \supp(\gamma_i)\\
&\Mat_{B_{T_0}}^{B_{\tilde{S}_1}}(\alpha^{T_0\to \tilde{S}_1})(i,j) \text{ otherwise}
 \end{cases}$$

 From Theorem \ref{thm:S_compatible} we have $B_{T_1}^{new}=P_{T_1}(B_{T_1})$.
 

Since $\{B_t\}$ is backwards simplified and no critical times appear in $(\tilde{S}_1, T_1]$ we have $\Mat_{B_{T_1}}^{B_{\tilde{S}_1}}(\beta^{T_1\to \tilde{S}_1})=\pi_{(\Mod{V}_{T_1}, \Mod{V}_{\tilde{S}_1}, |T_1-\tilde{S}_1|)}(\Id)$. By Lemma \ref{lem:IdIFF} this implies  $\Mat_{B_{T_0}}^{B_{T_1}}(\alpha^{T_0\to T_1})= \Mat_{B_{T_0}}^{B_{\tilde{S}_1}}(\alpha^{T_0\to S_1})$. Substituting into the formula for $P_{T_1}$ and $\tilde{P}_{\tilde{S}_1}$ (observing that $T_1\in \supp(\gamma_j)$ if and only if $\tilde{S}_1\in \supp(\gamma_j)$) the we see that $P_{T_1}=\tilde{P}_{\tilde{S}_1}$ which we now denote by $P$.

We first can verify by induction that $\pi_{(\Mod{V}_{\tilde{S}_i}, \Mod{V}_{\tilde{S}_{i+1}}, |\tilde{S}_i-\tilde{S}_{i+1}|)}(P)=P$ for all $i$, by comparing to $\pi_{(\Mod{V}_{\tilde{S}_{i+1}}, \Mod{V}_{\tilde{S}_{i
+1}}, 0)}(P)$. The base case of $i=0$ is  $\tilde{P}_{\tilde{S}_1}=P$, which is shown above. As $\tilde{S}_0$ is critical yet forwards compatible $\pi_{(\Mod{V}_{\tilde{S}_1}, \Mod{V}_{\tilde{S}_1}, 0)}(P)=P$. For the inductive step note that for $i\geq 1$, $\Vine{V}|_{[\tilde{S}_i, \tilde{S}_{i+1}]}$ is a segment of Type $1$ we have $$\pi_{(\Mod{V}_{\tilde{S}_i}, \Mod{V}_{\tilde{S}_{i+1}}, |\tilde{S}_i-\tilde{S}_{i+1}|)}=\pi_{(\Mod{V}_{\tilde{S}_i}, \Mod{V}_{\tilde{S}_i}, 0)}=\pi_{(\Mod{V}_{\tilde{S}_{i+1}}, \Mod{V}_{\tilde{S}_{i+1}}, 0)}.$$ 

We will show by induction that $\tilde{B}_{\tilde{S}_i}^{new}=P(B_{\tilde{S}_i})$. The base case of $i=1$ holds by Theorem \ref{thm:S_compatible}. 

As $\Vineyard{V}|_{[\tilde{S}_i, \tilde{S}_{i+1}]}$ is a segment of Type $1$  and $\{B_t\}$ were backwards simplified (by assumption) we know $\Mat_{B_{\tilde{S}_i}}^{B_{\tilde{S}_{i+1}}}(\alpha^{\tilde{S}_i\to \tilde{S}_{i+1}})=\Id_J$.  From Lemma \ref{lem:transformed_morphism_domain} we have
$$\Mat_{\tilde{B}_{\tilde{S}_i}^{new}}^{B_{\tilde{S}_{i+1}}}(\alpha^{\tilde{S}_i \to \tilde{S}_{i+1}})=\pi_{(\Mod{V}_{\tilde{S}_i}, \Mod{V}_{\tilde{S}_{i+1}}, |\tilde{S}_i-\tilde{S}_{i+1}|)}(\Id_J P)=\pi_{(\Mod{V}_{\tilde{S}_i}, \Mod{V}_{\tilde{S}_{i+1}}, |\tilde{S}_i-\tilde{S}_{i+1}|)}(P).$$


This implies that $\tilde{B}_{\tilde{S}_{i+1}}^{new}=P(B_{\tilde{S}_{i+1}})$. 

By induction, as $T_1=\tilde{S}_k$ for some finite $k$, $\tilde{B}_{T_1}^{new}=B_{T_1}^{new}$.  By Lemma \ref{lem:unique}, since $B_{T_1}^{new}=\tilde{B}_{T_1}^{new}$ and $B_{T_0}^{new}=\tilde{B}_{T_0}^{new}$,  we have $B_{t}^{new}=\tilde{B}_{t}^{new}$ for all $t\in [T_0, T_1]$.
\end{proof}

One final useful result is a formalisation of partitioning vineyard modules, not by segments, but instead by subintervals where the forwards and backwards simplification process can occur independently. Here the final vineyard and vector result can be computable separately and then concatenated. The sufficient criteria for when we can partition into independent subintervals is simple - we only need a unique choice of basis at that time. Here we will restrict to the choice of $\mathbb{F}_2$ as the underlying field so uniqueness is clear. If we wanted to consider more general fields care needs to be taken regarding rescaling.

\begin{proposition}\label{prop:separate}
Let $\Vine{V}|_{[T_0, T_1]}$ be a vineyard module over the field $\mathbb{F}_2$ and $S\in (T_0, T_1)$ be such that there is a unique choice of basis for $\Mod{V}_S$. Choose any partitions  $\{S_i^0\}$ and $\{S_i^1\}$ of $\Vine{V}|_{[T_0, S]}$ and $\Vine{V}|_{[S, T_1]}$ respectively into segments. Fix an original family of bases $\{B_t^{old}\}_{t\in [T_0, T_1]}$. If
\begin{align*}
\{B_t^{[T_0,S]}\}_{[T_0, S]}&=\Simp^f( \{S_i^0\},\{B_t^{old}\}_{t\in [T_0,S]})\\
\{B_t^{[S,T_1]}\}_{[S, T_1]}&=\Simp^f(\{S_i^1\},\{B_t^{old}\}_{t\in [S,T_1]})\\
\{B_t^{[T_0,T_1]}\}_{[T_0, T_1]}&=\Simp^f(\{S_i^0\}\cup \{S_i^1\},\{B_t^{old}\}_{t\in [T_0,T_1]})
\end{align*}
Then $B_t^{[T_0, S]}=B_t^{[T_0, T_1]}$ for $t\in [T_0, S]$ and $B_t^{[S, T_1]}=B_t^{[T_0, T_1]}$ for $t\in [S, T_1]$.  

If 
\begin{align*}
\{\hat{B}_t^{[T_0,S]}\}_{[T_0, S]}&=\Simp^b(\{S_i^0\},\{B_t^{[T_0,S]}\}_{t\in [T_0,S]})\\
\{\hat{B}_t^{[S,T_1]}\}_{[S, T_1]}&=\Simp^b( \{S_i^1\}, \{B_t^{[S, T_1]}\}_{t\in [S,T_1]})\\
\{\hat{B}_t^{[T_0,T_1]}\}_{[T_0, T_1]}&=\Simp^b( \{S_i^0\}\cup \{S_i^1\},\{B_t^{[T_0, T_1]}\}_{t\in [T_0,T_1]})
\end{align*}
Then $\hat{B}_t^{[T_0, S]}=\hat{B}_t^{[T_0, T_1]}$ for $t\in [T_0, S]$ and $\hat{B}_t^{[S, T_1]}=\hat{B}_t^{[T_0, T_1]}$ for $t\in [S, T_1]$.  
\end{proposition}

\begin{proof}
Recall simplification is done sequentially by segment. This implies by construction that $B_t^{[T_0, S]}=B_t^{[T_0, T_1]}$ for $t\in [T_0, S]$ as the same simplification process is happening in both. Since $\Mod{V}_S$ has a unique choice of basis we must have $B_S^{[T_0, T_1]}=B_S^{old}$. We then can argue that simplification is done sequentially by segment we also get by construction that $B_t^{[S, T_1]}=B_t^{[T_0, T_1]}$ for $t\in [S, T_1]$.

By symmetry, applying the first part to $\Vine{V}^{op}$, gives us the backwards simplified families of bases also agree.

\end{proof}

\section{The trivial vineyard module}\label{sec:trivial}

In this section we investigate the important case of the trivial vineyard module - when the vineyard module is isomorphic to the direct sum of the vine modules corresponding to the vines in the vineyard. We will see that if we forwards and then backwards simplify a trivial vineyard module then we always get the zero vector for the vineyard and vector representation. 
This provides a necessary and sufficient condition for a vineyard module to be trivial. It is worth noting that trivial vineyard modules are related to the notion of interval decomposable when viewed through the lens of 2-parameter persistence modules. The results in this section do not require any assumptions on the field used for the persistence modules in the vineyard module.

Before considering more complicated scenarios we should first cover the case with the vector in the forwards (or backwards) vineyard and vector representations is empty - that is there are no forwards (respectively backwards) incompatible times. Here the story is simple; all such vineyards are trivial.

\begin{proposition}\label{prop:noincompatible}
If $\Vine{V}$ is a vineyard module that has no forwards incompatible times, or no backwards incompatible times, then it is trivial. 
\end{proposition}
\begin{proof}
Suppose $\Vine{V}$ has no forwards incompatible time. For any forwards simplified family of bases $\{B_t\}$ we have $\Mat_{B_s}^{B_t}(\alpha^{s\to t})=\pi_{(\Mod{V}_s, \Mod{V}_t, |s-t|)}(\Id)$ for all $s<t<c^f(s)$. By considering the composition of the interleaving maps when we cross critical times we further have $\Mat_{B_s}^{B_t}(\alpha^{s\to t})=\pi_{(\Mod{V}_s, \Mod{V}_t, |s-t|)}(\Id)$ for all $s<t$. By Proposition \ref{prop:block} this implies $\Vine{V}$ decomposes into vine modules, that is it is trivial. 

If $\Vine{V}$ has no backwards incompatible time, then $\Vine{V}^{op}$ has no forwards incompatible time. This implies $\Vine{V}^{op}$ is trivial, and by construction $\Vine{V}$ must also be trivial.
\end{proof}

Throughout this section we will use consistent notation as follows:
\begin{itemize}
\item $\Vineyard{V}$ be a vineyard with vines $\{\gamma_i\}$
\item  $\Vine{W}=\oplus \Vine{I}(\gamma_i)$ is the direct sum of vine modules equipped with the standard family of bases $\{B_t^W\}$ (note this family of bases remains unchanged by both forwards and backwards simplification). 
\item  $\Vine{V}$ is a trivial vineyard module with associated vineyard $\Vineyard{V}$ with isomorphism $\rho:\Vine{V}\to \Vine{W}$ 
\item $\{B_t^V\}$ a forwards simplified family of bases for $\Vine{V}$
\item $M_t=\Mat_{B_t^V}^{B_t^W}(\rho_{t})$.
\item $\{\hat{B}_t^V\}$ the bases of $\{\Mod{V}_t\}$ by backwards simplifying $\{B_t^V\}$
\item $\hat{M}_t=\Mat_{\hat{B}_t^V}^{B_t^W}(\rho_{t})$
\item $A_t=\Mat_{\hat{B}_t^V}^{B_t^V}(\id)$ the basis transformation matrix from the forwards and then backwards simplified bases, to the merely forwards simplified bases
\end{itemize}

\begin{figure}[h]
\begin{center}
\begin{tikzpicture}[x=1cm,y=1cm,>=Stealth,
    every node/.style={inner sep=1pt},
    modV/.style={draw=violet!70!black, fill=violet!6, rounded corners=2pt,
                 minimum width=1.05cm, minimum height=0.72cm, font=\small},
    modW/.style={draw=orange!75!black, fill=orange!8, rounded corners=2pt,
                 minimum width=1.05cm, minimum height=0.72cm, font=\small},
    blab/.style={font=\footnotesize},
    note/.style={font=\scriptsize, text=black!62, align=center},
    mapname/.style={font=\scriptsize, text=black!55},
    matlab/.style={font=\small}
  ]

  \coordinate (VF) at (0,4.0);
  \coordinate (VB) at (0,1.0);
  \coordinate (W)  at (3.7,2.5);

  \node[modV] (nVF) at (VF) {$\Mod V_t$};
  \node[modV] (nVB) at (VB) {$\Mod V_t$};
  \node[modW] (nW)  at (W)  {$\Mod W_t$};

  \node[blab, text=violet!62!black, above=0pt of nVF] {basis $B_t^V$};
  \node[blab, text=violet!50!black, below=0pt of nVB] {basis $\hat B_t^V$};
  \node[blab, text=orange!72!black, right=2pt of nW] {basis $B_t^W$};

  \node[note, left=4pt of nVF, anchor=east] {forwards\\ simplified};
  \node[note, left=4pt of nVB, anchor=east] {forwards then\\ backwards simp.};

  \draw[->, violet!60!black, line width=0.9pt] (nVB) -- node[left=2pt,matlab]{$A_t$}
        node[right=2pt,mapname]{$\id$} (nVF);
  \draw[->, black, line width=0.9pt] (nVF) -- node[pos=0.42,above=2pt,matlab]{$M_t$}
        node[pos=0.74,below=1pt,mapname]{$\rho_t$} (nW);
  \draw[->, black, line width=0.9pt] (nVB) -- node[pos=0.42,below=2pt,matlab]{$\hat M_t$}
        node[pos=0.74,above=1pt,mapname]{$\rho_t$} (nW);

  \node[note, anchor=south, align=center, draw=black!25, rounded corners=2pt,
        fill=black!2, inner sep=3pt] at ($(W)+(0,1.7)$)
        {$\rho:\Vine V\xrightarrow{\ \sim\ }\Vine W$\\[1pt]
         $\Vine W=\bigoplus_i \Vine I[\gamma_i]$};

  \node[note, anchor=north west, align=left] at (-1.7,0.05)
    {$M_t=\Mat_{B_t^V}^{B_t^W}(\rho_t)$\quad
     $\hat M_t=\Mat_{\hat B_t^V}^{B_t^W}(\rho_t)$\quad
     $A_t=\Mat_{\hat{B}_t^V}^{B_t^V}(\id)$};

\end{tikzpicture}
\end{center}
\caption{For each fixed time $t$ the isomorphism
$\rho_t:\Mod V_t\to\Mod W_t$ is recorded using different choices of basis. The two
left-hand nodes are the \emph{same} module $\Mod V_t$ but with the different
bases, so the vertical map is the identity and $A_t$ is the transformation basis matrix;
the two diagonal maps are both $\rho_t$, but with different matrices ($M_t$ and $\hat{M}_t$) reflecting the difference in the choice of basis of $\Mod{V}_t$.}\label{fig:trivial-setup}
\end{figure}

We will need some technical lemmas relating the basis transformation matrices under the forwards and backwards simplification process. 

\begin{lemma}\label{lem:same up to shift}
Let $[S,T]$ be a segment and $t\in [S,T]$.
If $S$ is forwards compatible, then $$\pi_{(\Mod{V}_S, \Mod{W}_t, |S-t|)}(M_t)=\pi_{(\Mod{V}_S, \Mod{W}_t, |S-t|)}(M_S).$$ 
Moreover, if $[S,T]$ is a segment of Type $1$ then $M_t=M_S$.

If $T$ is backwards compatible, then $$\pi_{(\Mod{V}_T, \Mod{W}_t, |T-t|)}(\hat{M}_t)=\pi_{(\Mod{V}_T, \Mod{W}_t, |T-t|)}(\hat{M}_T).$$
Moreover, if $[S,T]$ is a segment of Type $1$ then $\hat{M}_t=\hat{M}_T$. 
\end{lemma}

\begin{lemma}\label{lem:COB forwards incompatible}

Let $\Vine{V}|_{[S, T]}$ be a segment of Type 3 ($S$ critical) which is forwards incompatible at $S$, and $(\gamma_{k}, \gamma_{l})$ are the vines that cause the forwards incompatibility. Let $\lambda=\frac{M_S(k,l)}{M_S(k,k)}$.
Then $$M_{t}= \pi_{(\Mod{V}_{t}, \Mod{W}_{t}, 0)}(M_S e^{-\lambda}_{kl})$$
for all $t\in (S,T]$.

Let $\Vine{V}|_{[S, T]}$ be a segment of Type 2 ($T$ critical) which is backwards incompatible at $T$, and $(\gamma_{k}, \gamma_{l})$ are the vines that cause the backwards incompatibility. Let $\hat{\lambda}=\frac{\hat{M}_T(k,l)}{\hat{M}_T(k,k)}$.
Then $$\hat{M}_{t}= \pi_{(\Mod{V}_{t}, \Mod{W}_{t}, 0)}(\hat{M}_T e^{-\hat{\lambda}}_{kl})$$
for all $t\in [S,T)$.
\end{lemma}

As disjoint intervals effectively have no observable relationship for the basis transformations we will need to define appropriate subsets of time where the intervals overlap.

For each $t$ let
$$\tau_t(i,j)=\sup\big\{\{\hat{t}\leq t\mid \gamma_i(\hat{t})\text{ and }\gamma_j(\hat{t}) \text{ are disjoint}\} \cup \{T_0\} \big\}.$$

Note that $\tau_t(i,j)=t$ if and only if $\gamma_i(t)\text{ and }\gamma_j(t) \text{ are disjoint}$ (as intervals). This includes the case when either $\gamma_i(t)$ or $\gamma_j(t)$ is empty, that is $t\notin \supp(\gamma_i) \cap \gamma_j(t)$.

\begin{definition}
We say that $\gamma_i$ is \textbf{observably below} $\gamma_j$ \textbf{before} $t$ if $\tau_t(i,j)\neq t$ and $\birth(\gamma_i(s))\leq \birth(\gamma_j(s))<\death(\gamma_i(s))\leq \death(\gamma_j(s))$ for all $s\in (\tau_t(i,j), t]$.
\end{definition}

We will want to show that off diagonal entries in $M_t$, $\hat{M}_t$ and $A_t$ force vines to be observably below. For this it will be useful to show there is a transitivity of this relationship.

\begin{lemma}\label{lem:transitivity}
Suppose $\gamma_i$ is observably below $\gamma_j$ before $t$, and $\gamma_j$ is observably below $\gamma_k$ before $t$.  If $\tau_t(i,k)\neq t$ then $\gamma_i$ is observably below $\gamma_k$ before $t$.
\end{lemma}

We will be wanting to prove that the basis transformation matrices for simplified bases are well behaved and appropriately have zero in entries related to pairs of vines that don't have consistent partial orders across time. 

\begin{proposition}\label{prop:tau_0} 
If $M_t(i,j)\neq 0$ then $\gamma_i$ is observably below $\gamma_j$ before $t$.
\end{proposition}

\begin{proof}
Let $\{S_m\}_{m=0}^{N}$ be a partition of $\Vineyard{V}$ into segments. 
We will prove this claim by induction on $m$, showing true for all $t\in (S_{m}, S_{m+1}]$ at each step.

With a slight abuse of notation, let the base case be the degenerate segment $\{T_0\}$. This base case holds as $\rho_{T_0}:\Mod{V}_{T_0} \to \Mod{W}_{T_0}$ is a morphism so $M_{T_0}(i,j)\neq 0$ implies  
$$\birth(\gamma_i(T_0))\leq \birth(\gamma_j(T_0))<\death(\gamma_i(T_0))\leq \death(\gamma_j(T_0)).$$

Suppose that $M_{t}(i,j)\neq 0$ for some $t\in (S_{m}, S_{m+1}]$. 

As $M_{t}(i,j)\neq 0$ for some $t\in (S_{m}, S_{m+1})$ we know $\birth(\gamma_i(s))\leq \birth(\gamma_j(s))<\death(\gamma_i(s))\leq \death(\gamma_j(s))$ for all $s\in (S_m, t]$. 
If $\tau_t(i,j)=S_m$ then we are done so assume $\tau_t(i,j)<S_m$ and hence $\tau_t(i,j)=\tau_{S_m}(i,j)$.

If $M_{S_m}(i,j)\neq 0$ then we can apply the inductive hypothesis to say $\birth(\gamma_i(s))\leq \birth(\gamma_j(s))<\death(\gamma_i(s))\leq \death(\gamma_j(s))$ for all $s\in (\tau_{t}(i,j), t].$

We thus have reduced the proof of the inductive step to the case where
\begin{itemize}
\item $M_{t}(i,j)\neq 0$ for some $t\in (S_m, S_{m+1}]$,
\item $M_{S_m}(i,j)=0$, and 
\item $\birth(\gamma_j(S_m))<\death(\gamma_i(S_m))$ 
\end{itemize}

We will now split into the two cases of $S_m$ forwards compatible and $S_m$ forwards incompatible.

Suppose $S_m$ is forwards compatible with $M_{S_m}(i,j)=0$, and $\birth(\gamma_j(S_m))<\death(\gamma_i(S_m))$.
By way of contradiction suppose there exists a $t\in (S_m, S_{m+1}]$ such that $M_{t}(i,j)\neq 0$.

By Lemma \ref{lem:same up to shift} we have $\pi_{(\Mod{V}_{S_m}, \Mod{W}_t, |S_m-t|)}(M_t)=\pi_{(\Mod{V}_{S_m}, \Mod{W}_t, |{S_m}-t|)}(M_{S_m}).$ Thus $M_{S_m}(i,j)=0$ implies that $\pi_{(\Mod{V}_{S_m}, \Mod{W}_t, |S_m-t|)}(M_t)(i,j)=0$. Combined with $M_{t}(i,j)\neq 0$ we must have 
$$\birth(\gamma_j(t))<\death(\gamma_i(t))\leq \birth(\gamma_j(S_m))+|S_m-t|$$
As $\birth(\gamma_j(s))$ is 1-Lipschitz in $s$ we have $\birth(\gamma_j(t))<\death(\gamma_i(t))<\birth(\gamma_j(S_m))+2|S_m-t|$. By the definition of segment this forces $\birth(\gamma_j(S_m))=\death(\gamma_i(S_m))$ which is a contradiction.

Now suppose that $S_m$ is not forwards compatible. Since $\{B_t\}$ is forwards simplified, if $M_{t}(i,j)= 0$ for any $t\in (S_m, S_{m+1})$ then $M_{S_{m+1}}(i,j)=0$. This implies it is sufficient to consider the case where $M_t(i,j)\neq 0$ for some $t\in (S_{m}, S_{m+1})$.

Let $(\gamma_k, \gamma_l)$ be the pair of vines causing the incompatibility at $S_m$.
By Lemma \ref{lem:COB forwards incompatible}, (and Lemma \ref{lem:same up to shift} which gives $M_t=M_{S_{m+1}}$ on $(S_m,S_{m+1}]$) we have $M_t=\pi_{(\Mod{V}_{t}, \Mod{W}_{t}, 0)}(M_{S_m}e^{-\lambda}_{kl})$ with $\lambda=M_{S_m}(k,l)/M_{S_m} (k,k)$. 
The only way for $M_t(i,j)\neq 0$ but $M_{S_m}(i,j)=0$ is for $\lambda \neq 0$, $j=l$ and $i\neq k$. Furthermore
we have $$M_t(i,l)=M_{S_m}(i,l)- \frac{M_{S_m}(k,l)}{M_{S_m} (k,k)}M_{S_m}(i,k).$$

From this we infer that both $M_{S_m}(k,l)\neq 0$ and $M_{S_m}(i,k)\neq 0$. Applying the inductive assumption for $(k,l)$ and $(i,k)$, and the transitivity of Lemma \ref{lem:transitivity} we have $\gamma_i$ is observably below $\gamma_l$ before $S_m$.
Since $S_m$ is a critical, forwards incompatible time we have $\tau_{S_m}(i,l)=\tau_{t}(i,l)$ and $\birth(\gamma_i(s))\leq \birth(\gamma_l(s))<\death(\gamma_i(s))\leq \death(\gamma_l(s))$ for all $s\in[S_m, t]$.

Having now exhaustively covered every case we have completed proving the inductive step.
\end{proof}

While forwards simplified bases do have nice properties, the true power of the simplification process emerges when we combine forwards and backwards simplification.
Alongside the basis transformation matrices between vineyard modules $\Vine{V}$ and $\Vine{W}$ it will be useful to consider the basis transformation matrices from $\Vine{V}$ to itself comparing after and before backwards simplifying.

\begin{proposition}\label{prop:backwards tau_0} 
If $\hat{M}_t(i,j)\neq 0$ or $A_t(i,j)\neq 0$  then $\gamma_i$ is observably below $\gamma_j$ before $t$.

Moreover, if $t$ is backwards incompatible with respect to vines $(\gamma_k, \gamma_l)$ then $\hat{M}_t(k,l)=0$.
\end{proposition}
\begin{proof}
Let $\{S_m\}_{m=0}^{N}$ be a segmentation of $\Vineyard{V}$ into segments. We will prove this claim by induction on $m$ (decreasing from $m=N$ to $m=0$), showing true for all $t\in [S_{m}, S_N]$ (given true for $t\in [S_{m+1}, S_N]$) at each step. The base case holds by Proposition \ref{prop:tau_0} as $\hat{B}_{S_N}=B_{S_N}$ (so $\hat{M}_{S_N}=M_{S_N}$ and $A_{S_N}=\pi_{(\Mod{V}_{S_N}, \Mod{V}_{S_N},0)}(\Id)$).
 
Let us now prove the inductive step. We will split into the cases of $\hat{M}_t(i,j)\neq 0$ for some $t\in [S_m, S_{m+1})$ and then $A_t(i,j)\neq 0$ for some $t\in [S_m, S_{m+1})$.

\noindent
\underline{For $t\in [S_m, S_{m+1})$, $\hat{M}_t(i,j)\neq 0$ implies $\gamma_i$ is observably below $\gamma_j$ before $t$}

 Note that $\hat{M}_t(i,j)\neq 0$ implies $$\birth(\gamma_i(t))\leq \birth(\gamma_j(t))<\death(\gamma_i(t))\leq \death(\gamma_j(t)).$$
If $\birth(\gamma_j(S_m))= \death(\gamma_i(S_m))$ then $\tau_t(i,j)=S_m$ and by definition $\gamma_i$ is observably below $\gamma_j$ before $t$.
We can thus assume $\birth(\gamma_j(S_m))<\death(\gamma_i(S_m))$.

We will split into the following cases:
\begin{enumerate}
\item[(i)] $S_{m+1}$ is backwards compatible and $\birth(\gamma_j(S_{m+1}))\neq \death(\gamma_i(S_{m+1}))$ 
\item[(ii)] $S_{m+1}$ is backwards compatible and $\birth(\gamma_j(S_{m+1})) = \death(\gamma_i(S_{m+1}))$ 
\item[(iii)] $S_{m+1}$ backwards incompatible.
\end{enumerate}

\underline{(i) $S_{m+1}$ is backwards compatible with $\birth(\gamma_j(S_{m+1}))\neq \death(\gamma_i(S_{m+1}))$:}

Combining $\birth(\gamma_j(S_m))\neq \death(\gamma_i(S_m))$ with $\birth(\gamma_j(S_{m+1}))\neq \death(\gamma_i(S_{m+1}))$ implies $\birth(\gamma_j(t))+|S_{m+1}-t| < \death(\gamma_i(S_{m+1}))$.

By Lemma \ref{lem:same up to shift} $\pi_{(\Mod{V}_{S_{m+1}}, \Mod{W}_t, |S_{m+1}-t|)}(\hat{M}_t)=\pi_{(\Mod{V}_{S_{m+1}}, \Mod{W}_t, |S_{m+1}-t|)}(\hat{M}_{S_{m+1}}).$
 As $\birth(\gamma_j(t))+|S_{m+1}-t| < \death(\gamma_i(S_{m+1}))$, $\hat{M}_t(i,j)=\hat{M}_{S_{m+1}}(i,j)$ so by our inductive hypothesis 
$\gamma_i$ is observably below $\gamma_j$ before $t$.

\underline{(ii) $S_{m+1}$ is backwards compatible with $\birth(\gamma_j(S_{m+1})) = \death(\gamma_i(S_{m+1}))$:}

Suppose $\hat{M}_t(i,j)\neq 0$ for some $t\in [S_m, S_{m+1})$.

From Theorem \ref{thm:backwards independent of segmentation} we know that $\hat{B}_t$ is independent of the choice of partition into segments. As $S_{m+1}$ is critical $\Vine{V}|_{[t, S_{m+1}]}$ is a segment and so we can compute $\hat{B}_t$ using Theorem \ref{thm:S_compatible}, done in the reverse direction, on the segment $\Vine{V}|_{[t, S_{m+1}]}$ to state $\hat{B}_{t}=\Mat_{\hat{B}_{S_{m+1}}}^{B_t}(\beta^{S_{m+1} \to t})(B_{t})$. 

Recall that $M_t=\Mat_{B_t}^{B_t^W}(\rho_t)$ and $\hat{M}_t=\Mat_{\hat{B}_t}^{B_t^W}(\rho_t)$. Thus by Lemma \ref{lem:transformed_morphism_domain} 
$$\hat{M}_t=\pi_{(\Mod{V}_t, \Mod{V}_t, 0)}\left(M_t (\Mat_{\hat{B}_{S_{m+1}}}^{B_t}(\beta^{S_{m+1} \to t}))\right).$$

Since
\begin{align*}
\Mat_{\hat{B}_{S_{m+1}}}^{B_t}(\beta^{S_{m+1} \to t})&=\pi_{(\Mod{V}_{S_{m+1}}, \Mod{V}_t, |S_{m+1}-t|)}(\Mat_{\hat{B}_{S_{m+1}}}^{B_{S_{m+1}}}(\id)\Mat_{B_{S_{m+1}}}^{B_t}(\beta^{S_{m+1} \to t}))\\
&=\pi_{(\Mod{V}_{S_{m+1}}, \Mod{V}_t, |S_{m+1}-t|)}(A_{S_{m+1}}\Id_J)\\
&=A_{S_{m+1}}
\end{align*}
we have  $\hat{M}_t(i,j)= \sum_k M_t(i,k)A_{S_{m+1}}(k,j)$ whenever $\hat{M}_t(i,j)\neq 0$.

From $\hat{M}_t(i,j)\neq 0$ at least one of $M_t(i,j)\neq 0$, $A_{S_{m+1}}(i,j)\neq 0$  or there exists a $k$ such that $M_t(i,k)\neq 0$ and $A_{S_{m+1}}(k,j)\neq 0$. If $M_t(i,j)\neq 0$ then by Proposition \ref{prop:tau_0} $\gamma_i$ is observably below $\gamma_j$ before $t$. If $A_{S_{m+1}}(i,j)\neq 0$ then by the inductive hypothesis  $\gamma_i$ is observably below $\gamma_j$ before $t$. If $M_t(i,k)\neq 0$ and $A_{S_{m+1}}(k,j)\neq 0$ then $\gamma_i$ is observably below $\gamma_k$ before $t$ (Proposition \ref{prop:tau_0}) and $\gamma_k$ is observably below $\gamma_j$ before $t$ (inductive hypothesis). Transitivity (Lemma \ref{lem:transitivity}) implies $\gamma_i$ is observably below $\gamma_j$ before $t$.

\underline{(iii) Suppose that $S_{m+1}$ is critical and backwards incompatible by vines $(\gamma_k, \gamma_l)$.}

As $S_{m+1}$ is critical and backwards incompatible by vines $(\gamma_k, \gamma_l)$ we have $S_m \in (\tau_{S_{m+1}}(k,l), S_{m+1}]$ but we do not have $\birth(\gamma_k(S_m))\leq \birth(\gamma_l(S_m))<\death(\gamma_k(S_m))\leq \death(\gamma_l(S_m))$. By the inductive hypothesis we have $\hat{M}_{S_{m+1}}(k,l)=0$ (so $\hat{\lambda}=0$).

Observe $\Vine{V}|_{[t,S_{m+1}]}$ is a segment of Type 2 ($T$ critical) which is backwards incompatible at $T$, and $(\gamma_{k}, \gamma_{l})$ are the vines that cause the backwards incompatibility. By Lemma \ref{lem:COB forwards incompatible} 
$$\hat{M}_{t}= \pi_{(\Mod{V}_{t}, \Mod{W}_{t}, 0)}(\hat{M}_{S_{m+1}} e^{-\hat{\lambda}}_{kl})=\pi_{(\Mod{V}_{t}, \Mod{W}_{t}, 0)}(\hat{M}_{S_{m+1}} ).$$
Suppose $(i,j)\neq (k,l)$. If $\hat{M}_{t}(i,j)\neq 0$ then $\hat{M}_{S_{m+1}}(i,j)\neq 0$, and by the inductive hypotheses $\gamma_i$ is observably below $\gamma_j$ before $S_{m+1}$. Since no change in the order of the births and deaths of $\gamma_i$ compared to $\gamma_j$ occur in the time interval $[S_m, S_{m+1}]$ we also have  $\gamma_i$ is observably below $\gamma_j$ before $t$.

Having now exhaustively covered every case we have completed proving the inductive step or $\hat{M}_t(i,j)\neq 0$.

\underline{For $t\in [S_m, S_{m+1})$, $A_t(i,j)\neq 0$ implies $\gamma_i$ is observably below $\gamma_j$ before $t$}:

Fix $j$. By way of contradiction suppose there exists an $i$ such that $A_t(i,j)\neq 0$ but $\gamma_i$ is not observably below $\gamma_j$ before $t$. Without loss of generality we can assume that that $i$ is maximal in the sense that there is no $p$ such that $A_t(p,j)\neq 0$, $\gamma_p$ is not observably below $\gamma_j$ before $t$, and $\gamma_i(t)\leq \gamma_p(t)$. If there exists a counterexample such maximal elements must exist as we have only finitely many indices (note we are not claiming that the maximal elements are unique).

As $\hat{M}_t=\pi_{(\Mod{V}_t, \Mod{W}_t, 0)}(M_t A_t)$ we have 
\begin{equation}\label{eq:A}
\hat{M}_t(i,j)=\sum_{k}M_t(i,k) A_t(k,j)= M_t(i,i) A_t(i,j)+ \sum_{\{k\neq i:\gamma_i(t)\leq \gamma_k(t)\leq \gamma_j(t)\}}M_t(i,k) A_t(k,j)
\end{equation}
As $\gamma_i$ is not observably below $\gamma_j$ before $t$ we have $M_t(i,j)=0$ and $\hat{M}_t(i,j)=0$. As $M_t(i,i)\neq 0$ \eqref{eq:A} implies there exists $k$ such that $\gamma_i(t)\leq \gamma_k(t)\leq \gamma_j(t)$ and $\hat{M}_t(i,k) \neq 0$ and $A_t(k,j)\neq 0$. As $i$ was chosen to be maximal and $A_t(k,j)\neq 0$ we know $\gamma_k$ is observably below $\gamma_j$ before $t$.
From $M_t(i,k) \neq 0$ we know  $\gamma_i$ is observably below $\gamma_k$ before $t$.  By transitivity (Lemma \ref{lem:transitivity}) $\gamma_i$ is observably below $\gamma_j$ before $t$, causing a contradiction.

\end{proof}

We can now combine all these results to prove our vineyard and vector characterisation of trivial vineyard modules.

\begin{theorem}\label{thm:trivial}
Let $\Vine{V}$ be a vineyard module which has been forwards and then backwards simplified with vines and backwards vector representation $\left(\Vineyard{V},(\lambda(t_1),\lambda(t_2),\ldots \lambda(t_K))\right)$. Then $\Vine{V}$ is isomorphic to the direct sum of vine modules if and only if $\lambda(t_k)=0$ for all $k$. 
\end{theorem}

\begin{proof}
For the direct sum of vine modules  all the interleaving maps are projections of the identity matrix so $\lambda(t_k)=0$ for all $k$ by construction. We can then apply Proposition \ref{prop:isom} to say that if $\lambda(t_k)=0$ for all $k$ then $\Vine{V}$ is isomorphic to a direct sum of vine modules. 

To prove the other direction, assume that $\Vine{V}$ is isomorphic to a direct sum of vine modules. Suppose that $\Vineyard{V}$ is backwards incompatible at $t_i$ with respect to vines $(\gamma_k, \gamma_l)$. By Lemma \ref{lem:COB forwards incompatible}, $\lambda(t_i)=\frac{\hat{M}_{t_{i}}(k,l)}{\hat{M}_{t_i}(k,k)}$. By Proposition \ref{prop:backwards tau_0} $\hat{M}_{t_i}(k,l)=0$, so $\lambda(t_i)=0$.

Thus $\lambda(t_i)=0$ for all $i=1, \ldots K$ and the backwards vector representation is the zero vector. 
\end{proof}

\section{Vineyards with two vines}\label{sec:two vines}
From now we will make a further assumption that the underlying field for the various persistence modules is $\mathbb{F}_2$.  In this section we will restrict our attention to vineyards with two vines. In this case we show that the vineyard and vector representation (after forwards and backwards simplification) uniquely describes the isomorphism class of the vineyard modules, and we can describe which binary vectors are realisable. This allows us to enumerate the number of isomorphism classes of vineyard modules over a fixed vineyard. It would be possible to consider more general fields but then we need to consider appropriate rescalings, not impossible but annoying to keep track of.

Let us start with an illustrative example of the simplest example of a vineyard module with two vines which is not decomposable into two vine modules, complete with a realisation as the degree 1 persistent homology of sublevel sets of a time varying family of functions. 

The underlying space is the disc of radius two in the plane which we split into four sets: $A=\{\|z\|<1\}$, $B=\{\|z\|=1\}$, $C=\{1<\|z\|<2\}$ and $D=\{\|z\|=2\}$. We have $K=A\cup B\cup C\cup D$ and $f_t:K\to \R$ continuous with respect to $t\in [0,10]$ and for each $t$ we have $f_t$ is constant on each of $A$, $B$, $C$, $D$. 
$$f_t(A)=\frac{1}{4}(21- t), \quad  f_t(B)=\frac{1}{4}t, \quad f_t(C)=\frac{1}{4}(15+t), \quad f_t(D)=\frac{1}{4}(14-t).$$
Note that all sublevel sets are closed; the open set $A$ has boundary $B$ and $f_t(B)\leq f_t(A)$ for all $t$, and the open set $C$ has boundary $B\cup D$ and $f_t(C) \geq f_t(B), f_t(D)$ for all $t$.

\begin{figure}[tbp]
\centering
\def\rO{0.2}    
\def\rI{0.1}    
\newcommand{\Kicon}[6]{%
  \begin{scope}[shift={(#1,#2)},x=1cm,y=1cm]
    \fill[#5,even odd rule] (0,0) circle (\rO) (0,0) circle (\rI);
    \fill[#3] (0,0) circle (\rI);
    \draw[#4,line width=0.6pt] (0,0) circle (\rI);
    \draw[#6,line width=0.6pt] (0,0) circle (\rO);
  \end{scope}}
\newcommand{\exampleVines}{%
  \draw[dashed,gray] (3,-4.4) -- (3,10);
  \draw[dashed,gray] (7,-4.4) -- (7,10);
  \fill[redlayer,opacity=0.25] (0,0) -- (10,2.5) -- (10,2.75) -- (0,5.25) -- cycle;
  \draw[thick,redlayer] (0,0) -- (10,2.5);
  \draw[thick,redlayer] (0,5.25) -- (10,2.75);
  \fill[bluelayer,opacity=0.25] (0,3.5) -- (10,1) -- (10,6.25) -- (0,3.75) -- cycle;
  \draw[thick,bluelayer] (0,3.5) -- (10,1);
  \draw[thick,bluelayer] (0,3.75) -- (10,6.25);
  \draw[->] (0,0) -- (10.6,0) node[right] {$t$};
  \draw[->] (0,-0.3) -- (0,6.9) node[above] {$h$};
  \foreach \x/\l in {0/0,3/{\tilde t=3},7/{\tilde s=7},10/10}
    \draw (\x,0) -- (\x,-0.15) node[below,font=\scriptsize,fill=white,inner sep=1pt] {$\l$};
  \node[redlayer,font=\small] at (1.5,1.6) {$\gamma_1$};
  \node[bluelayer,font=\small] at (8.5,4.7) {$\gamma_2$};}
\newcommand{\basisbox}[4]{%
  \node[draw,rounded corners,align=center,font=\scriptsize,fill=white,inner sep=2pt,#4]
    at (#1,#2) {$\color{redlayer}x_1^t=[B]$\\ $\color{bluelayer}x_2^t=#3$};}
\newcommand{\matbox}[3]{%
  \node[font=\scriptsize,fill=white,inner sep=1pt] at (#1,#2)
    {$\begin{bmatrix}#3\end{bmatrix}$};}
\newcommand{\backforth}[3]{%
  \draw[-{Stealth[length=3pt]}] (#1,#3+0.22) -- (#2,#3+0.22);
  \draw[-{Stealth[length=3pt]}] (#2,#3-0.22) -- (#1,#3-0.22);}
\tikzset{chosen/.style={very thick,fill=yellow!15}}

\begin{tikzpicture}[x=1cm]
\colorlet{gr}{black!20}
\definecolor{cA}{HTML}{FDB863} 
\definecolor{cB}{HTML}{E66101} 
\definecolor{cC}{HTML}{B2ABD2} 
\definecolor{cD}{HTML}{5E3C99} 

\begin{scope}[y=0.7cm]
\node[anchor=west,font=\small] at (-0.6,8.0) {(a) The functions $f_t$ and their sublevel sets};
\draw[->] (0,-0.8) -- (0,7.05) node[above] {$h$};
\draw[->] (0,0) -- (10.7,0) node[right] {$t$};
\node[below left,font=\scriptsize] at (0,0) {0};
\foreach \y in {1,2,3,4,5,6}
  \draw (0,\y) -- (-0.08,\y) node[left,font=\scriptsize] {\y};
\foreach \x in {3,7,10}
  \draw (\x,0) -- (\x,-0.1) node[below,font=\scriptsize] {\x};
\draw[dashed,gray] (3,-0.8) -- (3,6.8);
\draw[dashed,gray] (7,-0.8) -- (7,6.8);
\node[font=\footnotesize] at (1.5,7.1) {$0\le t < 3$};
\node[font=\footnotesize] at (5  ,7.1) {$3< t< 7$};
\node[font=\footnotesize] at (8.5,7.1) {$7< t\le10$};
\draw[cA,thick] (0,5.25) -- (10,2.75); 
\draw[cB,thick] (0,0   ) -- (10,2.5 ); 
\draw[cC,thick] (0,3.75) -- (10,6.25); 
\draw[cD,thick] (0,3.5 ) -- (10,1.0 ); 
\Kicon{1.5}{-0.45}{gr}{gr}{gr}{gr}
\Kicon{1.5}{1.75} {gr}{black}{gr}{gr}
\Kicon{1.5}{3.625}{gr}{black}{gr}{black}
\Kicon{1.5}{4.5}  {gr}{black}{black}{black}
\Kicon{1.5}{5.6}  {black}{black}{black}{black}
\Kicon{5}{-0.45}{gr}{gr}{gr}{gr}
\Kicon{5}{1.75} {gr}{black}{gr}{gr}
\Kicon{5}{3.125}{gr}{black}{gr}{black}
\Kicon{5}{4.5}  {black}{black}{gr}{black}
\Kicon{5}{5.75} {black}{black}{black}{black}
\Kicon{8.5}{-0.45}{gr}{gr}{gr}{gr}
\Kicon{8.5}{1.75} {gr}{gr}{gr}{black}
\Kicon{8.5}{2.625}{gr}{black}{gr}{black}
\Kicon{8.5}{4.5}  {black}{black}{gr}{black}
\Kicon{8.5}{6.4}  {black}{black}{black}{black}
\begin{scope}[shift={(11.7,5.2)},x=1cm,y=1cm]
  \fill[cC,even odd rule] (0,0) circle (0.6) (0,0) circle (0.3);
  \fill[cA] (0,0) circle (0.3);
  \draw[cB,line width=1.1pt] (0,0) circle (0.3);
  \draw[cD,line width=1.1pt] (0,0) circle (0.6);
  \node[font=\scriptsize,fill=white,inner sep=0.4pt] at (0,0)       {$A$};
  \node[font=\scriptsize,fill=white,inner sep=0.4pt] at (0.34,0.34) {$C$};
  \node[font=\scriptsize,fill=white,inner sep=0.4pt] at (0,0.3)     {$B$};
  \node[font=\scriptsize,fill=white,inner sep=0.4pt] at (0,0.6)     {$D$};
  \node[font=\scriptsize] at (0,-0.88) {domain $K$};
\end{scope}
\draw[cA,thick] (10.6,3.4)--(11.0,3.4); \node[right,font=\scriptsize] at (11.1,3.4){$f_t(A)=\frac{21-t}{4}$};
\draw[cB,thick] (10.6,2.8)--(11.0,2.8); \node[right,font=\scriptsize] at (11.1,2.8){$f_t(B)=\frac{t}{4}$};
\draw[cC,thick] (10.6,2.2)--(11.0,2.2); \node[right,font=\scriptsize] at (11.1,2.2){$f_t(C)=\frac{15+t}{4}$};
\draw[cD,thick] (10.6,1.6)--(11.0,1.6); \node[right,font=\scriptsize] at (11.1,1.6){$f_t(D)=\frac{14-t}{4}$};
\fill[black] (10.6,0.85) rectangle (10.85,1.1);
\node[right,font=\scriptsize] at (10.95,0.975){in the sublevel set};
\fill[gr] (10.6,0.35) rectangle (10.85,0.6);
\node[right,font=\scriptsize] at (10.95,0.475){not yet included};
\end{scope}

\begin{scope}[yshift=-6.6cm,y=0.45cm]
\node[anchor=west,font=\small] at (-0.6,11.9)
  {(b) The forwards vs backwards simplified families of bases, with their vineyard and vector representations.};
\exampleVines
\basisbox{1.3}{7.9}{[D+B]}{}
\basisbox{5}{7.9}{[D+B]}{chosen}
\basisbox{8.7}{7.9}{[D]}{}
\backforth{2.35}{3}{7.9}  \matbox{3}{10}{1&0\\0&1}
\backforth{7}{7.65}{7.9}  \matbox{7}{10}{1&\mathbf{1}\\0&1}
\node[anchor=west,align=left,font=\scriptsize] at (10.9,9.0)
  {\textbf{forwards simplified}\\ forwards vector $(1)$};
\basisbox{1.3}{-2.3}{[D+B]}{}
\basisbox{5}{-2.3}{[D]}{chosen}
\basisbox{8.7}{-2.3}{[D]}{}
\backforth{2.35}{3}{-2.3}  \matbox{3}{-4.4}{1&\mathbf{1}\\0&1}
\backforth{7}{7.65}{-2.3}  \matbox{7}{-4.4}{1&0\\0&1}
\node[anchor=west,align=left,font=\scriptsize] at (10.9,-3.4)
  {\textbf{backwards simplified}\\ backwards vector $(1)$};
\end{scope}
\end{tikzpicture}

\caption{The example vineyard module. (a) The sloping lines are the function values of the different sets and the pictures show the sublevel sets. The order in which different subsets appear partitions the interval into three sections shown by dashed lines.  (b) The forwards and backwards simplified families of bases, each with the corresponding vineyard and vector representations. The vine $\gamma_1$ has a unique choice of basis at all times. The vine $\gamma_2$ has unique choice of persistent homology representative over the interval $[0,3)$ and $(7,10]$ but two possible choices for times in $[3,7]$, one choice for forwards simplified and one for backwards simplified.} \label{fig:filtration}
\end{figure}

We have two times where critical heights coincide which are at $t=3$ (with $f_3(A)=f_3(C)$) and at $t=7$ (with $f_7(B)=f_7(D)$). The continuously changing sublevel set filtration defines a vineyard module where the persistence module $\Mod{V}_t$ is the $1$-homology degree persistent homology of the filtration by $f_t$, and the interleaving maps are defined by the natural inclusion maps $f_t^{-1}(-\infty, h]\subset f_s^{-1}(-\infty, h+|s-t|]$ for $h\in \R$, and $s,t\in [0,10]$.

Over the interval $[0, 3)$ there is only one choice of basis of  $\Mod{V}_t$, namely $x_1^t=[B]$ (with $\birth([B])=f_t(B)$ and $\death([B])=f_t(A)$) and $x_2^t=[D+B]$ (with $\birth([D+B])=f_t(D)$ and $\death([D+B])=f_t(C)$). Over the interval $(7, 10]$ there is only one choice of basis of  $\Mod{V}_t$, namely $x_1^t=[B]$ (with $\birth([B])=f_t(B)$ and $\death([B])=f_t(A)$) and $x_2^t=[D]$ (with $\birth([D])=f_t(D)$ and $\death([D])=f_t(C)$). In the middle section, for $t\in [3,7]$ there are two different possible choices of basis; $\{[B], [D+B]\}$ or $\{[B], [D]\}$. When we forward simplify we get bases over this middle range corresponding to $\{[B], [D+B]\}$ and then when we then backwards simplify we get instead the basis corresponding to $\{[B], [D]\}$. After forwards and backwards simplifying we have for small $\delta$, $\beta^{3\to 3-\delta}(x_2^3)=x_2^{3-\delta}+x_1^{3-\delta}$ and  $\beta^{3\to 3-\delta}(x_1^3)=x_1^{3-\delta}$. In matrix form:
$$\Mat_{B_3}^{B_{3-\delta}}(\beta)=\begin{bmatrix} 1 & 1  \\ 
0 & 1 \\
\end{bmatrix}$$
implying the vector in the vines and  vector representation is $(1)$ and the vineyard module is not isomorphic to the direct sum of vine modules.

We will show in Lemma \ref{lem:split} that we can split any vineyard module with two vines over smaller intervals which have at most one backwards incompatible time.

\begin{lemma}\label{lem:split}
Let $\Vineyard{V}$ be a vineyard with two vines. For any backwards incompatible time $\tilde{t}$ there exists an $\tilde{S}<\tilde{t}$, such that there are no critical times in $[\tilde{S}, \tilde{t})$ and $\Mod{V}_{\tilde{S}}$ has a unique choice of basis.

For any forwards incompatible time $\tilde{s}$ there exists an $\tilde{T}>\tilde{s}$ such that there are no critical times in $(\tilde{s}, \tilde{T}]$ and $\Mod{V}_{\tilde{T}}$ has a unique choice of basis.
\end{lemma}

For a vineyard $\Vineyard{V}$ over $[T_0, T_1]$ we can thus find  a sequence of times  $T_0<S_1<S_2<\ldots <S_N<T_1$ such that  $\Mod{V}_{S_i}$ has a unique choice of basis and at most one forwards incompatible and one backwards incompatible time appears in each of  $[T_0, S_1], [S_1, S_2], \ldots , [S_{N-1},S_N]$ and $[S_N, T_1]$. 

Whenever the two vines become disjoint we can further subdivide to reduce cases of interest. For each $t$ let 
$$\sigma(t)=\inf\left\{ \{s>t\mid  \gamma_1(s)\cap\gamma_2(s)=\emptyset\} \cup \{T_1\}\right\}.$$ Note that if $\gamma_1(s)\cap\gamma_2(s)=\emptyset$ then $\Mod{V}_s$ has a unique choice of basis as the vector spaces at the birth times are both one-dimensional. With this notation, also subdivide at $\sigma(\tilde{t})$ for each $\tilde{t}$ which is backwards incompatible.

As we are interested in the backwards vineyard and vector representation we can use Proposition \ref{prop:separate}, ignoring the subintervals containing no backwards incompatible time, to justify only studying vineyard modules $\Vine{V}$ (over $[T_0, T_1]$) with two vines $\gamma_1$ and $\gamma_2$  such that
\begin{itemize}
\item single backwards incompatible time $\tilde{t}$ with $\Vine{V}|_{[T_0,\tilde{t}]}$ a segment
\item either no forwards incompatible times or a single forwards incompatible time $\tilde{s}$ with $\Vine{V}|_{[\tilde{s},T_1]}$ a segment
\item $\sigma(\tilde{t})\geq T_1$.
\end{itemize}

These restrictions make it tractable to list all the cases, one of which is the example already given. Without loss of generality we can index so that $\gamma_2(T_0)\subset \gamma_1(T_0)$. 

\providecommand{\birth}{\operatorname{birth}}
\providecommand{\death}{\operatorname{death}}

\newcommand{\vyband}[3]{%
  \fill[#1,opacity=0.25] plot coordinates {#2} -- plot coordinates {#3} -- cycle;
  \draw[thick,#1] plot coordinates {#2};
  \draw[thick,#1] plot coordinates {#3};
}
\newcommand{\vyaxis}{%
  \draw[->,thin] (-0.7,0) -- (11.4,0) node[right,font=\scriptsize] {$t$};
  \draw[->,thin] (0,-0.4) -- (0,7.1) node[above left,font=\scriptsize,inner sep=1pt] {$h$};
  \node[below,font=\scriptsize,inner sep=2pt] at (0,0)  {$T_0$};
  \node[below,font=\scriptsize,inner sep=2pt] at (10,0) {$T_1$};
}
\newcommand{\vycrit}[2]{%
  \draw[gray,dashed] (#1,-0.4) -- (#1,6.9);
  \node[below,font=\scriptsize,inner sep=2pt] at (#1,0) {$#2$};
}
\newcommand{\vydot}[2]{\node[circle,fill=black,inner sep=1pt] at (#1,#2) {};}
\newcommand{\vylabels}[2]{\vylabelsat{5}{#1}{#2}}
\newcommand{\vylabelsat}[3]{%
  \node[redlayer,font=\bfseries\scriptsize] at (#1,#2) {$\gamma_1$};
  \node[bluelayer,font=\bfseries\scriptsize] at (#1,#3) {$\gamma_2$};}
\newcommand{\vycap}[1]{\node[font=\scriptsize,text width=4.6cm,align=center] at (5,-2.4) {#1};}
\newcommand{\vyhead}[1]{\node[font=\bfseries\scriptsize] at (5,8.4) {#1};}

\begin{figure}[p]
\centering
\begin{tikzpicture}[x=0.40cm,y=0.32cm]

\begin{scope}
  \vyhead{births cross at $\tilde{t}$}
  \vyaxis
  \vyband{bluelayer}{(0,2.5) (10,0)}{(10,4.25) (0,4.25)}
  \vyband{redlayer}{(0,1) (10,3.5)}{(10,5.75) (0,5.75)}
  \vycrit{3}{\tilde{t}}
  \vydot{3}{1.75}
  \vylabels{5}{1.75}
  \vycap{(a) no further critical time}
\end{scope}

\begin{scope}[shift={(0,-11.4)}]
  \vyaxis
  \vyband{bluelayer}{(0,2.5) (10,0)}{(10,3.5) (0,5.25)}
  \vyband{redlayer}{(0,1) (10,3.5)}{(10,5.75) (0,5.75)}
  \vycrit{3}{\tilde{t}}
  \vydot{3}{1.75}
  \vydot{10}{3.5}
  \vylabels{5}{1.75}
  \vycap{(c) $\death(\gamma_2)=\birth(\gamma_1)$ at $T_1$}
\end{scope}

\begin{scope}[shift={(0,-22.8)}]
  \vyaxis
  \vyband{bluelayer}{(0,2.5) (3,1.75) (5,1.25) (7,1.75) (10,2.5)}{(10,4.25) (0,4.25)}
  \vyband{redlayer}{(0,1) (3,1.75) (5,2.25) (7,1.75) (10,1)}{(10,5.75) (0,5.75)}
  \vycrit{3}{\tilde{t}}
  \vycrit{7}{\tilde s}
  \vydot{3}{1.75}
  \vydot{7}{1.75}
  \vylabels{5}{1.7}
  \vycap{(e) $\tilde s$ a crossing of births}
\end{scope}

\begin{scope}[shift={(0,-34.2)}]
  \vyaxis
  \vyband{bluelayer}{(0,2.5) (10,0)}{(10,5.25) (0,2.75)}
  \vyband{redlayer}{(0,1) (10,3.5)}{(10,3.75) (0,6.25)}
  \vycrit{3}{\tilde{t}}
  \vycrit{7}{\tilde s}
  \vydot{3}{1.75}
  \vydot{7}{4.5}
  \vylabels{4.5}{1.7}
  \vycap{(g) $\tilde s$ a crossing of deaths}
\end{scope}

\begin{scope}[shift={(0,-45.6)}]
  \vyaxis
  \vyband{bluelayer}{(0,3) (5,1.75) (10,3)}{(10,4.25) (0,4.25)}
  \vyband{redlayer}{(0,0.5) (5,1.75) (10,0.5)}{(10,5.75) (0,5.75)}
  \vycrit{5}{\tilde{t}=\tilde s}
  \vydot{5}{1.75}
  \vylabelsat{2.5}{5}{3.4}
  \vycap{(i) $\tilde s=\tilde{t}$, the births meeting without crossing}
\end{scope}

\begin{scope}[shift={(13.2,0)}]
  \vyhead{deaths cross at $\tilde{t}$}
  \vyaxis
  \vyband{bluelayer}{(0,2) (10,2)}{(10,6.25) (0,3.75)}
  \vyband{redlayer}{(0,0.5) (10,0.5)}{(10,2.75) (0,5.25)}
  \vycrit{3}{\tilde{t}}
  \vydot{3}{4.5}
  \vylabels{1.2}{4.5}
  \vycap{(b) no further critical time}
\end{scope}

\begin{scope}[shift={(13.2,-11.4)}]
  \vyaxis
  \vyband{bluelayer}{(0,1) (10,2.75)}{(10,6.25) (0,3.75)}
  \vyband{redlayer}{(0,0.5) (10,0.5)}{(10,2.75) (0,5.25)}
  \vycrit{3}{\tilde{t}}
  \vydot{3}{4.5}
  \vydot{10}{2.75}
  \vylabels{1.2}{4.5}
  \vycap{(d) $\death(\gamma_1)=\birth(\gamma_2)$ at $T_1$}
\end{scope}

\begin{scope}[shift={(13.2,-22.8)}]
  \vyaxis
  \vyband{bluelayer}{(0,2) (10,2)}{(10,3.75) (7,4.5) (5,5) (3,4.5) (0,3.75)}
  \vyband{redlayer}{(0,0.5) (10,0.5)}{(10,5.25) (7,4.5) (5,4) (3,4.5) (0,5.25)}
  \vycrit{3}{\tilde{t}}
  \vycrit{7}{\tilde s}
  \vydot{3}{4.5}
  \vydot{7}{4.5}
  \vylabels{1.2}{4.5}
  \vycap{(f) $\tilde s$ a crossing of deaths}
\end{scope}

\begin{scope}[shift={(13.2,-34.2)}]
  \vyaxis
  \vyband{bluelayer}{(0,3.5) (10,1)}{(10,6.25) (0,3.75)}
  \vyband{redlayer}{(0,0) (10,2.5)}{(10,2.75) (0,5.25)}
  \vycrit{3}{\tilde{t}}
  \vycrit{7}{\tilde s}
  \vydot{3}{4.5}
  \vydot{7}{1.75}
  \vylabels{1.75}{4.5}
  \vycap{(h) $\tilde s$ a crossing of births}
\end{scope}

\begin{scope}[shift={(13.2,-45.6)}]
  \vyaxis
  \vyband{bluelayer}{(0,2) (10,2)}{(10,3.25) (5,4.5) (0,3.25)}
  \vyband{redlayer}{(0,0.5) (10,0.5)}{(10,5.75) (5,4.5) (0,5.75)}
  \vycrit{5}{\tilde{t}=\tilde s}
  \vydot{5}{4.5}
  \vylabelsat{2.5}{1.2}{2.85}
  \vycap{(j) $\tilde s=\tilde{t}$, the deaths meeting without crossing}
\end{scope}

\end{tikzpicture}
\caption{The ten possible associated vineyards, with the restrictions described above, and vine indices chosen so $\gamma_2(T_0)\subset\gamma_1(T_0)$. The dashed lines are the critical times and the dots are
the critical heights, where two of the four values cross. Panel (h) is the same associated vineyard as the vineyard module in our worked example and shown in Figure \ref{fig:filtration}.}
\label{fig:two-vine-cases}
\end{figure}

Recall that if there are no forwards incompatible times, the vineyard module is automatically trivial (Proposition \ref{prop:noincompatible}). This implies that in the cases of (a), (b), (c) and (d) in Figure \ref{fig:two-vine-cases}, for any forwards and then backwards simplified family of bases, the corresponding backwards vineyard and vector representation has to be the zero vector.

\begin{lemma}\label{lem:unique_about_backwards_incompatible}
Let $\Vine{V}$ be a vineyard module over $[T_0, T_1]$ with two vines $\gamma_1$ and $\gamma_2$  such that \begin{itemize}
\item there is a single backwards incompatible time $\tilde{t}$ with no critical times in $[T_0, \tilde{t})$,
\item there is a single forwards incompatible time $\tilde{s}$ with no critical times in $(\tilde{s}, T_1]$
\item $\sigma(\tilde{t})\geq T_1$.
\end{itemize}
Then there is a unique backwards simplified family of bases.
\end{lemma}

The consequence of Lemma \ref{prop:realise} is that (for vineyards satisfying these restrictions) we can realise any backwards vineyard and vector representation via a family of bases that has been forwards and backwards simplified, and this canonically represents the vineyard module.

\begin{theorem}\label{thm:twovines}
Let $\Vine{V}$ be a vineyard module with two vines $\gamma_1$ and $\gamma_2$ and backwards incompatible times $(t_1, t_2, \ldots t_K)$. 

Let  $\{B_t^{old}\}$ be a backwards simplified family of bases with vineyard and vector representation $\lambda^{old}=(\lambda_1^{old}, \ldots \lambda_K^{old})$. Then the backwards vineyard and vector representations (after forwards and then backwards simplifying) is $\lambda^{new}=(\lambda_1^{new}, \ldots \lambda_K^{new})$ with 
$$
\lambda_i^{new}=
\begin{cases} \lambda_i^{old} & \text{if there exists a forwards incompatible time in $[t_i, \sigma(t_i)]$}\\
0 & \text{otherwise}.
\end{cases}
$$

Moreover, $\lambda^{new}$ is uniquely determined by $\Vine{V}$ (not dependent on the choice of $\lambda^{old}$ or the choice of the initial family of bases $\{B_t\}$).
\end{theorem}
\begin{proof}
Suppose $t_i$ is  backwards incompatible.
We then choose times $T_0^i$ and $T_1^i$ such that 
\begin{itemize}
\item $T_0^i< t_i< T_1^i$
\item $\Mod{V}_{T_0^i}$ and $\Mod{V}_{T_1^i}$ have unique choices of bases (except in the exceptional case of $T_1^i=T_1$)
\item $t_i$ is the only backwards incompatible time in $[T_0^i, T_1^i]$
\item no critical times occur in $[T_0^i, t_i)$ 
\item at most one  
\item $T_1^i\leq \sigma(t_i)$.
\end{itemize}
Furthermore, if there exists a forwards incompatible time $\tilde{s}_i\in [t_i, T_1^i]$ then there are no critical times in $(\tilde{s}_i, T_i^1].$

By Proposition \ref{prop:separate} we can compute $\lambda_i^{new}$ from forwards and backwards simplifying $\{B_t\}|_{[T_0^i, T_1^i]}$. We will split into the two cases of when $[T_0^i, T_1^i]$ contains a forwards incompatible time or not. It is worth observing that given the constraints above whether $[T_0^i, T_1^i]$ contains a forwards incompatible time or not is independent of the choice of $T_1^i$ - it is whenever there exists a forwards incompatible time in $[t_i, \sigma(t_i)]$.

Note that $[T_0^i, T_1^i]$ if and only if there exists a forwards incompatible time in $[t_i, \sigma(t_i)]$. If $[T_0^i, T_1^i]$ does not contain a forwards incompatible time then by Proposition \ref{prop:noincompatible} we know $\Vine{V}|_{[T_0^i, T_1^i]}$ is a trivial vineyard module. Thus by Theorem \ref{thm:trivial} the vector in the backwards vineyard and vector representation must be the zero vector and $\lambda_i^{new}=0$. This is independent of the initial choice of $\{B_t\}$.

 If $[T_0^i, T_1^i]$ does contain a forwards incompatible time then by Lemma \ref{lem:unique_about_backwards_incompatible} there is only one family of backwards simplified bases. The resulting forwards and backwards simplified family of bases must be $\{B_t\}$ and so $\lambda_i^{new}=\lambda_i^{old}$. Again, as the backwards family of bases is unique it is independent of choice (vacuously as there is no choice).
 
 The choice of bases $B_t$ for $t$ not in any of the $[T_0^i, T_1^i]$ does not impact any of the calculations (Proposition  \ref{prop:separate}) so $\lambda_i^{new}$ is independent of the choice of basis $B_t$.
\end{proof}

We can thus count the number of isomorphism classes of vineyard modules for a fixed associated vineyard with two vines.

\begin{theorem}
Let $\Vineyard{V}$ be a vineyard with two vines. Let $L$ be the number of backwards incompatible times $t_i$ such that there exists a forwards incompatible time in $[t_i, \sigma(t_i)]$. Then the number of isomorphism classes of vineyard modules with associated vineyard $\Vineyard{V}$ is $2^L$.
\end{theorem} 
\begin{proof}
Let the backwards incompatible times of $\Vineyard{V}$ be $(t_1, t_2, \ldots t_K)$. 
Let $J$ be the set of backwards incompatible time indices $j$  such that there are no forwards incompatible times in $[t_j, \sigma(t_j)]$. We have $L=K-|J|$.

Consider the set of vectors $\{(\lambda_1, \ldots \lambda_K)$ with $\lambda_j=0$ for $j\in J$. Each of these backwards vineyard and vector representation can be realised as a vineyard module (Proposition \ref{prop:realise}). The process of forwards and then backwards simplifying leaves such vineyard and vector representations unchanged (Theorem \ref{thm:twovines}). We further know by Theorem \ref{thm:twovines} that each of the vectors respresents distinct isomorphism classes, and that when we forwards and then backwards simplifying every vineyard module can be uniquely represented by one of these vectors. 

This means we can count the number of isomorphism classes. With $L$ entries, independently chosen from $\{0,1\}$ we get $2^L$ isomorphism classes.
\end{proof}

\section{Future work}

Given the framework of forwards and backwards simplification we have a tractable description of vineyard modules. It leads to many algebraic questions including:
\begin{itemize}
\item Is there a nice characterisation of the equivalence classes of the vineyard and vector representations (or vines and matrices) of vineyard modules that are isomorphic?
\item How can we count (and compute) the number of isomorphism classes with an associated vineyard? We have covered the case when there are two vines here but it gets much more complicated quickly for larger numbers of vines.
\item How can we determine the partition of the vines within the vineyards into those within the same indecomposable summand?
\item Studying the dynamic behaviour of the vineyard and vector representations under sequences of forwards and backwards simplification including its fixed points and orbits.
\item Revisiting known theory of 2-parameter persistence and understanding what it means in the vineyard module context. This includes the many invariants and distances defined and studied for 2-parameter persistence modules.
\end{itemize}

There are many ways the theory in this paper could be generalised including:
\begin{itemize}
\item Removing the various simplifying assumptions on our vineyard modules - allowing for countably many vines or allowing for higher multiplicity of intervals within a persistence module. In the case of higher multiplicity we could still guarantee unique vines by imposing some transversality condition.
\item How could we extend this approach to continuous persistence module valued functions over $S^1$? Here we can still define forward simplification locally but will have the potential of monodromy, both in terms of the algebraic maps but also just in the geometric decomposition of the vines (see \cite{canopy, WCT}). What about for more general persistence bundles (\cite{hickok2022persistence})?
\end{itemize}

Another avenue of research is related to applications. Beyond the obvious task of actually computing the vineyards and the matrices for the morphism there are interesting questions about inference and whether structures are realisable such as:
\begin{itemize}
\item  If we consider special cases of vineyard modules do we have nice decompositions or do these decompositions have nice geometric interpretations. For example, what happens when our input is a point cloud, and we construct the vineyard where the time parameter corresponds to a bandwidth and the persistence modules are built from height filtrations of kernel density estimates? Do these representations relate to topological simplification (such as in \cite{nigmetov2020topological})?
\item Can we find conditions for when a vineyard and vector representation is realisable for specified formats of data input (eg. as the Rips complex of some dynamic point cloud, or as sublevel set filtrations of time varying functions over a specified domain)? Can we bound the probabilities of different kinds of representations under randomly generated dynamic point clouds?
\end{itemize}

\bibliographystyle{plain}
\bibliography{representingbib}
  
\appendix
\section{Matrices of $\varepsilon$-morphisms}
\label{app:matrices}

\begin{proof}[Proof of Lemma \ref{lem:constructionMor}]

The map $\Mor_{(B_X, B_Y, \varepsilon)}(A)_t$ is linear by construction. To show $\{\Mor_{(B_X,B_Y, \varepsilon)}(A)_t\}$ form an $\varepsilon$-morphism we need to check they commute appropriately with the transition maps in $\Mod{X}$ and $\Mod{Y}$.

Let $s<t$  and $x\in X_s$. We can write $x$ uniquely as 
$$x= \sum_{\{j\mid \birth(x_j)\leq s<\death(x_j)\}} \lambda_j \phi_{\birth(x_j)}^{s}(x_j)$$ for some scalars $\lambda_j$. For $j$ where $s<\birth(x_j)\leq t$ set $\lambda_j=0$ which implies that $$\phi_s^t(x)=\sum_{\{j \mid \birth(x_j)\leq s<\death(x_j)\}} \lambda_j \phi_{\birth(x_j)}^{t}(x_j)=\sum_{\{j\mid \birth(x_j)\leq t<\death(x_j)\}} \lambda_j \phi_{\birth(x_j)}^{t}(x_j)$$ and that 
 
 \begin{equation}\label{eq:commute}
 \Mor_{(B_X,B_Y,\varepsilon)}(A)_t\circ \phi_s^t= \sum_{\{j\mid \birth(x_j)\leq t<\death(x_j)\}} \sum_{i} \lambda_j A(i,j) \psi_{\birth(y_i)}^{t +\varepsilon}(y_i).
 \end{equation}
At the same time 
\begin{align*}
\psi_{s+\varepsilon}^{t+\varepsilon}\circ \Mor_{(B_X, B_Y, \varepsilon)}(A)_s(x)=\sum_{\{j\mid \birth(x_j)\leq s<\death(x_j)\}}\sum_{i} \lambda_j A(i,j)  \psi_{\birth(y_i)}^{t +\varepsilon}(y_i)
\end{align*}
which agrees with \eqref{eq:commute} as $\lambda_j=0$ whenever where $s<\birth(x_j)\leq t$.
\end{proof}

\begin{proof}[Proof of Lemma \ref{lem:kernel}] 
The linearity of $\Mor_{(B_X,B_Y, \varepsilon)}$ follows directly from the definition.

To characterise the kernel first suppose that there exists an $(i, j)$ such that $A(i,j)\neq 0$ and  $\birth(x_j)+\varepsilon <\death(y_i)$.
Then $\Mor_{(B_X,B_Y, \varepsilon)}(A)_{\birth(x_{j})}(x_{j})=\sum_i A(i,j) \psi_{\birth(y_i)}^{\birth(x_j)+\varepsilon}(y_i)$ will have a non-zero coefficient for $\psi_{\birth(y_{i})}^{\birth(x_{j})+\varepsilon}(y_{i})$ and hence $\Mor_{(B_X,B_Y, \varepsilon)}(A)$ is a non-zero morphism and $A\notin \ker(\Mor_{(B_X,B_Y, \varepsilon)})$.

Now suppose that
$A(i,j)= 0$ for all $i,j$ where $\birth(x_j)+\varepsilon <\death(y_i)$.  
That is $A(i,j)\neq 0$ implies $\birth(x_j)+\varepsilon \geq \death(y_i)$.  For arbitrary $t$ let $x\in X_t$ and $\lambda_j$ such that $x=\sum_{\{j\mid \birth(x_j)\leq t<\death(x_j)\}}\lambda_j \phi_{\birth(x_j)}^{t}(x_j)$. We have
 $$\Mor_{(B_X,B_Y, \varepsilon)}(A)_t(x)=\sum_{\{i,j\mid \birth(x_j)\leq t<\death(x_j) \}} \lambda_j A(i,j) \psi_{\birth(y_i)}^{t +\varepsilon}(y_i).$$
 If $A(i,j)\neq 0$ then $\birth(x_j)+\varepsilon \geq \death(y_i)$ and $\psi_{\birth(y_i)}^{\birth(x_j) +\varepsilon}(y_i)=0$ so 
 $$\psi_{\birth(y_i)}^{t +\varepsilon}(y_i)= \psi_{\birth(x_j)+\varepsilon}^{t +\varepsilon}\left(\psi_{\birth(y_i)}^{\birth(x_j) +\varepsilon}(y_i)\right) =0.$$

 As every summand is zero, $\Mor_{(B_X,B_Y, \varepsilon)}(A)_t(x)=0$. Since this holds for all $t$ and all $x\in X_t$ we have $A\in \ker(\Mor_{(B_X,B_Y, \varepsilon)})$.
\end{proof}

\begin{proof}[Proof of Lemma \ref{lem:Mat(Mor)}]

Let $\alpha$ be an $\varepsilon$-morphism in $\MorC(\Mod{X}, \Mod{Y}, \varepsilon)$ and let $A = \Mat_{B_X}^{B_Y}(\alpha)$. 
To show $$\Mor_{(B_X,B_Y, \varepsilon)}(\Mat_{B_X}^{B_Y}(\alpha)) = \alpha$$ it is sufficient to show that they agree on the basis elements of $\Mod{X}$. 

Let $x_j\in X_{\birth(x_j)}$ be a basis element. We have
$\Mor_{(B_X,B_Y, \varepsilon)}(A)_{\birth(x_j)}(x_j) =  \sum_i A(i,j) \psi_{\birth(y_i)}^{\birth(x_j)+\varepsilon}(y_i)$.
But the definition of $A$ implies that $ \sum_i A(i,j) \psi_{\birth(y_i)}^{\birth(x_j)+\varepsilon}(y_i)=\alpha_{\birth(x_j)}(x_j)$. Together we conclude $\Mor_{(B_X,B_Y, \varepsilon)}(A)_{\birth(x_j)}(x_j) =\alpha_{\birth(x_j)}(x_j)$. Since this holds for all basis elements of $\Mod{X}$ we have  $$\Mor_{(B_X,B_Y, \varepsilon)}(\Mat_{B_X}^{B_Y}(\alpha)) = \alpha.$$

We next want to show that $\Mat_{B_X}^{B_Y}(\Mor_{(B_X,B_Y, \varepsilon)}(A))=\pi_{(\Mod{X}, \Mod{Y}, \varepsilon)}(A)$ for all $A\in  \MatC(\Mod{X}, \Mod{Y}, \varepsilon)$. 
We can view $\MatC(\Mod{X}, \Mod{Y}, \varepsilon)$ as a vector space with basis containing matrices of the form $1_{kl}$ (which are $1$ in the $(k,l)$ entry and $0$ elsewhere) for $(k,l)$ such that $\birth(y_k)\leq \birth(x_l)+\varepsilon$ and  $\death(y_k)\leq \death(x_l)+\varepsilon$ as $A\in \MatC(\Mod{X}, \Mod{Y}, \varepsilon)$.
As $\Mor_{(B_X,B_Y, \varepsilon)}$, $\Mat_{B_X}^{B_Y}$ and $\pi_{(\Mod{X}, \Mod{Y}, \varepsilon)}$ are all linear, it is sufficient to show this equation holds for the basis elements $1_{kl}$.

Suppose $\birth(y_k)\leq \birth(x_l)+\varepsilon$ and  $\death(y_k)\leq \death(x_l)+\varepsilon$  and set $\alpha^{(k,l)}=\Mor_{(B_X,B_Y, \varepsilon)}(1_{kl})$. For a basis element $x_j$, we have $$\alpha^{(k,l)}_{\birth(x_j)}(x_j) = \sum_{\{i \mid \birth(y_i)\leq \birth(x_j)+\varepsilon\}}1_{kl}(i,j) \psi_{\birth(y_i)}^{\birth(x_j)+\varepsilon}(y_i)=\sum_{\{i \mid \birth(y_i)\leq \birth(x_j)+\varepsilon\}} \delta_{ik}\delta_{jl} \psi_{\birth(y_i)}^{\birth(x_j)+\varepsilon}(y_i)$$
where $\delta_{**}$ is the Kronecker delta.
From this equation we read off  $\alpha^{(k,l)}_{\birth(x_l)}(x_l)=\psi_{\birth(y_k)}^{\birth(x_l)+\varepsilon}(y_k)$, and 
 $\alpha^{(k,l)}_{\birth(x_j)}(x_j)=0$ for $j\neq l$.

By definition of $\Mat_{B_X}^{B_Y}$, $\Mat_{B_X}^{B_Y}(\alpha^{(k,l)})(i,j)$ is the coefficient of $\psi_{\birth(y_i)}^{\birth(x_j)+\varepsilon}(y_i)$ in $\alpha^{(k,l)}_{\birth(x_j)}(x_j)$ if $\birth(x_j) + \varepsilon \in [\birth(y_i), \death(y_i))$, and $0$ otherwise. If $\birth(x_l) + \varepsilon < \death(y_k)$ then $\psi_{\birth(y_k)}^{\birth(x_l)+\varepsilon}(y_k)\neq 0$ and this implies $\Mat_{B_X}^{B_Y}(\alpha^{(k,l)})=1_{kl}$. If $\birth(x_l) + \varepsilon \geq \death(y_k)$ then $\alpha^{(k,l)}$ is the $0$ morphism. In either case $\Mat_{B_X}^{B_Y}(\Mor_{(B_X,B_Y, \varepsilon)}(1_{kl}))=\pi_{(\Mod{X},\Mod{Y}, \varepsilon)}(1_{kl})$.

The observation $\Mat_{B_X}^{B_Y}(\Mor_{(B_X,B_Y, \varepsilon)}(A))=A$ for all $A\in \ObMat(\Mod{X}, \Mod{Y}, \varepsilon)$ then follows directly from the definition of $\pi_{(\Mod{X},\Mod{Y}, \varepsilon)}$.
\end{proof}

 
 \begin{proof}[Proof of Lemma \ref{lem:compositionMor}]
 If $QP(i,j)\neq 0$ then there exists $k$ such that $Q(i,k)\neq 0$ (so $\birth(z_i)\leq \birth(y_k)+\delta$ and $\death(z_i)\leq \death(y_k)+\delta$) and $P(k,j)\neq 0$ (so $\birth(y_k)\leq \birth(x_j)+\varepsilon$ and $\death(y_k)\leq \death(x_j)+\varepsilon$). Together they imply  $\birth(z_i)\leq \birth(x_j)+\delta+\varepsilon$ and $\death(z_i)\leq \death(x_j)+\delta+\varepsilon$. This holds for all $(i,j)$ with $QP(i,j)\neq 0$ so $QP\in \MatC(\Mod{X}, \Mod{Z}, \varepsilon+\delta)$.
  
 Let $\psi$ denote the transition maps in $\Mod{Y}$ and $\phi$ the transition maps in $\Mod{Z}$. For each basis element $x_j\in B_X$ we have
\begin{align*}
\Mor_{(B_Y,B_Z, \delta)}&(Q)_{\birth(x_j)+\varepsilon}(\Mor_{(B_X, B_Y, \varepsilon)}(P)_{\birth(x_j)}(x_j))\\
&=\Mor_{(B_Y,B_Z,\delta)}(Q)_{\birth(x_j)+\varepsilon}\left(\sum_{\{i \mid \birth(y_i)\leq \birth(x_j)+\varepsilon\}} P(i,j)  \psi_{\birth(y_i)}^{\birth(x_j)+\varepsilon}(y_i)\right)\\
&=\sum_{\{i \mid \birth(y_i)\leq \birth(x_j)+\varepsilon\}}P(i,j)  \Mor_{(B_Y, B_Z,\delta)}(Q)_{\birth(x_j)+\varepsilon}\left( \psi_{\birth(y_i)}^{\birth(x_j)+\varepsilon}(y_i)\right)\\
&=\sum_{\{i \mid \birth(y_i)\leq \birth(x_j)+\varepsilon\}} P(i,j) \phi_{\birth(y_i)+\delta}^{\birth(x_j)+\varepsilon+\delta}\left(\Mor_{(B_Y, B_Z, \delta)}(Q)_{\birth(y_i)}(y_i)\right)\\
&=\sum_{\{i \mid \birth(y_i)\leq \birth(x_j)+\varepsilon\}}P(i,j) \phi_{\birth(y_i)+\delta}^{\birth(x_j)+\varepsilon+\delta}\left(\sum_{\{ k\mid \birth(z_k)\leq \birth(y_i)+\delta\}} Q(k,i) \phi_{\birth(z_k)}^{\birth(y_i)+\delta}(z_k)\right)\\
&= \sum_{\{i,k\mid \birth(z_k)\leq \birth(y_i)+\delta\leq \birth(x_j)+\varepsilon+\delta\}} Q(k,i)P(i,j)   \phi_{\birth(z_k)}^{\birth(x_j)+\varepsilon+\delta}(z_k)
\end{align*}
As $P(i,j)=0$ for all $(i,j)$ with $\birth(y_i)> \birth(x_j)+\varepsilon$ and $Q(k,i)=0$ for all $(k,i)$ with $\birth(z_k)> \birth(y_i)+\delta$  we can rewrite this sum as 
$$\Mor_{(B_Y,B_Z, \delta)}(Q)_{\birth(x_j)+\varepsilon}\left(\Mor_{(B_X, B_Y, \varepsilon)}(P)_{\birth(x_j)}(x_j)\right)=\sum_{\{k \mid \birth(z_k)\leq \birth(x_j)+\varepsilon+\delta\}} (QP)(k,j)   \phi_{\birth(z_k)}^{\birth(x_j)+\varepsilon+\delta}(z_k).$$
As this holds for all basis elements $x_j$ we conclude that 
$$\Mor_{(B_X, B_Z, \varepsilon +\delta)}(QP)=\Mor_{(B_Y,B_Z, \delta)}(Q)\circ \Mor_{(B_X, B_Y, \varepsilon)}(P).$$
\end{proof}

\section{Basis transformations}\label{app:basis}


\begin{proof}[Proof of Lemma \ref{lemma:newbasis}]
Fix a basis $B=\{y_1, \ldots y_N\}$ for $\Mod{Y}$ and $J=\{j\mid b_j<d_j\}$ be the subset of indices that determine the non-zero basis elements (that is $\{y_j\mid j\in J\}$ is a non-extended  basis of $\Mod{Y}$.

Let $P\in \BTM(\Mod{Y})$. Set $y^{new}_j=\sum P(i,j) \phi_{\birth(y_i)}^{\birth(y_j)}(y_i)$ which is well defined as $\birth(y_i)>\birth(y_j)$ implies $P(i,j)=0$ by the definition of basis transformation matrix.
By construction $y^{new}_j$ is an element of $Y_{\birth(y_j)}$ as each of the $\phi_{\birth(y_i)}^{\birth(y_j)}(y_i)\in Y_{\birth(y_j)}$.

Consider $j\notin J$ so $y_j$ is a zero element. As $P\in \BTM(\Mod{Y})$ the only possible non-zero element in column $j$ is $P(j,j)$. This implies that $y_j^{new}=P(j,j)y_j$ must be a zero element. 

Now fix $j\in J$. We want to show that $\birth(y_j^{new})=\birth(y_j)$.  As $y_j^{new}\in Y_{\birth(y_j)}$ (by construction) we know $\birth(y_j^{new})\leq \birth(y_j)$. 
Suppose that 
$y^{new}_j \in \im(\phi_{t}^{\birth(y_j)})$ for some $t<\birth(y_j)$. 
As $B$ is a basis of  $\Mod{Y}$ there exists $\lambda_k$ such that $$y^{new}_j=\phi_t^{\birth(y_j)}\left(\sum_{\{k\mid \birth(y_k)\leq t<\death(y_k)\}} \lambda_k \phi_{\birth(y_k)}^{t}(y_k)\right)= \sum_{\{k\mid \birth(y_k)\leq t<\death(y_k)\}} \lambda_k \phi_{\birth(y_k)}^{\birth(y_j)}(y_k).$$
Equating with our definition of $y_j^{new}$ we have 
\begin{equation}\label{eq:birthtimesame}
\sum_{\{k\mid \birth(y_k)\leq t<\death(y_k)\}} \lambda_k \phi_{\birth(y_k)}^{\birth(y_j)}(y_k)=\sum_{\{i\mid \birth(y_i)\leq \birth(y_j)<\death(y_i)\}} P(i,j) \phi_{\birth(y_i)}^{\birth(y_j)}(y_i).
\end{equation}
Recall that $\{\phi_{\birth(y_i)}^{\birth(y_j)}(y_i)\mid \birth(y_i)\leq \birth(y_j)<\death(y_i)\}$ form a basis of $Y_{\birth(y_j)}$. As $P(j,j)\neq 0$, we have a non-zero summand of 
$P(j,j)\phi_{\birth(y_j)}^{\birth(y_j)}(y_j)$ on the right hand side of  \eqref{eq:birthtimesame} which cannot appear on the left hand side as $\birth(y_j)>t$. This creates a contradiction, so we can conclude that $\birth(y_j^{new})=\birth(y_j)$.

We now want to show that $\death(y_j^{new})=\death(y_j)$. As $P\in \BTM(\Mod{Y})$, if $P(i,j)\neq 0$ then $\death(y_i)\leq \death(y_j)$ and thus $\phi_{\birth(y_i)}^{\death(y_j)} (y_i)=0$. Plugging in the formula for $y_j^{new}$ we have
$$\phi_{\birth(y_j)}^{\death(y_j)} (y_j^{new})=\sum_{\{i \mid \birth(y_i)\leq \birth(y_j)\}} P(i,j) \phi_{\birth(y_i)}^{\death(y_j)}(y_i)=0$$
which implies that $\death(y_j^{new})\leq \death(y_j)$.

Suppose that there exists $t\in [\birth(y_j), \death(y_j))$ with $\phi_{\birth(y_j)}^t(y_j^{new})=0$. 
This implies that $$\sum_{\{i \mid \birth(y_i)\leq \birth(y_j)\}} P(i,j) \phi_{\birth(y_i)}^t(y_i)=0.$$ As $B$ is a basis, and $P(j,j)\neq 0$ this in turn implies that $\phi_{\birth(y_j)}^t(y_j)=0$ which contradicts $t<\death(y_j)$. Thus by contradiction we conclude that $\death(y_j^{new})=\death(y_j)$.

We need to show that the set $\{\phi_{\birth(y_k)}^{t}(y_k^{new}) \mid \birth(y_k)\leq t<\death(y_k)\}$ is a basis of $Y_t$. Without loss of generality, rearrange the order of the indices and the corresponding rows and columns within $P$ such that $j\leq i$ whenever $\Mod{I}[b_j,d_j)\leq \Mod{I}[b_i,d_i)$ (for instance order by birth value and break ties by death value).
Our assumptions on the entries $P(i,j)$ imply that $P$ is an upper triangular matrix with non-zero diagonal entries for $j\in J$. 

Let $J_t=\{i\mid \birth(y_i)\leq t<\death(y_i)\}$ and note $J_t\subseteq J$. Set $P_t$ to be the matrix $P$ restricted to the columns and rows with indices in $J_t$. As $P_t$ must also be upper triangular with non-zero diagonal entries it is invertible. 
For $j\in J_t$, let $x_j=\phi_{\birth(y_j)}^t(y_j)$ and $x_j^{new}=\phi_{\birth(y_j)}^t(y_j^{new})$. As vectors in $Y_t$ we have the coordinate vector of $x_j^{new}$ in terms of $x_j$ the $j$th column of $P_t$ for each $j\in J_t$. Since $\{x_j\mid j\in J_t\}$ is a basis of $Y_t$ and $P_t$ is invertible we also have $\{x_j^{new}\mid j\in J_t\}$ is a basis of  $Y_t$.
\end{proof}

\begin{proof}[Proof of Lemma \ref{lem:transformed_morphism}]

$\Mat_{B_X}^{B_Y}(\alpha)=PA$
so for each $j$

\begin{align*}
\alpha_{\birth(x_j)}(x_j)&= \sum_{\{i \mid \birth(y_i)\leq \birth(x_j)+\varepsilon\}} PA(i,j)\phi_{\birth(y_i)}^{\birth(x_j)+\varepsilon}(y_i)\\
&= \sum_{\{i \mid \birth(y_i)\leq \birth(x_j)+\varepsilon\}} \left(\sum _{\{k\mid \birth(y_i)\leq \birth(y_k)\}} P(i,k)A(k,j)\right)\phi_{\birth(y_i)}^{\birth(x_j)+\varepsilon}(y_i)\\
&= \sum_k \sum_{\{i \mid \birth(y_i)\leq \birth(x_j)+\varepsilon \text{ and } \birth(y_i)\leq \birth(y_k)\}} P(i,k)A(k,j) \phi_{\birth(y_i)}^{\birth(x_j)+\varepsilon}(y_i)\\
&= \sum_k A(k,j) \sum_{\{i \mid \birth(y_i)\leq \birth(x_j)+\varepsilon \text{ and } \birth(y_i)\leq \birth(y_k)\}} P(i,k) \phi_{\birth(y_i)}^{\birth(x_j)+\varepsilon}(y_i)\\
&= \sum_k A(k,j)\phi_{\birth(y_k)}^{\birth(x_j)+\varepsilon}(y_k^{new})\\
&=\Mor_{(B_X, B_Y^{new},\varepsilon)}(A)_{\birth(x_j)}(x_j)
\end{align*}

We then can use  Lemma \ref{lem:Mat(Mor)} to say $$\Mat_{B_X}^{B_Y^{new}}(\alpha)=\Mat_{B_X}^{B_Y^{new}}\left(\Mor_{(B_X, B_Y^{new}, \varepsilon)}(A)\right)=\pi_{(\Mod{X}, \Mod{Y}, \varepsilon)}(A).$$  
\end{proof}

\begin{proof}[Proof of Lemma \ref{lem:transformed_morphism_domain}]
Write $A=\Mat_{B_X}^{B_Y}(\alpha)$ and recall $x_k^{new}=\sum_i P(i,k)\psi_{\birth(x_i)}^{\birth(x_k)}(x_i)$,
where $\psi$ are the transition maps of $\Mod{X}$; note $\birth(x_k^{new})=\birth(x_k)$.
Since $P\in\BTM(\Mod{X})$, each summand with $P(i,k)\neq 0$ has
$\birth(x_i)\leq\birth(x_k)$, so the commutativity of the $\varepsilon$-morphism $\alpha$ applies. For each $k$,
\begin{align*}
\alpha_{\birth(x_k)}(x_k^{new})
&= \sum_i P(i,k)\,\alpha_{\birth(x_k)}\!\left(\psi_{\birth(x_i)}^{\birth(x_k)}(x_i)\right)\\
&= \sum_i P(i,k)\,\phi_{\birth(x_i)+\varepsilon}^{\birth(x_k)+\varepsilon}\!\left(\alpha_{\birth(x_i)}(x_i)\right)\\
&= \sum_i P(i,k)\,\phi_{\birth(x_i)+\varepsilon}^{\birth(x_k)+\varepsilon}\!\left(
      \sum_{\{l\mid \birth(y_l)\leq \birth(x_i)+\varepsilon\}} A(l,i)\,\phi_{\birth(y_l)}^{\birth(x_i)+\varepsilon}(y_l)\right)\\
&= \sum_i \sum_{\{l\mid \birth(y_l)\leq \birth(x_i)+\varepsilon\}} P(i,k)A(l,i)\,\phi_{\birth(y_l)}^{\birth(x_k)+\varepsilon}(y_l)\\
&= \sum_l \left(\sum_i A(l,i)P(i,k)\right)\phi_{\birth(y_l)}^{\birth(x_k)+\varepsilon}(y_l)\\
&= \sum_l (AP)(l,k)\,\phi_{\birth(y_l)}^{\birth(x_k)+\varepsilon}(y_l)\\
&= \Mor_{(B_X^{new}, B_Y,\varepsilon)}(AP)_{\birth(x_k)}(x_k^{new}).
\end{align*}
As this holds on every basis element of $B_X^{new}$, we can use Lemma \ref{lem:Mat(Mor)} to conclude
$$\Mat_{B_X^{new}}^{B_Y}(\alpha)=\Mat_{B_X^{new}}^{B_Y}\!\left(\Mor_{(B_X^{new},B_Y,\varepsilon)}(AP)\right)
=\pi_{(\Mod{X}, \Mod{Y}, \varepsilon)}(AP).$$
\end{proof}

\section{Segments of a vineyard module}
\label{app:segment}
\begin{proof}[Proof of Lemma \ref{lem:union}]
To avoid confusion for each interval $[A,B]\subset [T_0, T_1]$, and $s\in [A,B)$ let $c_{[A,B]}^f(s)=\inf\left\{\{t \in[s,B] \mid  t \text{ critical }\} \cup B\right\}$.
By construction 
$$c_{[T_0,T_1]}^f(s)=
\begin{cases}
c_{[T_0, S]}^f(s) & \text{ if $s\in [T_0, S)$ and there exists a critical time in $(s,S]$}\\
c_{[T_0, T_1]}^f(S) & \text{ if $s\in [T_0, S)$ and there are no critical times in $(s,S]$}\\
c_{[S,T_1]}^f(s) &\text{ if $s\in [S, T_1)$} 
\end{cases}$$

If $[s, c_{[T_0, T_1]}^f(s)]\subset [T_0, S]$ or  $[s, c_{[T_0, T_1]}^f(s)]\subset [S, T_1]$ then as the restriction of  $\{B_t\}$ to $[T_0, S]$ or $[S, T_1]$ is a forwards simplified family of bases for $\Vine{V}|_{[T_0, S]}$ and $\Vine{V}|_{[S, T_1]}$ respectively, we have for $\Mat_{B_s}^{B_t}(\alpha^{s\to t})=\pi_{(\Mod{V}_s, \Mod{V}_t, |s-t|)}(\Id)$ for all $t\in [s, c_f(s)]$ if $s$ is forwards compatible and $\pi_{(\Mod{V}_s, \Mod{V}_t, |s-t|)}(e^{\lambda}_{kl})$ for all $t\in [s, c^f_{[T_0, T_1]}(s)]$ and some $\lambda$ if $s$ is forwards incompatible with respect to vines $(\gamma_k, \gamma_l)$. 

If $[s, c_{[T_0, T_1]}^f(s)]$ is not a subset of either $[T_0, S]$ or  $[S, T_1]$ then $s\in [T_0, S)$ and there are no critical times in $(s,S]$. We have
$c_{[T_0,T_1]}^f(s)=c_{[T_0, T_1]}^f(S)$ and $c_{[T_0, S]}^f(s)=S$. Since $S$ is not critical it is forwards compatible, so $\Mat_{B_S}^{B_t}(\alpha^{S\to t})=\pi_{(\Mod{V}_s, \Mod{V}_t, |s-t|)}(\Id)$ for all $t\in [s, c_f(s)]$.

Suppose $s$ is forwards compatible. From $c_{[T_0, S]}^f(s)=S$ we get $\Mat_{B_s}^{B_t}(\alpha^{s\to t})=\pi_{(\Mod{V}_s, \Mod{V}_t, |s-t|)}(\Id)$ for all $t\in [s, S]$. For $t\in [S, c_{[T_0,T_1]}^f(s)]$ we have $$\Mat_{B_s}^{B_t}(\alpha^{s\to t})=\pi_{(\Mod{V}_s, \Mod{V}_t, |s-t|)}(\Mat_{B_S}^{B_t}(\alpha^{S\to t})\Mat_{B_s}^{B_S}(\alpha^{s\to S}))=\pi_{(\Mod{V}_s, \Mod{V}_t, |s-t|)}(\Id).$$

If $s$ is forwards incompatible with respect to vines $(\gamma_k, \gamma_l)$ then there exists $\lambda$  such that  $\pi_{(\Mod{V}_s, \Mod{V}_t, |s-t|)}(e^{\lambda}_{kl})$ for all $t\in [s, S]$
For $t\in [S, c_{[T_0, T_1]}^f(s)]$ we have 
$$\Mat_{B_s}^{B_t}(\alpha^{s\to t})=\pi_{(\Mod{V}_s, \Mod{V}_t, |s-t|)}(\Mat_{B_S}^{B_t}(\alpha^{S\to t})\Mat_{B_s}^{B_S}(\alpha^{s\to S}))=\pi_{(\Mod{V}_s, \Mod{V}_t, |s-t|)}(e^{\lambda}_{kl}).$$
\end{proof}

\begin{proof}[Proof of Lemma \ref{lem:subinterval}]
$ $

\noindent
\underline{(i) If $\Vine{V}|_{[S,T]}$ is a segment of Type 1 then $[s,t]$ is a segment of Type 1 for all $S\leq s\leq t\leq T$:}

For $\Vine{V}|_{[S,T]}$ to be a segment of Type 1 and $S<\hat{t} < T$ observe that the distance between elements in $\Crit(\hat{t})$ must be at least $$4|S-T|-\min\{2|\hat{t}-S|, 2|T-\hat{t}|\}>\max\{4|S-\hat{t}|, 4|T-\hat{t}|\}.$$

\noindent
\underline{(ii) If $\Vine{V}|_{[S,T]}$ is a segment of Type 2 then $[s,T]$ is a segment of Type 2 for all $S\leq s\leq T$:}

As $\Vine{V}|_{[S,T]}$ is a segment no critical times occur in $[S,T)$, so none appear in $[s, T)\subset [S, T)$. Since $T$ is left critical all of its distinct critical heights in $\Mod{V}_T$ are at least $4|S-T|$ apart, and $4|s-T|\leq 4|S-T|$. Thus by definition $\Vine{V}|_{[s,T]}$ is a segment of Type 2.

\noindent
\underline{(iii) If $\Vine{V}|_{[S,T]}$ is a segment of Type 3 then $\Vine{V}|_{[S,t]}$ is a segment of Type 3 for all $S\leq t\leq T$:}

Mirror the argument to (ii) using $\Vine{V}^{op}$.
\end{proof}

\begin{proof}[Proof of Lemma \ref{lem:existence}]
Our assumptions on our vineyards ensure that there are only finitely many critical times. By subdividing $[S,T]$, it is sufficient to prove in the case where there are no critical times in $(S,T)$.

For each $t\in (S, T)$ let $f(t)$ denote the smallest distance between values in $\Crit(t)$.  Let $\varepsilon(t)=\min\{f(t)/12, |t-S|, |t-T|\}$. As our vineyards are $1$-Lipschitz we can infer that $f(t)$ is positive over $(S,T)$ and is $2$-Lipschitz.  In particular for all $s$ within $\varepsilon(t)$ of $t$ we have $f(s)\geq f(t)-2\frac{f(t)}{12}=\frac{5f(t)}{6}$. For any $a,b$ with $t-\varepsilon(t)<a<b<t+\varepsilon(t)$ we have 
$$4|a-b|\leq 8\varepsilon(t)\leq 2f(t)/3$$
which is strictly less than both $f(a)$ and $f(b)$. We thus can conclude that $[a,b]$ is a segment (of Type 1).

 Let $f(S)$ be the minimum distance between distinct values in $\Crit^+(S)$ and set $\varepsilon(S)=\min\{f(S)/6, |S-T|\}$. We claim that for any $b\in [S,S + \varepsilon(S))$ we have $[S,b]$ is a segment. If $S$ is not right critical then $f(S)$ is the one-sided limit of $f(t)$ as $t\to S$ and we can apply the arguments above to show $|S-b|\leq \varepsilon(S)\leq f(S)/6$ and thus $[S,b]$ is a segment. If $S$ is right critical then by construction the smallest distance between the distinct values within $\Crit^+(S)$ is $f(S)>6\varepsilon(S)>4|S-b|$ and $[S,b]$ is a segment of Type 3.

We can analogously define $\varepsilon(T)$ such that $[a,T]$ is a segment for all $T-\varepsilon(T)<a<T$.
 
We can cover $[S,T]$ by open cover $U=\{(t-\varepsilon(t), t+ \varepsilon(t))\mid t\in (S,T)\} \cup \{[S, S+\varepsilon(S)), (T-\varepsilon(T), T]\}$. As $[S,T]$ is compact $U$ contains a finite subcover of the form $$\{(t_i - \varepsilon(t_i), t_i+\varepsilon(t_i))\} \cup \{[S, S+\varepsilon(S)), (T-\varepsilon(T), T]\}$$ with $S<t_1<t_2<\ldots t_m<T$.  We can further assume that $U$ is minimal, that is the removal of any cover element will result in not covering $[S,T]$. Note that $[S, S+\varepsilon(S))$ is the only set in $U=\{(t-\varepsilon(t), t+ \varepsilon(t))\mid t\in (S,T)\} \cup \{[S, S+\varepsilon(S)) ,(T-\varepsilon(T), T]\}$ that contains $S$ so $[S, S+\varepsilon(S))$ must be in $U$. Similarly $(T-\varepsilon(T), T]$ must also be in $U$.

As $U$ is minimal we can show that all triplewise intersections are empty. This implies we have a cover of the form shown in Figure \ref{fig:cover}. Choose $S_1<S_2< \ldots < S_N$ in the overlaps of adjacent cover elements so each $[S_i, S_{i+1}]$ is contained in an element of the cover. The choice of the $\varepsilon(t)$ implies each $[S_i, S_{i+1}]$ is a segment. The intervals $[S_i, S_{i+1}]$ form a partition of $[S,T]$ into finitely many segments.

\begin{figure}
\begin{tikzpicture}[scale=1.5]

\draw[very thick, black] (0,0) -- (8,0);
\fill[black] (0,0) circle (1.5pt);
\fill[black] (8,0) circle (1.5pt);
\node[below=3mm] at (0,0) {$S=S_0$};
\node[below=3mm] at (8,0) {$T=S_6$};

\fill[red, fill opacity=0.2] (1.5, 0) circle (1);
\draw[red, thick] (1.5, 0) circle (1);
\fill[red] (1.5, 0) circle (1.5pt);
\node[below=3mm, red] at (1.5, 0) {$t_1$};

\fill[green!60!black, fill opacity=0.2] (3.0, 0) circle (0.7);
\draw[green!60!black, thick] (3.0, 0) circle (0.7);
\fill[green!60!black] (3.0, 0) circle (1.5pt);
\node[below=3mm, green!60!black] at (3.0, 0) {$t_2$};

\fill[orange, fill opacity=0.2] (4.5, 0) circle (1.2);
\draw[orange, thick] (4.5, 0) circle (1.2);
\fill[orange] (4.5, 0) circle (1.5pt);
\node[below=3mm, orange] at (4.5, 0) {$t_3$};

\fill[purple, fill opacity=0.2] (6.5, 0) circle (1);
\draw[purple, thick] (6.5, 0) circle (1);
\fill[purple] (6.5, 0) circle (1.5pt);
\node[below=3mm, purple] at (6.5, 0) {$t_4$};

\begin{scope}
\clip (0,-1) rectangle (0.8,1);
\fill[blue, fill opacity=0.2] (0, 0) circle (0.8);
\draw[blue, thick] (0, 0) circle (0.8);
\end{scope}
\draw[blue, thick] (0, -0.8) -- (0, 0.8);

\begin{scope}
\clip (7.3,-1) rectangle (8,1);
\fill[brown, fill opacity=0.2] (8, 0) circle (0.7);
\draw[brown, thick] (8, 0) circle (0.7);
\end{scope}
\draw[brown, thick] (8, -0.7) -- (8, 0.7);

\fill[black] (0.7, 0) circle (2pt);
\node[above=3mm] at (0.7, 0) {$S_1$};

\fill[black] (2.4, 0) circle (2pt);
\node[above=3mm] at (2.4, 0) {$S_2$};

\fill[black] (3.5, 0) circle (2pt);
\node[above=3mm] at (3.5, 0) {$S_3$};

\fill[black] (5.6, 0) circle (2pt);
\node[above=3mm] at (5.6, 0) {$S_4$};

\fill[black] (7.4, 0) circle (2pt);
\node[above=3mm] at (7.4, 0) {$S_5$};

\end{tikzpicture}
\caption{A minimal cover of the interval $[S,T]$ by open intervals of the form $\{(t_i- \varepsilon(t_i), t_i+\varepsilon(t_i))\} \cup \{[S, S+\varepsilon(S)), (T-\varepsilon(T), T]\}$ with $S=S_0<S_1<S_2<S_3 <S_4<S_5<S_6=T$. For visual clarity, balls (and half-balls) in $\R^2$ are used to represent intervals in $\R$. The $S_i$ are chosen in the intersections of adjacent balls, so each $[S_i, S_{i+1}]$ is contained in an element of the cover. The choice of the radii implies each $[S_i, S_{i+1}]$ is a segment.}\label{fig:cover}
\end{figure}

\end{proof}

\begin{proof}[Proof of Lemma  \ref{lem:refinement}]
Let  $\{\tilde{S}_l\}_{l=1}^M=\{S_i^1\}\cup \{S_j^2\}$ reindexed in increasing order. For each $l$ Lemma \ref{lem:existence} implies there exists $\{\tilde{S}_{l,n}\}_{n=1}^{K_l}$ such that  $\{\tilde{S}_{l,j}\}$ is a partition of $\Vine{V}|_{[\tilde{S}_l, \tilde{S}_{l+1}]}$ into segments. Set $\{\hat{S}_k\}=\cup_{l=1}^{M-1} \{\tilde{S}_{l,n}\}$. Each of the $\Vine{V}|_{[\tilde{S}_{l,n}, \tilde{S}_{l, n+1}]}$ is a segment and their interiors are disjoint by construction (as we reindexed  the $\tilde{S}_l$ to be in increasing order).
\end{proof}

\begin{proof}[Proof of Lemma \ref{lem:sequence}]
Define the function $f:[S, T)\to [0, \infty)$ by $f(t)$ as the smallest distance between any pair of critical values in $\Mod{V}_t$. Note that $f(t)>0$ for all $t\in [S,T)$ and that $\lim_{t\to T^-}f(t)=0$, and $f$ is $2$-Lipschitz as the vineyard is assumed to be $1$-Lipschitz. In particular this implies $0<f(t)\leq 2|T-t|$ and for any $t\in [S,T)$ we have $t+f(t)/16 \in [S,T)$.

Construct the strictly increasing sequence $\{s_n\}\subset [S, T)$ with $s_0:=S$ and $s_{n}:=s_{n-1}+ f(s_{n-1})/16$. As $\{s_n\}$ is a bounded increasing sequence it must converge to some limit which we will denote $L\in[S,T]$. Suppose that $L<T$ which by assumption implies $f(L)>0$. Choose $k$ such that $s_k\in (L-f(L)/32, L]$. Since $|f(s_k)-f(L)|<2|L-s_k|$ we have $f(s_k)>\frac{15}{16}f(L)$. 
This implies that
 $$s_{k+1}=s_k+\frac{f(s_k)}{16}> \left(L-\frac{f(L)}{32}\right)+\frac{\frac{15}{16}f(L)}{16}=L + \frac{7}{16^2}f(L)>L$$ which is a contradiction as $\{s_n\}$ is increasing. We conclude that $\lim_{n\to \infty} s_n= T$. Thus every $s\in (S,T]$ will satisfy $s\in (s_n, s_{n+1}]$ for some $n$. 
 
For a subinterval $[a,b]\subset [S,T]$ to be a segment of Type 1 we need $f(a)>4|b-a|$ and $f(b)>4|b-a|$. As $f$ is $2$-Lipschitz  it will be sufficient to show that $f(a)>6|b-a|$.
 
Consider the interval $[s_n, s_{n+1}]$. By construction $s_{n+1}$ is chosen so $f(s_n)=16|s_{n+1}-s_n|$ so that $[s_n, s_{n+1}]$ satisfies the definition of a segment of Type 1.

Consider the interval $[s_{n-1}, s_{n+1}]$. As $f$ is $2$-Lipschitz, we know 
$$f(s_n)\leq f(s_{n-1})+2|s_{n-1} - s_n|=\frac{9}{8}f(s_{n-1})$$ and so 
$|s_{n+1}-s_{n}|=\frac{f(s_n)}{16} \leq \frac{9}{8\cdot 16}f(s_{n-1})$. We thus can bound
$$6|s_{n+1}-s_{n-1}|=6|s_{n+1}-s_n|+6|s_n-s_{n-1}| \leq (\frac{6\cdot 9}{8\cdot 16}+\frac{6}{16})f(s_{n-1})<f(s_{n-1})$$ so that $[s_{n-1}, s_{n+1}]$ satisfies the definition of a segment of Type 1.

\end{proof}


\begin{proof}[Proof of Lemma \ref{lem:IdIFF}]
Note that by definition $\Mat_{B_S}^{B_T}(\alpha^{S\to T})$ and $\Mat_{B_T}^{B_S}(\beta^{T\to S})$ must both be zero on every row and column with index outside $J$. Let $\Id_J$ be the diagonal matrix with $\Id_J(j,j)=1$ for $j\in J$ and $\Id_J(j,j)=0$ for $j\notin J$. Without loss of generality assume $[S,T]$ is a segment of Type 1 or 2 (as otherwise we could consider the reverse direction).

Suppose $\Mat_{B_S}^{B_T}(\alpha^{S\to T})=\pi_{(\Mod{V}_S, \Mod{V}_T, |S-T|)}(\Id)$. 
 
\begin{align*}
\Mat_{B_T}^{B_S}(\beta^{T\to S})&=\pi_{(\Mod{V_T}, \Mod{V_S}, |S-T|)}\left(\Mat_{B_T}^{B_S}(\beta^{T\to S})\right)\\
 &=\pi_{(\Mod{V_T}, \Mod{V_T}, 2|S-T|)}\left( \Mat_{B_S}^{B_T}(\alpha^{S\to T}) \Mat_{B_T}^{B_S}(\beta^{T\to S})\right) \\
 &= \Mat_{B_T}^{B_T}(\alpha^{S\to T}\circ \beta^{T\to S}) \\ 
  &=\Id_J= \pi_{(\Mod{V}_T, \Mod{V}_S, |S-T|)}(\Id)
\end{align*}

The proof that $\Mat_{B_T}^{B_S}(\beta^{T\to S})=\pi_{(\Mod{V}_T, \Mod{V}_S, |S-T|)}(\Id)$ implies $\Mat_{B_S}^{B_T}(\alpha^{S\to T})=\pi_{(\Mod{V}_S, \Mod{V}_T, |S-T|)}(\Id)$ is analogous, with the main calculation
$$\Mat_{B_S}^{B_T}(\alpha^{S\to T})=\pi_{(\Mod{V_T}, \Mod{V_T}, 2|S-T|)}\left( \Mat_{B_S}^{B_T}(\alpha^{S\to T}) \Mat_{B_T}^{B_S}(\beta^{T\to S})\right) = \pi_{(\Mod{V}_S, \Mod{V}_T, |S-T|)}(\Id).$$

\end{proof}

\begin{proof}[Proof of Lemma \ref{lem:lambdaIFF}]

By symmetry it is sufficient to consider that case where $S$ is forwards incompatible. 
From construction $e^{\lambda}_{kl}$ and $e^{-\lambda}_{kl}$ are both in $\BTM(\Mod{V}_S)$.  Consider the basis $\hat{B}_S=e^{-\lambda}_{kl}(B_S)$. Observe that $B_S=e^{\lambda}_{kl}(\hat{B}_S)$.

If $\Mat_{B_S}^{B_T}(\alpha^{S\to T})=\pi_{(\Mod{V}_S, \Mod{V}_T, |S-T|)}(e^{\lambda}_{kl})$ then by Lemma \ref{lem:transformed_morphism_domain} 
$$\Mat_{\hat{B}_S}^{B_T}(\alpha^{S\to T})=\pi_{(\Mod{V}_S, \Mod{V}_T, |S-T|)}(e^{\lambda}_{kl}e^{-\lambda}_{kl})=\pi_{(\Mod{V}_S, \Mod{V}_T, |S-T|)}(\Id).$$ 

By Lemma \ref{lem:IdIFF} we know $\Mat_{B_T}^{\hat{B}_S}(\beta^{T\to S})=\pi_{(\Mod{V}_T, \Mod{V}_S, |S-T|)}(\Id)=\pi_{(\Mod{V}_T, \Mod{V}_S, |S-T|)}(e^{\lambda}_{kl}e^{-\lambda}_{kl})$. Applying Lemma \ref{lem:transformed_morphism}
$$\Mat_{B_T}^{B_S}(\beta^{T\to S})= \Mat_{B_T}^{e^{\lambda}_{kl}(\hat{B}_S)}(\beta^{T\to S})=\pi_{(\Mod{V}_T, \Mod{V}_S, |S-T|)}(e^{-\lambda}_{kl}).$$

If $\Mat_{B_T}^{B_S}(\beta^{T\to S})=\pi_{(\Mod{V}_T, \Mod{V}_S, |S-T|)}(e^{-\lambda}_{kl})$ then rewriting $B_S=e^{\lambda}_{kl}(\hat{B}_S)$ we have from Lemma \ref{lem:transformed_morphism} that
$$\Mat_{B_T}^{\hat{B}_S}(\beta^{T\to S})=\pi_{(\Mod{V}_T, \Mod{V}_S, |S-T|)}(e^{\lambda}_{kl}e^{-\lambda}_{kl})=\pi_{(\Mod{V}_T, \Mod{V}_S, |S-T|)}(\Id).$$ 
By Lemma \ref{lem:IdIFF} we know $\Mat_{\hat{B}_S}^{B_T}(\alpha^{S\to T})=\pi_{(\Mod{V}_S, \Mod{V}_T, |S-T|)}(\Id)=\pi_{(\Mod{V}_S, \Mod{V}_T, |S-T|)}(e^{\lambda}_{kl}e^{-\lambda}_{kl})$. Again applying Lemma \ref{lem:transformed_morphism_domain} we have 
$\Mat_{B_S}^{B_T}(\alpha^{S\to T})=\pi_{(\Mod{V}_S, \Mod{V}_T, |S-T|)}(e^{\lambda}_{kl})$.

 \end{proof}

\begin{proof}[Proof of Lemma \ref{lem:aliveVines}]
We can write
\begin{equation}\label{eq:interleaving1}
\pi_{(\Mod{V_S}, \Mod{V_S}, 2|S-T|)}\left( \Mat_{B_T}^{B_S}(\beta^{T\to S}) \Mat_{B_S}^{B_T}(\alpha^{S\to T})\right)=\Mat_{B_S}^{B_S}(\beta^{T\to S} \circ \alpha^{S\to T}).
\end{equation}
Since $\alpha$ and $\beta$ form an $|S-T|$-interleaving, the map $\beta^{T\to S}\circ \alpha^{S\to T}$ is the $2|S-T|$-morphism from $\Mod{V}_S$ to itself corresponding to transitions of $2|S-T|$ and thus
\begin{align*}
\Mat_{B_S}^{B_S}(\beta^{T\to S} \circ \alpha^{S\to T})=\pi_{(\Mod{V_S}, \Mod{V_S}, 2|S-T|)}(\Id).
\end{align*}

Fix $j$ such that $\birth(\gamma_j(S))<\death(\gamma_j(S))$ and $\birth(\gamma_j(T))<\death(\gamma_j(T))$. 
Regardless of which type of segment, we can show that $ \birth(\gamma_j(S))< \death(\gamma_j(S))-2|S-T|$. If $[S,T]$ is a segment of Type 1 or 3 then this follows directly from the definition as we have assumed $\birth(\gamma_j(S))<\death(\gamma_j(S))$ so $\birth(\gamma_j(S))<\death(\gamma_j(S))-4|S-T|$. If $[S,T]$ is of Type 2 then since $\birth(\gamma_j(T))<\death(\gamma_j(T))$ we have $|\death(\gamma_j(T))- \birth(\gamma_j(T))|>4|T-S|$ and since the birth and death heights are $1$-Lipschitz we have $$|\death(\gamma_j(S))- \birth(\gamma_j(S))|>|\death(\gamma_j(T))- \birth(\gamma_j(T))|-2|T-S|>2|T-S|.$$

From  $\birth(\gamma_j(S))\leq \death(\gamma_j(S))-2|S-T|$ we can infer $\pi_{(\Mod{V_S}, \Mod{V_S}, 2|S-T|)}(\Id)(j,j)=1$. From \eqref{eq:interleaving1}, the $(j,j)$ entry of $\Mat_{B_T}^{B_S}(\beta^{T\to S}) \Mat_{B_S}^{B_T}(\alpha^{S\to T})$ must be $1$. That is,
\begin{equation}\label{eq:(j,j)}
\sum_{k}\Mat_{B_T}^{B_S}(\beta^{T\to S})(j,k) \Mat_{B_S}^{B_T}(\alpha^{S\to T})(k,j)=1.
\end{equation}
There must be some $k$ with both $\Mat_{B_S}^{B_T}(\alpha^{S\to T})(j,k)$ and  $\Mat_{B_T}^{B_S}(\beta^{T\to S})(k,j)$ are non-zero.  
This implies $\birth(\gamma_j(T))\leq \birth(\gamma_k(S)) +|S-T|$ and $\birth(\gamma_k(S))\leq \birth(\gamma_j(T)) +|S-T|$ which together force $|\birth(\gamma_j(T))-\birth(\gamma_k(T))|\leq 2|S-T|$. Similarly we also can show $|\death(\gamma_j(T))-\death(\gamma_k(T))|\leq 2|S-T|$. As $[S,T]$ is a segment this is only possible if $k=j$. Thus \eqref{eq:(j,j)} becomes $$\Mat_{B_S}^{B_T}(\alpha^{S\to T})(j,j)\Mat_{B_T}^{B_S}(\beta^{T\to S})(j,j)=1.$$ 
\end{proof}

\begin{proof}[Proof of Lemma \ref{lem:triangle}]
By symmetry it is sufficient to prove the first statement. Let $\{x_j^S\}$, $\{x_j^s\}$ and $\{x_j^T\}$ denote the basis elements in $B_S$, $B_s$ and $B_T$. 
For each $j$ let $I_j(t, T)=\{i\mid \birth(\gamma_i(T))\leq \birth(\gamma_j(t))+|t-T| < \death(\gamma_i(T))\}$ for $t=S,s$. 

Suppose $\Mat_{B_S}^{B_s}(\alpha^{S\to s})=\pi_{(\Mod{V}_S, \Mod{V}_s, |s-S|)}(\Id)$. Fix $j$ such that $T\in \supp(\gamma_j)$ (and hence also $s\in \supp(\gamma_j)$ for all $s\in [S,T)$).

From the definition of the matrices, and using that the various maps in the vineyard module commute we see
\begin{align*}
\sum_{i\in I_j(S, T)} &\Mat_{B_S}^{B_T}(\alpha^{S\to T})(i,j)\phi_{\birth(\gamma_i(T))}^{\birth(\gamma_j(S))+|S-T|}(x_i^T)\\
&= \alpha_{\birth(\gamma_j(S))}^{S\to T}(x_j^S)\\
&=\alpha_{\birth(\gamma_j(S))+|S-s|}^{s\to T}\left(\alpha_{\birth(\gamma_j(S))}^{S\to s}(x_j^S)\right)\\
&=\alpha_{\birth(\gamma_j(S))+|S-s|}^{s\to T} \left(\phi_{\birth(\gamma_j(s))}^{\birth(\gamma_j(S))+|S-s|}(x_j^s)\right)\\
&=\phi_{\birth(\gamma_j(s))+|s-T|}^{\birth(\gamma_j(S))+|S-T|}\left(\alpha_{\birth(\gamma_j(s))}^{s\to T} (x_j^s)\right)\\
&=\phi_{\birth(\gamma_j(s))+|s-T|}^{\birth(\gamma_j(S))+|S-T|}\left(\sum_{i\in I_j(s,T)} \Mat_{B_s}^{B_T}(\alpha^{s\to T})(i,j) \phi_{\birth(\gamma_i(T))}^{\birth(\gamma_j(s))+|s-T|}(x_i^T)\right)\\
&=\sum_{i\in I_j(s,T)} \Mat_{B_s}^{B_T}(\alpha^{s\to T})(i,j)\phi_{\birth(\gamma_i(T))}^{\birth(\gamma_j(S))+|S-T|}(x_i^T).
\end{align*}

For $i\notin I_j(S,T)$ we have by definition $\Mat_{B_S}^{B_T}(\alpha^{S\to T})(i,j)=0= \Mat_{B_s}^{B_T}(\alpha^{s\to T})(i,j)$.

We have $\phi_{\birth(\gamma_i(T))}^{\birth(\gamma_j(S))+|S-T|}(x_i^T)\neq 0$  for all $i\in I_j(S,T)$ by construction. 
From our definition of basis of a persistence module we have $\{\phi_{\birth(\gamma_i(T))}^{\birth(\gamma_j(S))+|S-T|}(x_i^T)\mid i\in I_j(S,T)\}$ forms a basis of the vector space $(\Mod{V}_T)_{\birth(\gamma_j(S))+|S-T|}$. This linear independence, combined with the equality
$$\sum_{i\in I_j(S,T)} \Mat_{B_S}^{B_T}(\alpha^{S\to T})(i,j)\phi_{\birth(\gamma_i(T))}^{\birth(\gamma_j(S))+|S-T|}(x_i^T) =\sum_{i\in I_j(S,T)} \Mat_{B_s}^{B_T}(\alpha^{s\to T})(i,j)\phi_{\birth(\gamma_i(T))}^{\birth(\gamma_j(S))+|S-T|}(x_i^T),$$ implies that  $\Mat_{B_S}^{B_T}(\alpha^{S\to T})(i,j)= \Mat_{B_s}^{B_T}(\alpha^{s\to T})(i,j)$.

This proof can be easily modified to prove if $\Mat_{B_s}^{B_S}(\beta^{s\to S})=\pi_{(\Mod{V}_s, \Mod{V}_S, |s-S|)}(\Id)$
then $\Mat_{B_T}^{B_s}(\beta^{T\to s})= \Mat_{B_T}^{B_S}(\beta^{T\to S}).$ The relevant equations are now
\begin{align*}
\sum_{i\in I_j(S,T)} &\Mat_{B_T}^{B_S}(\beta^{T\to S})(i,j)\phi_{\birth(\gamma_i(S))}^{\birth(\gamma_j(T))+|S-T|}(x_i^S)\\
&= \beta_{\birth(\gamma_j(T))}^{T\to S}(x_j^T)\\
&=\beta_{\birth(\gamma_j(T))+|T-s|}^{s\to S}\left(\beta_{\birth(\gamma_j(T))}^{T\to s}(x_j^T)\right)\\
&=\beta_{\birth(\gamma_j(T))+|T-s|}^{s\to S} \left(\sum_{i\in I_j(s,T)} \Mat_{B_T}^{B_s}(\beta^{T\to s})(i,j) \phi_{\birth(\gamma_i(s))}^{\birth(\gamma_j(T))+|s-T|}(x_i^s)\right)\\
&=\sum_{i\in I_j(s,T)} \Mat_{B_T}^{B_s}(\beta^{T\to s})(i,j) \beta_{\birth(\gamma_j(T))+|T-s|}^{s\to S} ( \phi_{\birth(\gamma_i(s))}^{\birth(\gamma_j(T))+|s-T|}(x_i^s))\\
&=\sum_{i\in I_j(s,T)} \Mat_{B_T}^{B_s}(\beta^{T\to s})(i,j)  \phi_{\birth(\gamma_i(S))}^{\birth(\gamma_j(T))+|S-T|}(x_i^S).
\end{align*}
\end{proof}

\begin{proof}[Proof of Lemma \ref{lem:lambda_t}]
For all $s\in (S,T]$ there exists a partition of $\Vine{V}|_{[s, T]}$ into finitely many segments. As equality is transitive, it is sufficient to show for the case where $\Vine{V}|_{[s, T]}$ is a segment.

If $\Mat_{B_S}^{B_s}(\alpha^{S\to s})(i,k)\neq 0$ then 
\begin{equation}\label{eq:lambda(i,k)}
\birth(\gamma_k(s))\leq \birth(\gamma_i(S))+|S-s|<\death(\gamma_k(s))\leq \death(\gamma_i(S)) +|S-s|.
\end{equation}

If $\Mat_{B_s}^{B_T}(\alpha^{s\to T})(k,i)\neq 0$ then 
\begin{equation}\label{eq:lambda(k,i)}
\birth(\gamma_k(T))\leq \birth(\gamma_i(s))+|T-s|<\death(\gamma_k(T))\leq \death(\gamma_i(s)) +|T-s|.
\end{equation}

If $\Mat_{B_s}^{B_T}(\alpha^{s\to T})(k,i)\Mat_{B_S}^{B_s}(\alpha^{S\to s})(i,k) \neq 0$ then both \eqref{eq:lambda(k,i)} and \eqref{eq:lambda(i,k)}  hold. With our Lipschitz assumption on birth and death heights, this implies $|\birth(\gamma_i(T))-\birth(\gamma_k(T))|<2|S-T|$ and $|\death(\gamma_i(T))-\death(\gamma_k(T))|<2|S-T|$. As $\Vine{V}|_{[S,T]}$ is a segment this implies $i=k$.

With $\Mat_{B_s}^{B_T}(\alpha^{s\to T})(k,i)\Mat_{B_S}^{B_s}(\alpha^{S\to s})(i,k) = 0$ for $i\neq k$ we get
\begin{align*}
\Mat_{B_S}^{B_T}(\alpha^{S\to T})(k,k)& = \sum_i  \Mat_{B_s}^{B_T}(\alpha^{s\to T})(k,i)\Mat_{B_S}^{B_s}(\alpha^{S\to s})(i,k)\\
&= \Mat_{B_s}^{B_T}(\alpha^{s\to T})(k,k)\Mat_{B_S}^{B_s}(\alpha^{S\to s})(k,k).
\end{align*}

As $S$ is forwards incompatible with respect to vines $(\gamma_k, \gamma_l)$ we have $\birth(\gamma_k(S))\leq \birth(\gamma_l(S))$ and $\birth(\gamma_k(S))+4|S-T|\leq \death(\gamma_l(S)) $ and $\death(\gamma_l(S))\leq \death(\gamma_k(S))$. The Lipschitz assumption implies 
$\birth(\gamma_l(S))\leq \birth(\gamma_k(T))+|S-T|$ and $\birth(\gamma_k(T))+ 3|S-T|< \death(\gamma_l(S))$ and $\death(\gamma_l(S))\leq \death(\gamma_k(T)) +|S-T|$, so the $(k,l)$ entry is not projected to 0 in $\pi_{(\Mod{V}_S, \Mod{V}_T, |S-T|)}$.

Since $\Mat_{B_S}^{B_T}(\alpha^{S\to T})=\pi_{(\Mod{V}_S, \Mod{V}_T, |S-T|)}(\Mat_{B_s}^{B_T}(\alpha^{s\to T})\Mat_{B_S}^{B_s}(\alpha^{S\to s}))$ we get 

\begin{equation}\label{eq:lambdasum}
\Mat_{B_S}^{B_T}(\alpha^{S\to T})(k,l)=\sum_i  \Mat_{B_s}^{B_T}(\alpha^{s\to T})(k,i)\Mat_{B_S}^{B_s}(\alpha^{S\to s})(i,l).
\end{equation}

If $\Mat_{B_S}^{B_s}(\alpha^{S\to s})(i,l)\neq 0$ then 
\begin{equation}\label{eq:lambda(i,l)}
\birth(\gamma_i(s))\leq \birth(\gamma_l(S))+|S-s|<\death(\gamma_i(s))\leq \death(\gamma_l(S)) +|S-s|.
\end{equation}
If $\Mat_{B_s}^{B_T}(\alpha^{s\to T})(k,i)\Mat_{B_S}^{B_s}(\alpha^{S\to s})(i,l)\neq 0$ then both \eqref{eq:lambda(k,i)} and \eqref{eq:lambda(i,l)} hold 
and $$\birth(\gamma_i(S))\in [\birth(\gamma_k(S))-2|S-T|, \birth(\gamma_l(S))+2|S-T|]$$ and $$\death(\gamma_i(S))\in [\death(\gamma_k(S))-2|S-T|, \death(\gamma_l(S))+2|S-T|].$$ As $S$ is forwards incompatible with respect to vines $(\gamma_k, \gamma_l)$ either $\birth(\gamma_k(S))=\birth(\gamma_l(S))$ or $\death(\gamma_k(S))= \death(\gamma_l(S))$. If $\birth(\gamma_k(S))=\birth(\gamma_l(S))$ then $|\birth(\gamma_i(S))-\birth(\gamma_k(S))|<4|S-T|$. As $\Vine{V}|_{[S,T]}$ is a segment this implies $\birth(\gamma_i(S))=\birth(\gamma_k(S))$ and either $i=k$ or $i=l$. If $\death(\gamma_k(S))= \death(\gamma_l(S))$ the same argument with deaths replacing births works to again conclude $i=k$ or $i=l$.

Substituting into \eqref{eq:lambdasum} we get
$$\Mat_{B_S}^{B_T}(\alpha^{S\to T})(k,l)= \Mat_{B_s}^{B_T}(\alpha^{s\to T})(k,k) \Mat_{B_S}^{B_s}(\alpha^{S\to s})(k,l) + \Mat_{B_s}^{B_T}(\alpha^{s\to T})(k,l) \Mat_{B_S}^{B_s}(\alpha^{S\to s})(l,l)$$
As $\Vine{V}|_{[s, T]}$ is a segment of Type $1$ and either $\gamma_k(T)\subset \gamma_l(T)$ or $\gamma_l(T)\subset \gamma_k(T)$ we know $\Mat_{B_s}^{B_T}(\alpha^{s\to T})(k,l) =0$

Therefore $\Mat_{B_S}^{B_T}(\alpha^{S\to T})(k,l)= \Mat_{B_s}^{B_T}(\alpha^{s\to T})(k,k) \Mat_{B_S}^{B_s}(\alpha^{S\to s})(k,l)$ which combined with $\Mat_{B_S}^{B_T}(\alpha^{S\to T})(k,k)= \Mat_{B_s}^{B_T}(\alpha^{s\to T})(k,k)\Mat_{B_S}^{B_s}(\alpha^{S\to s})(k,k)$ gives
$$\frac{\Mat_{B_S}^{B_T}(\alpha^{S\to T})(k,l)}{\Mat_{B_S}^{B_T}(\alpha^{S\to T})(k,k)}=\frac{\Mat_{B_S}^{B_s}(\alpha^{S\to s})(k,l)}{\Mat_{B_S}^{B_s}(\alpha^{S\to s})(k,k)}.$$ 
\end{proof}

\section{Forward Simplifying Bases for Vineyard Modules}

\begin{proof}[Proof of Lemma \ref{lem:unique}]
$ $

\noindent
\underline{$\Vine{V}|_{[S,T]}$ of Type $1$:}
First observe that for segments of Type $1$ a family of bases is forwards simplified if and only if it is backwards simplified, so we can assume that $\{B_t\}$ and $\{\tilde{B}_t\}$ are each both forwards simplified and backwards simplified. From $\{B_t\}$ forwards simplified we know $ \Mat_{B_S}^{B_t}(\alpha^{S\to t})=\pi_{(\Mod{V}_S, \Mod{V}_t, |S-t|)}(\Id)$ and from $\{\tilde{B}_t\}$ backwards simplified we know $ \Mat_{\tilde{B}_t}^{\tilde{B}_S}(\beta^{t\to S})=\pi_{(\Mod{V}_t, \Mod{V}_S, |S-t|)}(\Id)$. 

Suppose $\tilde{B}_S=B_S$. Fix $t\in (S, T]$. Let $\psi_{2|S-t|}$ be the $2|S-t|$-morphism from $\Mod{V}_S$ to itself corresponding to the transition maps in $\Mod{V}_S$. 
\begin{align*}
\pi_{(\Mod{V}_S, \Mod{V}_S, 2|S-t|)}(\Id)&=\Mat_{B_S}^{\tilde{B}_S}(\psi_{2|S-t|})\\
&=\pi_{(\Mod{V}_S, \Mod{V}_S, 2|S-t|)}\left(\Mat_{\tilde{B}_t}^{\tilde{B}_S}(\beta^{t\to S})
\Mat_{B_t}^{\tilde{B}_t}(\id)
 \Mat_{B_S}^{B_t}(\alpha^{S\to t})\right)\\
&=\pi_{(\Mod{V}_S, \Mod{V}_S, 2|S-t|)}(\Mat_{B_t}^{\tilde{B}_t}(\id))\\
&=\pi_{(\Mod{V}_t, \Mod{V}_t, 0)}(\Mat_{B_t}^{\tilde{B}_t}(\id))\\
&=\Mat_{B_t}^{\tilde{B}_t}(\id)
\end{align*}
Here we use $\pi_{(\Mod{V}_S, \Mod{V}_S, 2|S-t|)}=\pi_{(\Mod{V}_S, \Mod{V}_S, 0)}$ by the separation of the critical heights by at least $4|S-T|>2|S-t|$. As $\Mat_{B_t}^{\tilde{B}_t}(\id)$ is the identity on the non-extended basis, we conclude $B_t=\tilde{B}_t$.

If $\tilde{B}_T=B_T$, then we use a symmetric calculation replacing $S$ with $T$.

\noindent
\underline{$\Vine{V}|_{[S,T]}$ of Type $2$}:

For each $t\in (S,T)$ there exists a partition $\{s_i\}$ of $\Vine{V}|_{[S,t]}$ into finitely many segments each of Type $1$. Inductively (with base case the assumption $B_S=\tilde{B}_S$) we can apply the Type $1$ case above to the segments $\Vine{V}|_{[s_i, s_{i+1}]}$ to say $B_{s_i}=\tilde{B}_{s_i}$ for all $i$ as $t=s_N$ for some $N$ we have $B_t=\tilde{B}_t$.
By assumption we also have $\tilde{B}_T=B_T$. Together we have $\tilde{B}_t=B_t$ for all $t\in [S, T]$.

\noindent
\underline{$\Vine{V}|_{[S,T]}$ of Type $3$}:

This case is symmetric to segments of Type $2$, swapping the roles of $S$ and $T$.
\end{proof}

\begin{proof}[Proof of Lemma \ref{lem:forward_simp_idempotent}]
By Lemma \ref{lem:unique} it is sufficient to check that for any segment $\Vine{V}|_{[S,T]}$ and a forwards simplified family of bases $\{B_t^{old}\}$, that if $\{B_t^{new}\}$ is the family of bases under forwards simplification procedure then  $B_t^{new}=B_t^{old}$ for all $t\in [S,T]$.

If $S$ is forwards compatible and $\{B_t^{old}\}$ forwards simplified. Then by Theorem \ref{thm:S_compatible} $\Mat_{B_S^{old}}^{B_T^{old}}(\alpha^{S\to T})=\pi_{(\Mod{V}_S, \Mod{V}_T, |T-S|)}(\Id)$ so $B_T^{new}=B_T^{old}$. 

Suppose $S$ is forwards incompatible, and $(\gamma_{k}, \gamma_{l})$ are the vines that cause the forwards incompatibility at $S$. There is some $\lambda_{old}$ such that $\Mat_{B_S^{old}}^{B_T^{old}}(\alpha^{S\to T})=\pi_{(\Mod{V}_S, \Mod{V}_T, |T-S|)}(e^{\lambda_{old}}_{kl})$. 
Then $$\lambda_{new}=\frac{\Mat_{B_S^{old}}^{B_T^{old}}(\alpha^{S\to T})(k,l)}{\Mat_{B_S^{old}}^{B_T^{old}}(\alpha^{S\to T})(k,k)}=\lambda_{old}$$
and $$B_T^{new}=\pi_{(\Mod{V}_T, \Mod{V}_T, 0)}(\Mat_{B_S^{old}}^{B_T^{old}}(\alpha^{S\to T})e_{kl}^{-\lambda_{new}})(B_T^{old})=\pi_{(\Mod{V}_T, \Mod{V}_T, 0)}(\Id)(B_T^{old})=B_T^{old}$$

\end{proof}

\section{Trivial Vineyard Modules}

\begin{proof}[Proof of Lemma \ref{lem:same up to shift}]
By symmetry, it is sufficient to prove the first statement.

As $S$ is forwards compatible, and both $\{B_t^V\}$ and $\{B_t^W\}$ are forwards simplified families of bases, we know
\[
\mathrm{Mat}^{B^V_t}_{B^V_S}(\alpha^{S \to t}_V) = \pi_{(\Mod{V}_S, \Mod{V}_t, |t-S|)}(\Id) \quad \text{and} \quad \mathrm{Mat}^{B^W_t}_{B^W_S}(\alpha^{S \to t}_W) = \pi_{(\Mod{W}_S, \Mod{W}_t, |t-S|)}(\Id).
\]

Using the commuting square, $\alpha^{S \to t}_W \circ \rho_S = \rho_t \circ \alpha^{S \to t}_V$ we have $$\Mat_{B_S^V}^{B_t^W}(\alpha^{S \to t}_W \circ \rho_S )=\Mat_{B_S^V}^{B_t^W}( \rho_t \circ \alpha^{S \to t}_V).$$ We have 
\begin{align*}
\Mat_{B_S^V}^{B_t^W}(\alpha^{S \to t}_W \circ \rho_S )&=\pi_{(\Mod{V}_S, \Mod{W}_t, |S-t|)}\left(\Mat_{B_S^W}^{B_t^W}(\alpha^{S \to t}_W)\Mat_{B_S^V}^{B_S^W}(\rho_S)\right)\\
&=  \pi_{(\Mod{V}_S, \Mod{W}_t, |S-t|)}\left(\pi_{(\Mod{W}_S, \Mod{W}_t, |S-t|)}(\Id)M_S\right)\\
&= \pi_{(\Mod{V}_S, \Mod{W}_t, |S-t|)}(M_S)
\end{align*}
and similarly $\Mat_{B_S^V}^{B_t^W}(\rho_t \circ \alpha^{S \to t}_V)= \pi_{(\Mod{V}_S, \Mod{W}_t, |S-t|)}(M_t)$. Combining gives $$\pi_{(\Mod{V}_S, \Mod{W}_t, |S-t|)}(M_t)= \pi_{(\Mod{V}_S, \Mod{W}_t, |S-t|)}(M_S).$$

If $[S,T]$ is a segment of Type $1$ then $\pi_{(\Mod{V}_S, \Mod{W}_t, |S-t|)}=\pi_{(\Mod{V}_t, \Mod{W}_t, 0)}=\pi_{(\Mod{V}_S, \Mod{W}_S, 0)}$ so $$M_t= \pi_{(\Mod{V}_S, \Mod{W}_t, |S-t|)}(M_t)=\pi_{(\Mod{V}_S, \Mod{W}_t, |S-t|)}(M_S)=\pi_{(\Mod{V}_S, \Mod{W}_S,0)}(M_S)=M_S.$$

\end{proof}

\begin{proof}[Proof of Lemma \ref{lem:COB forwards incompatible}]
By symmetry, it is sufficient to prove the first statement.

Fix $t\in (S, T)$ and consider a partition of $[t, T]$ into segments which by construction will all be of Type 1 as no critical times appear in $(S,T]$. We can apply Lemma \ref{lem:same up to shift} inductively along the segments in this new partition to conclude that $M_t=M_T$.  As no critical times occur in $(S, T]$ we can infer that $\pi_{(\Mod{V}_{T}, \Mod{W}_{T}, 0)}=\pi_{(\Mod{V}_{t}, \Mod{W}_{t}, 0)}$ for all $t\in (S, T]$. It thus only remains to prove that $M_{T}= \pi_{(\Mod{V}_{T}, \Mod{W}_{T}, 0)}(M_S e^{-\lambda}_{kl})$

We first want to establish that $\pi_{(\Mod{V}_T, \Mod{W}_T, 0)}(M_T)$ and $\pi_{(\Mod{V}_T, \Mod{W}_T, 2|S-T|)}(M_T)$ agree.
As $S$ is forwards incompatible with respect to vines $(\gamma_k, \gamma_l)$ we have exactly one of $\birth(\gamma_k(S))=\birth(\gamma_l(S))$ or $\death(\gamma_k(S))=\death(\gamma_l(S))$ and all other pairs of critical heights in $S$ are strictly at least $4|S-T|$ apart. 

This implies that exactly one of $|\birth(\gamma_k(S))-\birth(\gamma_l(S))| < 2|S-T|$ or $|\death(\gamma_k(S))-\death(\gamma_l(S))|<2|S-T|$ and all pairs of critical heights are strictly at least $2|S-T|$ apart. This implies that $\pi_{(\Mod{V}_T, \Mod{W}_T, 2|S-T|)}=\pi_{(\Mod{V}_T, \Mod{W}_T, 0)}$.

Let $\phi$ be the $2|S-T|$-morphism from  $\Mod{W}_T$ to itself via its transition maps - shifting the height up by $2|S-T|$. Consider the composition $\alpha_W^{S\to T} \circ \rho_S \circ \beta_V^{T\to S}$ which is a $2|S-T|$-morphism from $\Mod{V}_T$ to $\Mod{W}_T$. By commutativity this will agree with the composition of $\rho_T$ with $\phi$. So
\begin{align*}
\pi_{(\Mod{V}_T, \Mod{W}_T, 0)}(M_T)&=
\pi_{(\Mod{V}_T, \Mod{W}_T, 2|S-T|)}(M_T)\\
&=\Mat_{B_T^V}^{B_T^W}(\phi \circ  \rho_T)\\
&=\Mat_{B_T^V}^{B_T^W}(\alpha_W^{S\to T} \circ \rho_S \circ \beta_V^{T\to S})\\
&=\pi_{(\Mod{V}_T, \Mod{W}_T, 2|S-T|)}\left(\Mat_{B_S^W}^{B_T^W}(\alpha_W^{S\to T})\Mat_{B_S^V}^{B_S^W}( \rho_S)\Mat_{B_T^V}^{B_S^V} ( \beta_V^{T\to S})\right)\\
&=\pi_{(\Mod{V}_T, \Mod{W}_T, 2|S-T|)}(M_S e^{-\lambda_V}_{kl})\\
&=\pi_{(\Mod{V}_T, \Mod{W}_T, 0)}(M_S e^{-\lambda_V}_{kl})\\
\end{align*}
for some $\lambda_V\in \mathbb{F}$. Here we are using that the basis is forwards simplified and Lemma \ref{lem:lambdaIFF}.

All it remains to show is that $\lambda_V=\frac{M_S(k,l)}{M_S(k,k)}$.
As $(\gamma_{k}, \gamma_{l})$ are the vines that cause the forwards incompatibility at S we know that $\death(\gamma_k(S))> \birth(\gamma_l(S))+4|S-T|$, so $\death(\gamma_k(T))> \birth(\gamma_l(T))+2|S-T|$, and that $M_T(k,l)=0$. This implies that $0=M_Se^{-\lambda_V}_{kl}(k,l)=M_S(k,l)-\lambda_V M_S(k,k)$.

\end{proof}

\begin{proof}[Proof of Lemma \ref{lem:transitivity}]
From $\gamma_i$ is observably below $\gamma_j$ before $t$ we have $\birth(\gamma_i(t))\leq \birth(\gamma_j(t))<\death(\gamma_i(t))\leq\death(\gamma_j(t))$. Similarly $\birth(\gamma_j(t))\leq \birth(\gamma_k(t))<\death(\gamma_j(t))\leq\death(\gamma_k(t))$. As $\tau_t(i,k)\neq t$, $\birth(\gamma_k(t))<\death(\gamma_i(t))$ (as otherwise the intervals are disjoint).
Together they provide an order
$$\birth(\gamma_i(t))\leq \birth(\gamma_j(t))\leq \birth(\gamma_k(t))<\death(\gamma_i(t))\leq \death(\gamma_j(t))\leq \death(\gamma_k(t)).$$
Consider the most recent time before $t$, denoted $\tilde{t}$ where either the order of these six birth/death values changes or a birth and death value become equal (as this can make the intervals disjoint). 
As $\gamma_i$ is observably below $\gamma_j$ before $t$ the swap cannot be between $\birth(\gamma_i(s))$ and $\birth(\gamma_j(s))$ or between $\death(\gamma_i(s))$ and $\death(\gamma_j(s))$. 
Similarly the swap cannot be between $\birth(\gamma_j(s))$ and $\birth(\gamma_k(s))$ or between $\death(\gamma_k(s))$ and $\death(\gamma_j(s))$. 
This implies that at time $\tilde{t}$ we have $\birth(\gamma_k(\tilde{t}))=\death(\gamma_i(\tilde{t}))$, forcing $\tau_{t}(i,k)=\tilde{t}$ and that $\birth(\gamma_i(s))\leq\birth(\gamma_k(s))<\death(\gamma_i(s))\leq \death(\gamma_k(s))$ for all $s\in (\tilde{t}, t]$.
\end{proof}

\section{Vineyard Modules with Two Vines}
\begin{proof}[Proof of Lemma \ref{lem:split}]

Let $\{S_k\}$ be a partition of $\Vine{V}$ into segments. As $\tilde{t}$ is left critical it must be in $\{S_k\}$. Let $\Vine{V}|_{[S_i,\tilde{t}]}$ be one of the segments in the partition with $\tilde{t}$ as its right endpoint. Note that $S_i$ is not right critical as $\tilde{t}$ is left critical. As $\tilde{t}$ is backwards incompatible we must have either $\birth(\gamma_1(S_i))<\birth(\gamma_2(S_i))<\death(\gamma_2(S_i))<\death(\gamma_1(S_i))$ or 
$\birth(\gamma_2(S_i))<\birth(\gamma_1(S_i))<\death(\gamma_1(S_i))<\death(\gamma_2(S_i))$. In either case there is a unique choice of basis. As $\Vine{V}|_{[S_i,\tilde{t}]}$ is a segment with $\tilde{t}$ critical there are no critical times in $[S_i, \tilde{t})$.

The same argument works for $\tilde{s}$ using the segment $\Vine{V}|_{[\tilde{s}, S_j]}$.
\end{proof}

\begin{proof}[Proof of Lemma \ref{lem:unique_about_backwards_incompatible}]

Suppose that we have two backwards simplified bases $\{B_t\}$ and $\{\hat{B}_t\}$. Observe that for $t\in [T_0, \tilde{t})\cup (\tilde{s}, T_1]$ there is a unique choice of basis as we have an inclusion of intervals. 

Choose $T$ such that $\Vine{V}|_{[\tilde{s}, T]}$ is a segment. We will be using the basis of $\Mod{V}_T$ is unique so $B_T=\hat{B}_T$. Since $\tilde{s}$ is critical we have $\pi_{(\Mod{V}_T, \Mod{V}_{\tilde{s}}, |T-\tilde{s}|)}=\pi_{(\Mod{V}_T, \Mod{V}_{\tilde{s}}, 0)}$. From our backwards simplified assumption we also know $\Mat_{B _T}^{B_{\tilde{s}}}(\beta^{T\to \tilde{s}})=\Id=\Mat_{B _T}^{\hat{B}_{\tilde{s}}}(\beta^{T\to \tilde{s}})$.

\begin{align*}
\Id&=\pi_{(\Mod{V}_T, \Mod{V}_{\tilde{s}}, |T-\tilde{s}|)}(\Mat_{B _T}^{B_{\tilde{s}}}(\beta^{T\to \tilde{s}}))\\
&=\pi_{(\Mod{V}_T, \Mod{V}_{\tilde{s}}, |T-\tilde{s}|)}\left(\Mat_{\hat{B}_{\tilde{s}}}^{B_{\tilde{s}}}(\id)\Mat_{B _T}^{\hat{B}_{\tilde{s}}}(\beta^{T\to \tilde{s}})\right)\\
&=\pi_{(\Mod{V}_T, \Mod{V}_{\tilde{s}}, |T-\tilde{s}|)}(\Mat_{\hat{B}_{\tilde{s}}}^{B_{\tilde{s}}}(\id))\\
&=\Mat_{\hat{B}_{\tilde{s}}}^{B_{\tilde{s}}}(\id)
\end{align*}
implying $B_{\tilde{s}}=\hat{B}_{\tilde{s}}$.

If $\tilde{t}=\tilde{s}$ (the cases (i) and (j) in Figure \ref{fig:two-vine-cases}) then we are done. Suppose that $\tilde{t}<\tilde{s}$.

Let $\{S_i\}_{i=0}^N$ be a partition of $\Vine{V}|_{[\tilde{t}, \tilde{s}]}$ into segments (so $S_N=\tilde{s}$ and $S_0=\tilde{t}$). If $B_{S_i}=\hat{B}_{S_i}$ for all $i$ then we can apply Lemma \ref{lem:unique} to conclude that $B_t=\hat{B}_t$ for all $t\in [\tilde{t}, \tilde{s}]$. 

We have $\Vine{V}|_{[S_{N-1}, \tilde{s}]}$ is a segment. As $\tilde{s}$ is backwards compatible (since $\tilde{s} \neq \tilde{t}$ which is the only backwards incompatible time) 
$$\Mat_{B _{\tilde{s}}}^{B_{S_{N-1}}}(\beta^{\tilde{s}\to S_{N-1}})=\Id=\Mat_{\hat{B}_{\tilde{s}}}^{\hat{B}_{S_{N-1}}}(\beta^{\tilde{s}\to S_{N-1}}).$$ 
By Lemma \ref{lem:IdIFF} we also have $$\Mat_{B_{S_{N-1}}}^{B_{\tilde{s}}}(\alpha^{S_{N-1}\to \tilde{s}})=\Id=\Mat_{\hat{B}_{S_{N-1}}}^{\hat{B}_{\tilde{s}}}(\alpha^{S_{N-1}\to \tilde{s}}).$$

Analogous to before (now using $\hat{B}_{\tilde{s}}=B_{\tilde{s}}$)
\begin{align*}
\Id&=\pi_{(\Mod{V}_{S_{N-1}}, \Mod{V}_{\tilde{s}}, |S_{N-1}-\tilde{s}|)}(\Mat_{\hat{B}_{S_{N-1}}}^{B_{\tilde{s}}}(\alpha^{S_{N-1}\to \tilde{s}}))\\
&=\pi_{(\Mod{V}_{S_{N-1}}, \Mod{V}_{\tilde{s}}, |S_{N-1}-\tilde{s}|)}\left(\Mat_{B_{S_{N-1}}}^{B_{\tilde{s}}}(\alpha^{S_{N-1}\to \tilde{s}})\Mat_{\hat{B}_{S_{N-1}}}^{B_{S_{N-1}}}(\id)\right)\\
&=\pi_{(\Mod{V}_{S_{N-1}}, \Mod{V}_{\tilde{s}}, |S_{N-1}-\tilde{s}|)}(\Mat_{\hat{B}_{S_{N-1}}}^{B_{S_{N-1}}}(\id))\\
&=\Mat_{\hat{B}_{S_{N-1}}}^{B_{S_{N-1}}}(\id)
\end{align*}
implying $B_{S_{N-1}}=\hat{B}_{S_{N-1}}$.

We can then inductively show $B_{S_{N-j}}=\hat{B}_{S_{N-j}}$ from $j=1$ to $j=N-1$ using Lemma \ref{lem:unique} as the remaining segments are all of Type $1$.

The proof that $B_{\tilde{t}}=\hat{B}_{\tilde{t}}$ from $B_{S_1}=\hat{B}_{S_1}$ is the same calculations as showing  $B_{\tilde{s}}=\hat{B}_{\tilde{s}}$ from $B_{T}=\hat{B}_{T}$.
\end{proof}
\end{document}